\documentclass{amsart}

\usepackage[numbers,sort&compress]{natbib}

\usepackage{amsmath}
\usepackage{xcolor}
\usepackage{mathrsfs}
\usepackage{enumerate}
\usepackage[hidelinks]{hyperref}
\usepackage{cleveref}

\theoremstyle{plain}
\newtheorem{theorem}{Theorem}[section]
\newtheorem{lemma}[theorem]{Lemma}
\newtheorem{corollary}[theorem]{Corollary}
\newtheorem{proposition}[theorem]{Proposition}
\theoremstyle{definition}
\newtheorem{definition}[theorem]{Definition}
\newtheorem{problem}[theorem]{Problem}
\newtheorem{remark}[theorem]{Remark}
\newtheorem{hypothesis}[theorem]{Hypothesis}

\AtBeginDocument{
  \label{CorrectFirstPageLabel}
  
}
\usepackage{amsfonts}
\usepackage{algorithm}
\usepackage{algorithmic}
\usepackage{graphicx}
\usepackage{epstopdf}
\ifpdf
  \DeclareGraphicsExtensions{.eps,.pdf,.png,.jpg}
\else
  \DeclareGraphicsExtensions{.eps}
\fi

\newcommand{\R}{\mathbb{R}}
\newcommand{\Domain}{\R^\Dim}
\newcommand{\N}{\mathbb{N}}
\newcommand{\Rp}{\R_{\geq 0}}
\newcommand{\Rpn}{\Rp^{\Ntdens}}
\newcommand{\M}{\mathcal{M}}
\newcommand{\W}{\mathcal {W}}
\newcommand{\PMeaN}{\mathcal{M}_{\IndexN}}
\newcommand{\PMeaI}{\mathcal{M}_{\infty}}
\newcommand{\bs}[1]{\boldsymbol{#1}}

\newcommand{\comp}[2]{{#1}_{\{#2\}}}
\newcommand{\ei}{\comp{\bs{e}}{i}}

\newcommand{\Ntdens}{N}
\newcommand{\Npot}{M}
\newcommand{\IndexN}{n}
\newcommand{\IndexM}{m}

\newcommand{\Of}[1]{[#1]}
\newcommand{\Leb}{dx}
\newcommand{\Ene}{\mathcal{E}}
\newcommand{\EneN}{\mathcal{E}_{\IndexN}}
\newcommand{\EneNM}{\mathcal{E}_{\IndexN,\IndexM}}
\newcommand{\DEneN}{E_{\IndexN}}
\newcommand{\DEneNM}{E_{\IndexN,\IndexM}}
\newcommand{\FEneN}{F_{\IndexN}}
\newcommand{\FEneNM}{F_{\IndexN,\IndexM}}
\newcommand{\GEneN}{G_{\IndexN}}

\newcommand{\Triang}{\mathcal{T}}
\newcommand{\Dim}{d}

\newcommand{\bmu}{\bs{\mu}}
\newcommand{\bmuh}{\bmu_\IndexN}
\newcommand{\bmui}{\comp{\bmu}{i}}
\newcommand{\bmuj}{\comp{\bmu}{j}}
\newcommand{\bmuhi}{\comp{\bmuh}{i}}

\newcommand{\muN}{\mu_{\IndexN}}
\newcommand{\Ipot}{p}
\newcommand{\Jpot}{q}
\newcommand{\bu}{\bs{u}}

\newcommand{\bui}{\comp{\bu}{\Ipot}}

\newcommand{\bsigma}{\bs{\sigma}}
\newcommand{\bsigmah}{\bsigma_{\IndexN}}

\newcommand{\bsigmai}{\comp{\bsigma}{i}}
\newcommand{\bsigmaj}{\comp{\bsigma}{j}}
\newcommand{\bsigmahi}{\comp{\bsigmah}{i}}

\newcommand{\hatbsigmah}{\hat{\bsigma}_{\IndexN}}
\newcommand{\hatbsigmahi}{\hat{\bs{\sigma}}_{\IndexN,\{i\}}}
\newcommand{\bm}{\bs{a}}
\newcommand{\bmh}{\bm_{\IndexN}}

\newcommand{\bsf}{\bs{f}}
\newcommand{\bfk}{\bsf_{\IndexM}}
\newcommand{\bfki}{\comp{\bfk}{\Ipot}}
\newcommand{\iter}{\ell}
\newcommand{\bss}{\bs{s}}
\newcommand{\bsse}{\bss_\epsilon}
\newcommand{\Anm}{A^{(\IndexN,\IndexM)}}
\newcommand{\Anmi}{A^{(\IndexN,\IndexM,i)}}
\newcommand{\Anmj}{A^{(\IndexN,\IndexM,j)}}
\newcommand{\Am}{A^{(\infty,\IndexM)}}

\newcommand{\Loja}{\L{}ojasiewicz}

\newcommand{\tolkkt}{\mbox{toll}}
\newcommand{\eigenval}{\lambda}
\newcommand{\Lift}{\delta}
\newcommand{\Lifth}{\delta_{\IndexN}}
\newcommand{\MeshPar}{h}

\newcommand{\gammaconverge}{\xrightarrow[]{\Gamma}}
\newcommand{\Sq}[1]{{#1}^{\boldsymbol{2}}}

\DeclareMathOperator{\dimension}{dim}
\DeclareMathOperator{\diam}{diam}
\DeclareMathOperator{\sign}{sign}
\DeclareMathOperator{\Span}{span}

\DeclareMathOperator{\co}{conv}
\DeclareMathOperator{\support}{supp}
\DeclareMathOperator{\divergence}{div}
\DeclareMathOperator*{\scm}{\textsc{sc}^{\,-}}
\DeclareMathOperator*{\gammaliminf}{\Gamma-\liminf}
\DeclareMathOperator*{\gammalimsup}{\Gamma-\limsup}
\DeclareMathOperator*{\gammalim}{\Gamma-\lim}
\DeclareMathOperator*{\argmin}{argmin}
\DeclareMathOperator{\Hess}{Hess}
\DeclareMathOperator{\Ker}{Ker}

\makeatletter
\newcommand{\leqnomode}{\tagsleft@true}
\newcommand{\reqnomode}{\tagsleft@false}
\makeatother

\numberwithin{equation}{section}

\begin{document}
\title{Computing the $L^1$ optimal transport density: a FEM approach}
%
\author{Federico Piazzon}
\address{Department of Mathematics ``Tullio Levi-Civita'', University of Padua, via Trieste 63, Padova, Italy}
\email{fasso@math.unipd.it}

\author{Enrico Facca}
\address{Univ. Lille, Inria, CNRS, UMR 8524 - Laboratoire Paul Painlev\'e, F-59000 Lille, France} 
\email{enrico.facca@inria.fr} 

\author{Mario Putti}
\address{Department of Mathematics ``Tullio Levi-Civita'', University of Padua, via Trieste 63, Padova, Italy}
  \email{putti@math.unipd.it}

\date{\today}
\begin{abstract}
  The $L^1$ optimal transport density $\mu^*$ is the unique $L^\infty$
solution of the Monge-Kantorovich equations. It has been recently
characterized also as the unique minimizer of the $L^1$-transport
energy functional $\mathcal E$. In the present work we develop and we
prove convergence of a numerical approximation scheme for $\mu^*.$ Our
approach relies upon the combination of a FEM-inspired variational
approximation of $\Ene$ with a minimization algorithm based on a
gradient flow method.

\end{abstract}
%
%
\maketitle
\section{Introduction}
\label{sec:intro}
\subsection{$L^1$ optimal transport}
Optimal transport is the mathematical problem modeling the quest for an optimal strategy for moving a mass displacement of a material to a target distribution. Apart from the original motivation of soil dragging \cite{Mo91}, optimal transport finds a large variety of applications, e.g, metrics in probability theory, smoothing in image processing, baricentric interpolation and many others (see for instance \cite{Sa14,Sa15} and references therein). The modern formulation of optimal transport is due to Kantorovich, see \cite{Ka42}, and can be stated as follows. Let $\nu^+,\nu^-$ be two Borel measures defined on $\Domain$ with finite equal masses. Let $c:\Domain\times\Domain\rightarrow \R\cup\{+\infty\}$ be a Borel function. 
Find a non-negative Borel Measure $\gamma$ on $\Domain\times \Domain$ (termed \emph{optimal transport plan}) realizing the following infimum
\begin{equation}
\label{kantorovichformulation}\inf\left\{\int_{\Domain\times \Domain} c(x,y)d\gamma(x,y) \right\},
\end{equation}
under the constraints 
\begin{align*} 
&\gamma (A,\Domain) = \nu^{+}(A)\  \forall A \mbox{ Borel set in } \Domain,\\ 
&\gamma (\Domain,B) = \nu^{-}(B)\  \forall B \mbox{ Borel set in } \Domain.
\end{align*}
Existence of such an optimal transport plan is guaranteed by mild assumptions on $c$, e.g., lower semicontinuity and boundedness from below.

In this paper we focus on the $L^1$ case, i.e., we assume $c(x,y):=\|x-y\|$. Also, we will work under the (classical) further assumption of $\nu^+=f^+dx$, $\nu^-=f^-dx$ for a given zero-mean function $f\in L^\infty(\Domain).$ In this setting not only the problem \eqref{kantorovichformulation} admits a unique optimal transport plan, also such a plan is defined by an \emph{optimal transport map}, i.e., a Borel function $T:\Domain\rightarrow\Domain$ realizing the minimum in the Monge formulation of optimal transport, i.e.
$$\inf\left\{\int_{\Domain} c(x,T(x))d\nu^+, T_\#\nu^+ = \nu^-\right\},$$
where $T_\#\nu^+$ is the push-forward measure. 

  The aim of the present work is to study an approximation algorithm for the solution of the $L^1$-optimal transport.
While the numerical solution of $L^2$ optimal transport has been tackled in several previous works~(see e.g. \cite{BeGaVi19, Be21, CaCrWaWe22} for recent contributions), schemes for the $L^1$ case are sparser.
Typical methods for the numerical solution of this latter problem are based on the entropic regularization of the linear programming problem associated to the Kantorovich formulation and augmented Lagrangian methods, or on the reformulation of the problem in Beckmann form~\cite{Cuturi:2013, Benamou-et-al:2015, Li-et-al:2018, Peyre-Cuturi:2019}.

In our work, we take inspiration from the work of~\cite{FaCaPu18,FaDaCaPu20}, where a numerical solver for the PDE based formulation of this problem was proposed. The PDE formulation of the $L^1$-optimal transport originated in \cite{EvGa99} who showed that the $L^1$-optimal transport map $T$ can be constructed starting from the \emph{optimal transport density}, i.e., the unique $L^\infty$  solution $\mu^*$ of the so-called \emph{Monge-Kantorovich equations}
 \begin{equation}\label{eq:EvGaPDE}
  \begin{cases} -\divergence{(\mu^*\nabla u^*)}=f^+-f^-,& \text{ in
    }\Omega\\ |\nabla u^*|\leq 1,& \text{ in }\Omega\\ |\nabla u^*|=1\,&
    \mu^*\text{ a.e. in }\Omega
  \end{cases},
\end{equation}
where $\Omega$ is a bounded convex domain compactly containing $\support f$.
The approximation of the solution of \eqref{eq:EvGaPDE} is a very difficult task. Indeed, the coupling of the \emph{linear} PDE with the constraints on the gradient of $u$ give rise to a highly \emph{nonlinear} problem with inequality constraints in the variable $\mu$. Also, due to the nonlinearity of \eqref{eq:EvGaPDE}, it is not \emph{a priori} clear  weather (and in which sense) the solution $\muN$ of certain discretized version of \eqref{eq:EvGaPDE} would converge to the solution of \eqref{eq:EvGaPDE}.

The fundamental advance introduced in~\cite{FaCaPu18,FaDaCaPu20} is to address the numerical solution of the variational reformulation of the Monge-Kantorovich equation presented by~\cite{BoBuSe97} and solve the related optimization problem by a gradient method.
In these works, the authors introduce a Lyapunov-candidate functional and conjecture the existence of a continuous time gradient flow whose asymptotic solution coincides with the optimal transport density. Relying on this conjecture, the authors propose an efficient numerical scheme for the integration of the flow based on the explicit Euler discretization of the gradient flow and the finite element approximation of the relevant integrals. We would like to note that this approach is based on the formulation of a minimization problem and its subsequent discretization. The resolution of the conjecture has been hampered by the lack of results related to existence and uniqueness of the gradient flow and its trajectories, an outcome of the setup of the minimization problem in non-separable infinite dimensional spaces. As a consequence, no theoretical convergence of the above numerical scheme have been obtained and only experimental results are available to justify this approach.

The aim of the present work is to fill this gap and develop the complete theoretical convergence analysis of the FEM-based numerical method for the approximation of the $L^1$ optimal transport density. The properties of the studied discretization naturally suggest that, in contrast with the explicit nature of the method proposed in~\cite{FaCaPu18,FaDaCaPu20}, implicit Euler and Newton schemes must combine effectively to provide a super-linearly convergent method for the solution of the $L^1$-optimal transport problem.

The strategy that allows the completion of our goal can be summarized as follows. Rather than tackling the proof of the conjecture, we reformulate the problem by first discretizing the relevant functional and then defining a gradient flow in a finite-dimensional setting. This discretize-minimize approach, opposite to the minimize-discretize suggested in~\cite{FaDaCaPu20}, defines a sequence of finite-dimensional approximations of the above mentioned Lyapunov-candidate functional that are designed to provide a sequence of minimizers that converge towards the optimal transport density. The minimization step of the proposed strategy exploits the real analyticity of these approximations to derive a well-posed and smooth gradient flow that converges exponentially towards the sought solution under mild hypotheses.  Within this proposed framework, we are able to prove the convergence of the proposed algorithm towards the optimal transport density, to provide a sharp convergence rate estimate, and to design robust and efficient stopping criteria. The final scheme includes an adaptive time-stepping strategy that allows the use of a geometrically increasing sequence of time-step sizes leading to super-exponential convergence towards the minimizer, thus drastically improving the computational performance with respect to previous work.

\subsection{Our study}

Our starting point is the variational characterization of the transport density proposed in~\cite{FaPiPu19} and the following definition of a $L^1$-\emph{transport energy}:
\begin{definition}[$L^1$-Transport energy~\cite{FaPiPu19}]\label{def:Edefinition}
  Let $\Omega$ be a bounded Lipschitz domain of $\R^n$ and let
  $f\in L^\infty(\Omega)$ be such that
  $\co(\support f)\subset\subset \Omega$ and $\int_\Omega f dx=0$. We
  denote by $\Ene: \M^+(\Omega)\rightarrow [0,+\infty]$
  the transport energy functional defined by
  \begin{equation}\label{eq:Edefinitioneq}
    \Ene(\mu):=
    \sup_{u\in \mathscr C^1(\overline\Omega),\int_\Omega u dx=0}
    \left(2\int_\Omega fu\, dx-\int_\Omega |\nabla u|^2d\mu\right)
    +\int_\Omega d\mu.
  \end{equation}
\end{definition}
and we consider the following variational problem:
\begin{problem}[Minimization of the transport energy]
  \label{prob:problemEnricoMario}
  Given $f,\Omega$ as above, find $\mu_{\Ene}$ such that
  $$\Ene(\mu_{\Ene})=\inf_{\nu\in \M^+(\Omega)}\Ene(\nu).$$
\end{problem}
Here and throughout the paper we denote by $\M(\Omega)$ the
space of Borel signed measures on $\Omega),$ by $\M^+(\Omega)$
the non-negative Borel measures, and by $\M^1(\Omega)$ the
space of Borel probability measures.
In addition, our work relies on the following fundamental property, see \cite[Prop. 2.1]{FaPiPu19}.
\begin{proposition}
Problem \ref{prob:problemEnricoMario} has an unique solution $\mu_{\Ene}$ which is an absolutely continuous measure with
$$d\mu_{\Ene}=\mu^*dx.$$
\end{proposition}

We start by introducing a double indexed sequence of functionals $\EneNM$ corresponding to the restriction of $\Ene$ to certain finite element spaces $\PMeaN^+$ and $\W_\IndexM$ on which we approximate $\mu$ and $u$, respectively (Section \ref{sec:gamma}). We consider the diagonal sequence $\EneN:=\Ene_{\IndexN,\IndexN}$ and we show that its $\Gamma$-limit is $\Ene_\infty$, the lower semicontinuous regularization of the restriction of $\Ene$ to a dense subspace. Nevertheless, we prove in Theorem \ref{thm:mainresult1} that 
$$\argmin_{\M^+(\Omega)} \Ene_\infty\subseteq \argmin_{\M^+(\Omega)} \Ene,$$
moreover \emph{any sequence $\{\muN^*\}$ of minimizers of $\EneN$ converges to the unique minimizer of $\Ene$} in the weak$^*$ topology of measures, i.e.,
$$\muN^*\rightharpoonup \mu^*.$$
This result motivates the study of the numerical solution of the problems
$$\muN^*\in \argmin_{ \PMeaN^+} \EneN.$$

Next (Section \ref{sec:conditioning}), we study conditions that ensure well-posedness and well-{con\-di\-tio\-ning} of these discrete minimum problems. In particular we show in Proposition \ref{prop:wellpos} that under the Hypothesis \ref{hyp:H1} there exists a unique minimizer of $\EneN$ for any $\IndexN\in \N$, i.e., the discretized minimum problems are well-posed. Under a slightly different assumption, namely Hypothesis \ref{hyp:H2}, we prove that such problems are also well-conditioned, see Proposition \ref{prop:wellcond}.

The following step (Section \ref{sec:continuous_time}) is to prove global existence, uniqueness, and regularity of the following two gradient flows
\begin{equation*}
\begin{cases}
\frac d{dt}\bs \mu=-\partial^\circ \DEneN(\bs \mu),& \forall t>0\\
\bs \mu(0)=\bs \mu_0\in \Rpn& 
\end{cases}\;\;,\;\;\;
\begin{cases}
\frac d{dt}\bs\sigma=-\nabla \FEneN(\bs\sigma),& \forall t>0\\
\bs\sigma(0)=\bs\sigma^0
\end{cases}\;.
\end{equation*}
Namely the gradient flow equation \eqref{eqn:h-gradientflow} for the function $\DEneN$ which is obtained from $\EneN:=\Ene_{\IndexN,\IndexN}$ with the use of canonical coordinates on $\PMeaN^+\cong \Rpn$, and the gradient flow equation \eqref{eq:GhGradientflow} for the function $\FEneN$ which is obtained from $\DEneN$ by composition with the coordinate square map (i.e., $\R^{\Ntdens}\ni\bsigma\mapsto (\sigma_1^2,\sigma_2^2,\dots\sigma_\Ntdens^2)\in \Rpn$). In particular, in Theorem \ref{thm:existenceuniqueness},  we prove that any trajectory of the first flow is converging to a global minimizer of $\EneN$, while we prove in  Theorem \ref{thm:GhLongTime} that any trajectory of the latter is converging to a local minimizer of $\FEneN$: applying the coordinate square map we recover a global minimizer of $\EneN.$

The definition and study of the optimization algorithms for $\DEneN$ and $\FEneN$ derived from the time discretization of the two flows above conclude our theoretical contributions (Section \ref{sec:discrete_time}). In the first case one needs to take into account the non-negativity constraint on $\bmu$, so an explicit time discretization seems to be necessary. Thus the gradient flow \eqref{eqn:h-gradientflow} of $\EneN$ can be integrated by the so-called \emph{projected forward Euler scheme}, which can be solved by Algorithm \ref{alg:projEuler}. On the contrary, the gradient flow of $\FEneN$ is defined on $\R^{\Ntdens}$, and this makes feasible the use of an implicit scheme: our choice is indeed the \emph{backward Euler scheme} coupled with the (zero finding) Newton's Method with a prescribed initial guess. We focus on the latter scheme, both because it is more promising from a numerical point of view, e.g., there are no inequality constraints, and because -roughly speaking- the higher smoothness of the continuous time trajectories should provide faster convergence. We prove in Theorem \ref{thm:concistency} the consistency of this scheme, i.e., the Newton's Method converges at each time step and the sequence of the backward Euler iterates converges to a local minimizer of $\FEneN.$  Then we define Algorithm \ref{alg:backEuler} by choosing a particular \emph{a posteriori} stopping criterion that allows us to prove a stability estimate (see Proposition \ref{prop:stability}). We finally prove in Theorem \ref{thm:convergence} the convergence of Algorithm \ref{alg:backEuler} and give in Proposition \ref{prop:rateofconv} a sharp estimate of the rate of such convergence. The combination of Theorem \ref{thm:mainresult1}, Theorem \ref{thm:concistency}, and Theorem \ref{thm:convergence} leads to (see Corollary \ref{cor:fundamentalcorollary})
$${\lim_\IndexN}^*\mathcal{I}_\IndexN\big(\lim_\iter \bmuh^\iter  \big)=\mu^*dx,$$
where we denoted by $\lim^*$ the limit in the weak$^*$ topology of measures, and by $\mathcal{I}_\IndexN$ the embedding $\Rpn\hookrightarrow\M^+$ deriving from the FEM coordinates. 

Finally (Section \ref{sec:gamma}), we perform some numerical tests of Algorithm \ref{alg:backEuler} in few cases where the optimal transport density is known. Also we check numerically the assumptions we made for the well-posedness and well-conditioning of the discrete variational problems.

\section{$\Gamma$-convergence of finite elements discretization}
\label{sec:gamma}
\subsection{Finite elements discretization of $\Ene$}\label{subsec:FEMdis}
We want to introduce and study a sequence of functionals that
provides a well-designed sequence of approximations of $\Ene$
for minimization purposes. In order to be termed "well designed" the
sequence needs to fulfill certain desirable properties as preserving
the convexity of $\Ene$, having (at each stage) finite
dimensional domain, possibly being smooth, and being relatively easy
to be evaluated at any given point. In addition, since
we are interested in the minimization of $\Ene$, we need to be
able to prove that any sequence of minimizers converges to the
minimizer of $\Ene$. We introduce a suitable FEM-based
sequence of approximations of $\Ene$ in the present subsection
and we will prove its convergence in the next subsection.

Let us remark that, as pointed out in \cite{EvGa99} and in \cite{FaPiPu19},
the role of $\Omega$ is not relevant in our work as long as it
is an open bounded convex set compactly containing the set
$\co \left(\support f^+\cup \support f^-\right),$ where $\co$ stands for the convex envelope operator and $\support$ for the support. Therefore, possibly
up to translations an dilations, we can assume $\Omega:=(0,1)^\Dim,$ a very convenient choice for our purposes. This choice prevents the need of the analysis of
the so-called geometric error that occurs when considering
approximations of $\Omega$ by unions of simplices which are the elements of choice in our approach.

First, we pick a sequence of nested triangulations
$\Triang^\IndexN =\{T_j^{\IndexN}\}_{j=1}^{\Ntdens}$ of $\Omega$ indexed over $\IndexN\in \N$ and obtained by uniform refinements.  Let us denote by $\mathcal P_1(\Triang^\IndexN)$ the space of continuous real
functions $v$ such that $v|_{T_j^\IndexN}$ is an affine function for any
$j=1,2,\dots,\Ntdens.$ Similarly, $\mathcal P_0(\Triang^\IndexN)$ is the
space of functions $v\in L^\infty(\Omega)$ such that $v|_{ T_j^\IndexN}$ is equivalent to a constant function for any $j=1,2,\dots,\Ntdens.$ Then we can consider the finite dimensional function spaces
\begin{equation*}
  \PMeaN:=\mathcal P_0(\Triang^\IndexN),\;\;
  \W_\IndexM:=\left\{u\in \mathcal P_1(\Triang^{\IndexM+1}):
    \int_\Omega u dx=0\right\}\;,
\end{equation*}
and denote by $\PMeaN^+$ the non-negative cone in $\PMeaN$, i.e.,
$$\PMeaN^+:=\left\{\mu\in \mathcal P_0(\Triang^\IndexN):
    \mu|_{T_j^\IndexN}\geq 0,\, \forall T_j^\IndexN\in \Triang^{\IndexN}\right\}.$$
Second, we pick a monotone sequence $\{\delta_{\IndexM}\}$ such that $\delta_{\IndexM}\downarrow 0$ as $\IndexM\to +\infty$.

We can now introduce a double sequence $\{\EneNM\}_{\IndexN,\IndexM\in \N}$ of approximations of the functional $\Ene:\M^+(\Omega)\rightarrow \R\cup \{+\infty\}$ that
is a good candidate for the variational approximation of $\Ene.$ Notice that we can re-write $\Ene$, for notational convenience,
as
\begin{equation*}
  \Ene(\mu):=\mathcal L(\mu)\;+\;\int_\Omega d\mu,
\end{equation*}
where
$\mathcal L(\mu):=\sup_{u\in L^2(\Omega),\int_\Omega u dx=0} \mathcal
D(\mu,u)$ and
\begin{equation*}
  \mathcal D(\mu,u):=
  \begin{cases}
    2 \int_\Omega f u\,dx-\int_\Omega|\nabla u|^2d\mu
    & \text{ if } u\in \mathscr C^1(\overline{\Omega}),\;\int_\Omega u dx=0\\
    +\infty& \text{ otherwise}
  \end{cases}.\label{eq:Ddef}
\end{equation*}
Then we define
\begin{equation}
\EneNM(\mu):=\begin{cases}
    \sup_{u\in \W_\IndexM} \mathcal D_\IndexM(\mu,u)+\int_\Omega d\mu& \text{ if }\mu\in\PMeaN^+ \\
    +\infty& \text{ otherwise}
  \end{cases},\label{eq:Llddef}
\end{equation}
where, for any $\mu\in L^\infty_+(\Omega):=\{\mu\in L^\infty(\Omega):\,\mu\geq 0\text{ a.e.}\}$, we set
\begin{equation}
  \mathcal D_\IndexM(\mu,u):=
    \begin{cases}
      2 \int_\Omega f u\,dx-\int_\Omega(\mu+\delta_{\IndexM})|\nabla u|^2dx
      & \text{ if } u\in \W_\IndexM\\
      +\infty& \text{ otherwise}
    \end{cases}.\label{eq:Dkdef}
\end{equation}

\begin{remark}
  We would like to stress here that the use of the two indices $\IndexN$ and $\IndexM$ is dictated by notational convenience. Indeed, as we will see in  the next sections, the principal object of our study will be the diagonal sequence $\Ene_{\IndexN,\IndexN}$, and in this case it produces the wanted combination of meshes and spaces $\Triang^{\IndexN}$ for $\PMeaN$ and $\Triang^{\IndexN+1}$ for $\W_{\IndexN}$.
\end{remark}

In our work we will use the functionals
\begin{align}
\Ene_{\infty,\IndexM}(\mu)&:=\inf_\IndexN \EneNM\\
\widetilde{\Ene}_\IndexM(\mu)&:=\scm \Ene_{\infty,\IndexM}(\mu),\label{eq:scmdef}\\
\Ene_\infty(\mu)&:=\sup_\IndexM\widetilde{\Ene}_\IndexM(\mu),\label{eq:einftydef}
\end{align}
where $\scm $ stands for the lower semicontinuous envelope with
respect to the weak$^*$ topology of measures, i.e.,
\begin{equation*}
  \scm \mathcal F(\mu)
  :=\sup\left\{\mathcal G(\mu),\;\mathcal G
    \leq \mathcal F,\; \mathcal G\text{ is l.s.c. in the weak$^*$ topology}\right\}.
\end{equation*}
Here and throughout the paper, when we consider measures that are absolutely continuous with respect to the Lebesgue measure, we use (by a slight abuse of notation) the same symbol both for the density of a measure and the measure itself. Accordingly, $\mu+\delta$ denotes the measure $\mu+ \delta dx.$

Note that the definition of the functionals $\EneNM$ is taylored to the need of treating the convergence of minimizers. A more explicit representation is obtained by means of the linear isomorphism $\mathcal{I}_\IndexN$ mapping $\R^{\Ntdens}$  onto $\PMeaN$. Precisely, for any $\bmu=\{\comp{\bmu}{1},\comp{\bmu}{2},\dots,\comp{\bmu}{\Ntdens}\}\in \R^{\Ntdens}$ we can define
$$\mathcal {I}_\IndexN(\bmu):=\sum_{i=1}^{\Ntdens} \bmui \chi_{T_i^\IndexN}(x),$$ 
and, for any $\mu\in\PMeaN$, we denote by $\bmu$ the vector
\begin{equation}\label{eq:coordinatedef}
	\bmu:=\mathcal I_\IndexN^{-1}(\mu)=\left(\frac{\int_{T_1^\IndexN} \mu dx}{|T_1^\IndexN|},\frac{\int_{T_2^\IndexN} \mu dx}{|T_2^\IndexN|},\dots,\frac{\int_{T_{\Ntdens}^\IndexN} \mu dx}{|T_{\Ntdens}^\IndexN|}\right)^t.
\end{equation}
Analogously, given a basis $\{\phi_{\IndexM,1},\phi_{\IndexM,2},\dots,\phi_{\IndexM,\Npot}\}$ of $\W_\IndexM$ (where $\Npot:=\dimension \W_\IndexM$), we can identify any function $u\in \W_\IndexM$ with its coordinates $\bu\in \R^{\Npot}$, where $u(x)=\sum_{\Ipot=1}^{\Npot}\bui \phi_{\IndexM,\Ipot}(x).$

Given $\IndexN,\IndexM\in \N$ and $i\in\{1,2,\dots,\Ntdens\}$, we us introduce the matrices
\begin{equation}
\Anmi:=\left(\int_{T_i^{\IndexN}}\nabla \phi_{\IndexM,\Ipot}\cdot\nabla \phi_{\IndexM,\Jpot}dx\right)_{\Ipot,\Jpot=1,\dots,\Npot}
\end{equation}
and, for any $\bmu\in \Rpn$, the matrix-valued function (the stiffness matrix)
\begin{equation}
\Anm(\bmu):=\sum_{i=1}^{\Ntdens}(\bmui+\delta_{\IndexM})\Anmi.
\end{equation}
We extend such a definition to each $\mu\in L^\infty(\Omega)$, $\mu\geq 0$ a.e., by setting
\begin{equation}
\Anm(\mu):=\left(\int_{\Omega}(\mu+\delta_{\IndexM})\nabla \phi_{\IndexM,\Ipot}\cdot\nabla \phi_{\IndexM,\Jpot}dx\right)_{\Ipot,\Jpot=1,\dots,\Npot}.
\end{equation}
We also denote by $\bfk$ the load vector $(\comp{\bfk}{1},\dots, \comp{\bfk}{\Npot})^t$, where $\bfki=\int_{\Omega}f\phi_{\IndexM,\Ipot}dx$, $\Ipot=1,2,\dots,\Npot.$
\begin{remark}
We remark that, in order to distinguish indeces denoting an element of a sequence from indeces refering to a component of a vector, we use the following convention: in the latter instance we add parentheses to the index, e.g., given the sequence of vectors $\{\bmuh\}_{\IndexN\in \N}$, we denote by $\bmuhi$ the $i$-th component of the vector $\bmuh$.
\end{remark}
It is worth pointing out that the operation of taking the supremum of  $\mathcal D_\IndexM$ among $u\in \W_\IndexM$ (used in the definition of $\EneNM$ given in eq. \eqref{eq:Llddef}) corresponds to calculating the $\mathcal P_1$-finite element approximation of the solution of
\begin{equation}\label{eq:pdewithdelta}
\begin{cases}
-\divergence((\mu+\delta_{\IndexM})\nabla u)=f& \text{ in }\Omega\\
\partial_n u=0& \text{ on }\partial\Omega\\
\int u dx=0
\end{cases}.
\end{equation}
More precisely, given $\mu\in \PMeaN^+$, we have
\begin{equation}\label{eq:supissolution}
\sup_{u\in \W_\IndexM}\mathcal D_\IndexM(\mu,u)=\max_{u\in \W_\IndexM}\mathcal D_\IndexM(\mu,u)=\int f u\Of{\mu} dx=\int (\mu+\delta_{\IndexM})|\nabla u\Of{\mu}|^2dx,
\end{equation}
where $u\Of{\mu}$ is the unique element of $\W_\IndexM$ such that 
\begin{equation}\label{eq:umukdef}
\int_\Omega \nabla u\Of{\mu}\cdot \nabla v(\mu+\delta_{\IndexM})dx=\int_\Omega f v dx,\;\forall v\in \W_\IndexM,
\end{equation}
in other words, 
\begin{equation}\label{u_mu_definition}
    \Anm(\bmu)\bu\Of{\bmu}=\bfk.
\end{equation} 
In view of this we set $\DEneNM(\bmu):=\EneNM(\mu)$, that is
\begin{equation}\label{eq:Edef}
\DEneNM(\bmu)=\begin{cases}
(\bfk)^t [\Anm(\bmu)]^{-1}\bfk+\langle \bmu;\bmh\rangle& \text{ if }\ \bmui\geq 0,\,i=1,2,\dots,\Ntdens\\
+\infty&\text{ otherwise}
\end{cases},
\end{equation}
where we denoted by $\bmh$ the area vector $\left(\int_{T_1^{\IndexN}}dx,\int_{T_2^\IndexN}dx,\dots,\int_{T_{\Ntdens}^\Ntdens}dx\right)^t.$

Let us consider the coordinate square map $\cdot^{\bs 2}:\R^{\Ntdens}\rightarrow\R^{\Ntdens}$ defined by setting
\begin{equation}\label{eqn:coordinatesquaremap}
\Sq{\bsigma}:=(\comp{\bsigma}{1}^2,\comp{\bsigma}{2}^2,\dots,\comp{\bsigma}{\Ntdens}^2).
\end{equation}
We can introduce the functional
$$\FEneNM(\bs \sigma):=\DEneNM(\bs{\sigma^2}).$$
Notice that the image under the map $\cdot^{\bs 2}$ of any local minimizer $\bs{\hat\sigma}$ of $\FEneNM$ on $\R^{\Ntdens}$ is a local minimizer of $\DEneNM$ on $\Rpn.$ Moreover $\Sq{\bs{\hat\sigma}}$ is a global minimizer of $\DEneN$ on $\Rpn$ due to the convexity of the objective.

The next proposition collects some properties of $\DEneNM$ and $\FEneNM$ that will be useful later on.
\begin{proposition}[Differential properties of $\DEneNM$ and $\FEneNM$]\label{prop:differentiability}
For any $\IndexN,\IndexM\in \N$ the functional $\DEneNM$ is convex on $\R^{\Ntdens}$ and real analytic on 
$$\R_{\delta_{\IndexM}}^{\Ntdens}:=\{\bmu\in \R^{\Ntdens}:\,\bmui>-\delta_{\IndexM},\,i=1,2,\dots,\Ntdens\}\supset\Rpn.$$
The functional $\FEneNM$ is real analytic on $\R^{\Ntdens}.$ In particular, $\forall \bs\mu\in \Rpn$ and any $\bsigma\in \R^{\Ntdens}$ such that $\Sq{\bsigma}=\bmu$, we have
\begin{align}
&\frac{\partial \DEneNM}{\partial\bmui} (\bmu)=1- (\bu\Of{\bmu})^t \Anmi\bu\Of{\bmu}\label{eqn:firstderivative}\\
&\frac{\partial \FEneNM}{\partial\bsigmai} (\bsigma)=2 \bsigmai \frac{\partial \DEneNM}{\partial\bmui} (\Sq{\bsigma})\label{eqn:firstderivativeF}\\
&\frac{\partial^2 \DEneNM}{\partial\bmui\partial\bmuj}(\bmu)=(\bu\Of{\bmu})^t \Anmi[\Anm(\bmu)]^{-1}\Anmj\bu\Of{\bmu}\label{eqn:secondderivative}\\
&\frac{\partial^2 \FEneNM}{\partial\bmui\partial\bmuj}(\bsigma)=4\bsigmai\bsigmaj\frac{\partial^2 \DEneNM}{\partial\bmui\partial\bmuj}(\Sq{\bsigma})+2\delta_{i,j}\frac{\partial \DEneNM}{\partial\bmui} (\Sq{\bsigma}),\label{eqn:secondderivativeF}
\end{align}
where $\bu\Of{\bmu}$ is defined in \eqref{u_mu_definition}. The convex subdifferential $\partial \DEneNM(\bmu)$ of $\DEneNM$ can be characterized on the closure of $\R_{\delta_{\IndexM}}^{\Ntdens}$ by
\begin{equation}\label{eqn:subdifferential}
\partial \DEneNM(\bs\mu)=\left\{\bs\xi \in \R^{\Ntdens}:\comp{\bs\xi}{i}\leq \frac{\partial \DEneNM}{\partial \bmui}(\bmu)\;\forall i,\;\comp{\bs\xi}{i}=\frac{\partial \DEneNM}{\partial \bmui}(\bmu)\,\forall i:\bmui>0\right\}.
\end{equation}
The minimal norm subdifferential $\partial^\circ \DEneNM(\bmu):=\argmin_{\xi\in \partial \DEneNM(\bmu)}\|\xi\|_2$ satisfies
\begin{equation}\label{eqn:minimalsubdifferential}
\big(\partial^\circ \DEneNM(\bmu)\big)_i=\begin{cases}
\frac{\partial}{\partial\bmui}\DEneNM(\bmu)& \text{ if }\bmui>0\\
\left(\frac{\partial}{\partial\bmui} \DEneNM(\bmu) \right)^-& \text{ if }\bmui=0
\end{cases}
\end{equation}
\end{proposition}
\begin{proof}
The convexity of $\EneNM$ is elementary because it is defined as the supremum of affine functionals and the composition with $\mathcal I_\IndexN$ clearly preserves such a property.

Notice that $A(\bmu+\delta_\IndexN)$ is a positive definite (and thus invertible) symmetric matrix for any $\bmu\in \Rpn$ and linearly depending on $\bmu$. Since the matrix inversion is a real analytic operation the function $\DEneNM$ is real analytic on $\R^{\Ntdens}_{\delta_{\IndexM}}.$

Using $\partial A^{-1}(\bmu+\delta_h)/\partial \bmui=-A^{-1}(\bmu+\delta_h)\partial A^{-1}(\bmu+\delta_h)/\partial\bmui A^{-1}(\bmu+\delta_h)$ and \eqref{u_mu_definition}, we obtain \cref{eqn:firstderivative} and similarly \cref{eqn:secondderivative}. By the chain rule we obtain \cref{eqn:firstderivativeF} and \cref{eqn:secondderivativeF}. Notice that $\DEneNM$ is the restriction to the non-negative cone of a (similarly defined) real analytic function defined on $\R_{\delta_{\IndexM}}^{\Ntdens}.$ For $\epsilon>0$ small enough, using the definition of convex subdifferential, we can write
\begin{align*}
&\DEneNM(\bmu+\epsilon \ei)-\DEneNM(\bmu)\geq \epsilon\langle\bs\xi;\ei\rangle=\epsilon\comp{\bs\xi}{i},\;\;\forall \bs\xi\in \partial \DEneNM(\bmu),\,\forall i\\
&\DEneNM(\bmu-\epsilon \ei)-\DEneNM(\bmu)\geq \epsilon\langle\bs\xi;-\ei\rangle=-\epsilon\comp{\bs\xi}{i},\;\;\forall \bs\xi\in \partial \DEneNM(\bmu),\,\forall i:\bmui>0.
\end{align*}
Dividing by $\epsilon$ and passing to the limit as $\epsilon\to 0^+$ we get
\begin{align*}
\frac{\partial \DEneNM}{\partial \bmui}(\bmu)&\geq \comp{\bs\xi}{i},\;\forall \bs\xi\in \partial \DEneNM(\bmu),\,\forall i\\
\frac{\partial \DEneNM}{\partial \bmui}(\bmu)&\leq \comp{\bs\xi}{i},\;\forall \bs\xi\in \partial \DEneNM(\bmu),\,\forall i:\bmui>0
\end{align*}
Thus \eqref{eqn:subdifferential} follows. Equation \eqref{eqn:minimalsubdifferential} is immediately obtained by minimizing the $L^2$ norm over this set. 
\end{proof}
\begin{corollary}
Any critical point $\bmu_{\IndexN,\IndexM}^*\in \Rpn$ for $\DEneNM$ (i.e., $\partial^\circ \DEneNM(\bmu_{\IndexN,\IndexM}^*)=0$) is a global minimizer for $\DEneNM$. If $\bsigma_{\IndexN,\IndexM}^*\in \R^N$ is a local minimizer for  $\FEneNM$, then $\bmu_{\IndexN,\IndexM}^*:=\Sq{(\bsigma_{\IndexN,\IndexM}^*)}$ is a global minimizer for $\DEneNM$.
\end{corollary}
Also, using Proposition \ref{prop:differentiability}, we can derive a discrete version of Monge Kantorovich equations \eqref{eq:EvGaPDE}.
\begin{corollary}[Discrete Monge-Kantorovich equations]
For any $\IndexN,\IndexM\in \N$ there exists at least one minimizer $\mu_{\IndexN,\IndexM}^*$ of $\EneNM$ on $\mathcal \mathcal \PMeaN^+.$ We have $\mu\in \argmin \EneNM $ if and only if there exists a (unique) $u\in \W_\IndexM$ such that the following equations hold true
\begin{equation}\label{eqn:discretemongekantorovich}
\begin{cases}
\Anm(\bmu)\bu=\bfk& \\
\bu^t \Anmi\bu=1&\,\forall i:\,\bmui>0\\
\bu^t \Anmi\bu\leq 1&\,\forall i:\,\bmui=0
\end{cases}
\end{equation}
\end{corollary}
\begin{proof}
The function $\DEneNM$ is convex and coercive (i.e., $\liminf_{\|\bmu\|_1\to+\infty}\DEneNM(\bmu)=+\infty$) and thus the existence of a minimizer follows by the direct method.

Again by convexity the condition $0\in \partial \DEneNM(\bmu)$ is \emph{equivalent} to $\bmu\in \argmin \DEneNM$ and thus equivalent to $0= \partial^\circ \DEneNM(\bmu).$ The latter equations can be written precisely as \eqref{eqn:discretemongekantorovich} using \cref{eqn:minimalsubdifferential}.
\end{proof}
\subsection{Convergence of minimizers of $\EneN:=\Ene_{\IndexN,\IndexN}$ to $\mu^*$}
The aim of the present section is to prove the following result.
\begin{theorem}\label{thm:mainresult1}
Let $\muN^*\in \argmin \EneN.$ Then 
\begin{equation}
\muN^*\rightharpoonup \mu^*,
\end{equation}
the optimal transport density of the Monge-Kantorovich equations.
\end{theorem}
\begin{remark}
We choose to focus our study on the "diagonal sequence" of functionals $\EneN:=\Ene_{\IndexN,\IndexN}$ to simplify our notation and avoid few technicalities in our proofs. However, the reader can easily check that our results still hold for any subsequence $\{\Ene_{\IndexN_j,\IndexM_j}\}_{j\in \N}$ such that $\{\IndexN_j\}$ and $\{\IndexM_j\}$ are non-decreasing diverging sequences.
\end{remark}

$\Gamma$-convergence is the main tool that we are to going to use in the proof of  Theorem \ref{thm:mainresult1}, and, to fix our notation, we recall in the following paragraphs the needed definitions and properties. Let $(X,\tau)$ be a topological space and, for any $x\in X$, let us denote by $\mathcal N(x)$ the filter of the neighborhoods of $x.$ Let $f_{\IndexN}:X\rightarrow \overline{\R},$ $\IndexN\in \N.$ We define
\begin{align*}
\left(\gammaliminf_{\IndexN\to +\infty}f_{\IndexN}\right)(x)&:=\sup_{U\in \mathcal N(x)}\liminf_{\IndexN\to \infty}\inf_{y\in U}f_{\IndexN}(y),\\
\left(\gammalimsup_{\IndexN\to +\infty}f_{\IndexN}\right)(x)&:=\sup_{U\in \mathcal N(x)}\limsup_{\IndexN\to \infty}\inf_{y\in U}f_{\IndexN}(y).
\end{align*}
If there exists a function $f:X\rightarrow \R\cup\{-\infty,+\infty\}$ such that
\begin{equation}\label{gammadef}
\left(\gammalimsup_{\IndexN\to +\infty}f_{\IndexN}\right)(x)\leq f(x)\leq \left(\gammaliminf_{\IndexN\to +\infty}f_{\IndexN}\right)(x),\;\;\forall x\in X,
\end{equation}
then we say that $f_{\IndexN}$ $\Gamma$-converges to $f$ with respect to the topology $\tau$ and we write $f_{\IndexN}\gammaconverge f$ or $\gammalim f_{\IndexN}=f.$
Our main interest on this notion of convergence is given by the following property (cfr. for instance \cite[Cor. 7.20]{Da93}).   Assume that $f_{\IndexN}\gammaconverge f$ and $x_{\IndexN}$ is a minimizer of $f_{\IndexN}$. Then any cluster point $x$ of $\{x_{\IndexN}\}$ is a minimizer of $f$ and $f(x)=\limsup_{\IndexN} f_{\IndexN}(x_{\IndexN}).$ If moreover $x_{\IndexN}$ converges to $x$ in the topology $\tau,$ then $f(x)=\lim_{\IndexN} f_{\IndexN}(x_{\IndexN}).$

A useful property of $\Gamma$-convergence is that it is well-behaving under monotone limits of lower semicontinuous functionals, \cite[Prop. 5.4, Rem. 5.5, Prop. 5.7]{Da93}. Indeed we can prove a preliminary result exploiting the two monotonicities of the double sequence $\EneNM.$
\begin{proposition}
Let the functionals $\Ene$, $\EneNM$ $\widetilde {\Ene}_\IndexM$ and $\Ene_\infty$ be defined as in \eqref{eq:Llddef}, \eqref{eq:scmdef}, and \eqref{eq:einftydef}. Then
\begin{equation}
\gammalim_\IndexM \widetilde{\Ene}_\IndexM=\Ene_\infty\label{eqn:easygammalim2}
\end{equation}
and
\begin{equation}
\gammalim_\IndexN \EneN=\gammalim_\IndexM\gammalim_\IndexN \EneNM=\gammalim_\IndexN\gammalim_\IndexM \EneNM=\Ene_\infty\label{eqn:easygammalim3}
\end{equation}
\end{proposition}
\begin{proof}
Equation \ref{eqn:easygammalim2} follows easily observing that $\{\widetilde{\Ene}_\IndexM\}$ is an increasing sequence of lower semicontinuous functionals and hence its $\Gamma$-limit coincides with its point-wise limit.   

To prove the second statement, we first notice that 
\begin{equation}
\label{eqn:interchange}
\sup_\IndexM \Ene_{\infty,\IndexM}=\sup_\IndexM\inf_\IndexN\EneNM=\inf_\IndexN\sup_\IndexM \EneNM.
\end{equation}
The first equality follows by definition. For the second equality we proceed as follows. We first notice that, if $\mu\in \cup_{\IndexN\in \N}\PMeaN^+,$ then there exists $\bar \IndexN$ such that, for any  $\IndexN\geq \bar \IndexN$, $\mu\in \PMeaN^+$ and $\EneNM(\mu)=\Ene_{\bar \IndexN,\IndexM}(\mu)$ for all $\IndexM\in \N.$ Then we can write
\begin{align*}
\sup_\IndexM\inf_ \IndexN\EneNM(\mu)=&\sup_\IndexM\begin{cases}
	\Ene_{\bar \IndexN,\IndexM}(\mu)& \text{ if }\mu\in \cup_{ \IndexN \in \N}\PMeaN^+\\
	+\infty& \text{ otherwise}
\end{cases}\\
=&\begin{cases}
	\sup_\IndexM \Ene_{\bar \IndexN,\IndexM}(\mu)& \text{ if }\mu\in \cup_{ \IndexN \in \N}\PMeaN^+\\
	+\infty& \text{ otherwise}
\end{cases}= \inf_ \IndexN\sup_\IndexM\EneNM(\mu)\;.
\end{align*}
If we pick the lower semicontinuous regularization of \cref{eqn:interchange} we get
\begin{equation}\label{scmgammalim1}
\scm\sup_\IndexM \Ene_{\infty,\IndexM}=\scm\sup_\IndexM\inf_\IndexN\EneNM=\scm\inf_\IndexN\sup_\IndexM \EneNM.
\end{equation}
Note that, since iterated $\sup$ operators commute, using \cref{eqn:easygammalim2}  we have
\begin{equation}\label{scmgammalim2}
\scm\sup_\IndexM \Ene_{\infty,\IndexM}=\sup_\IndexM\scm\Ene_{\infty,\IndexM}=\gammalim_\IndexM\widetilde{\Ene}_\IndexM=\Ene_\infty.
\end{equation}
On the other hand, using again the monotonicity we have
\begin{align}
\scm\sup_\IndexM\inf_\IndexN\EneNM=&\sup_\IndexM\scm\inf_\IndexN\EneNM=\sup_\IndexM\gammalim_\IndexN \EneNM\notag\\
=&\gammalim_\IndexM\gammalim_\IndexN \EneNM,\label{gammalim_kh}\\
\scm\inf_\IndexN\sup_\IndexM\EneNM=&\scm\inf_\IndexN\gammalim_\IndexM\EneNM=\gammalim_\IndexN\gammalim_\IndexM \EneNM\label{gammalim_hk}
\end{align}
The combination of equations \eqref{scmgammalim1}, \eqref{scmgammalim2}, \eqref{gammalim_kh}, and \eqref{gammalim_hk} leads to 
$$\gammalim_\IndexN\gammalim_\IndexM \EneNM=\gammalim_\IndexM\gammalim_\IndexN \EneNM=\Ene_\infty.$$
Since the iterated Gamma-limits exist and coincide, also the diagonal Gamma-limit exists and it is equal to the iterated limits, i.e., the first equality of \eqref{eqn:easygammalim3} holds true. 
\end{proof}

The above proven $\Gamma$-convergence ensures that any cluster point of a sequence of minimizers of $\EneN$ is a minimizer of $\Ene_\infty,$ but it does not imply by itself the existence of the cluster point. However, we are able to show that indeed this cluster point does exist.
\begin{proposition}\label{prop:gammaconvergence}
For any $\IndexN\in \N$, let $\muN^*\in \argmin \EneN.$ Then $\{\muN^*\}$ is pre-compact in the weak$^*$ topology of measures. Any cluster point $\mu_\infty^*$ of $\{\muN^*\}$ lies in $\argmin \Ene_\infty$ and we can extract a subsequence $l\mapsto \IndexN_l$ such that
\begin{equation}
\mu_{\IndexN_l}^* \rightharpoonup \mu_\infty^*\;\;\text{ and }\;\;\EneN(\muN^*)\to \Ene_\infty(\mu_\infty^*)\;\text{ as }l\to +\infty.
\end{equation}
\end{proposition} 
\begin{proof} Notice that, denoting by $\Leb$ the $\Dim$-dimensional Lebesgue measure,  
\begin{equation}\label{massestimate}
\sup_\IndexN\int\muN^*\,dx\leq \sup_\IndexN \EneN(\muN^*)\leq \sup_h\EneN(\Leb)=\sup_\IndexM\Ene_{1,\IndexM}(\Leb).
\end{equation}
Let $u_\IndexM\in \W_\IndexM$ be the solution of \eqref{eq:umukdef} with $\mu=\Leb.$ Notice that, denoting by $C$ the Poincar\'e constant of $\Omega$ and using \eqref{eq:supissolution}, we have
\begin{align*}
\sup_\IndexM\Ene_{1,\IndexM}(\Leb)=& \sup_\IndexM 2\int_\Omega f u_\IndexM dx-\int_\Omega (1+\delta_{\IndexM}) |\nabla u_\IndexM|^2 dx+\int_\Omega dx\\
=&\sup_\IndexM\int_\Omega f u_\IndexM dx+\int_\Omega dx\leq \sup_\IndexM\|f\|_2\|u_\IndexM\|_2 +m(\Omega)\\
\leq& \sup_\IndexM\left(\frac C{1+\delta_{\IndexM}}\|f\|_\infty^2+1\right)\Leb(\Omega)=(C\|f\|_\infty^2+1)\Leb(\Omega)<+\infty.
\end{align*}
Here we used the Poincar\'e Inequality and the weak formulation \eqref{eq:pdewithdelta}.
It follows that the sequence $\muN$ has bounded mass:
$$\sup_\IndexN\int\muN^*<+\infty.$$
Thus we can extract a weak$^*$ converging subsequence, and the limit is a minimizer of the $\Gamma$-limit functional (i.e., of $\Ene_\infty$ by \ref{eqn:easygammalim2})  due to basic properties of $\Gamma$-limits, cfr. e.g., \cite[Cor. 7.20]{Da93}. 
\end{proof}

The proof of Theorem \ref{thm:mainresult1} essentially relies on  the uniqueness and the $L^\infty$ regularity of the optimal transport density $\mu^*.$ Indeed, under our assumptions, the optimal transport density is uniquely determined and it is an absolutely continuous  measure having $L^\infty$ density with respect to the Lebesgue measure, cfr. \cite{EvGa99,AmCaBrBuVi03,FeMc02}. This is a key element of our construction because the functional $\Ene$ and $\Ene_\infty$ coincide on $L^\infty_+$, as we state in Proposition \ref{technicallemma} below. We need to prove a continuity property first.

\begin{lemma}\label{continuitylemma}
Given $\mu\in L^\infty_+(\Omega)$,then, for any $\delta>0$, there exists a sequence $\{\muN\}$, with $\muN\in \PMeaN^+$ such that
\begin{enumerate}[i)]
\item $\muN\to\mu$ almost everywhere in $\Omega$,
\item $\max\{\|\mu\|_\infty,\sup_\IndexN\|\muN\|_\infty\}<+\infty$,
\item for any $\IndexM\in \N$,
\begin{equation}
\lim_\IndexN\ \mathcal D_\IndexM(\muN,u) +\int \muN dx= \mathcal D_\IndexM(\mu,u) +\int \mu. 
\end{equation}
\end{enumerate}
\end{lemma}
\begin{proof}
Let us define
\begin{equation}
\muN:=\sum_{i=1}^{\Ntdens}\bmuhi\chi_{T^\IndexN_i}(x):=\sum_{i=1}^{\Ntdens}\frac{\int_{T^\IndexN_i}\mu dx}{|T^\IndexN_i|}\chi_{T^\IndexN_i}(x).
\end{equation}
By the Lebesgue Differentiation Theorem it follows that $\muN\to\mu$ almost everywhere in $\Omega$. Clearly
$$\|\muN\|_\infty\leq \max_{i=1,\dots,\Ntdens}|\bmuhi|\leq \|\mu\|_\infty.$$
By the Lebesgue Dominated Convergence Theorem we have that 
$$\lim_\IndexN \int_\Omega \muN\psi dx=\int_\Omega \mu\psi dx,\;\forall \psi \in L^1(\Omega).$$
Note that in particular $\lim_\IndexN \int_\Omega \muN dx=\int_\Omega \mu dx$ follows. 
Thus, noticing that $\nabla \phi_{\IndexM,\Ipot}\cdot\nabla \phi_{\IndexM,\Jpot}\in L^1(\Omega)$, and using the notation $\Am_{\Ipot,\Jpot}(\mu):=\int_\Omega \nabla \phi_{\IndexM,\Ipot}\cdot\nabla \phi_{\IndexM,\Jpot} (\mu+\delta_{\IndexM})dx,$ for any $\mu\in L^\infty_+(\Omega)$, we have 
$$\lim_\IndexN \Anm(\muN)=\lim_\IndexN \Am(\muN)=\Am(\mu)$$
 for any $\IndexM\in \N.$ Also notice that, denoting by $\Am(1-\delta_{\IndexM})$ the (strictly positive definite) stiffness matrix of the FEM $\mathcal P^1$ discretization of the Neumann Laplacian on zero-mean functions, we have  
\begin{align*}
&\Am(\mu) \succeq \delta_{\IndexM} \Am(1-\delta_{\IndexM})\succ 0\,,\\
&\Anm(\muN) \succeq \delta_{\IndexM} \Anm(1-\delta_{\IndexM})\succ 0\text{ uniformly in }\IndexN,
\end{align*}
where $\succeq$ denotes the canonical ordering of symmetric semidefinite matrices.
Hence all the considered  matrices are invertible and we have
$$\lim_\IndexN[\Anm(\muN)]^{-1}=\lim_\IndexN[\Am(\muN)]^{-1}=[\Am(\mu)]^{-1}.$$
Finally we have
\begin{align*}
&\lim_\IndexN \mathcal D_\IndexM(\muN,u)=\lim_\IndexN \bs f_\IndexM^t[\Anm(\muN)]^{-1} \Anm(\muN) [\Anm(\muN)]^{-1} \bs f_\IndexM\\
=& \bs f_\IndexM^t[\Am(\mu)]^{-1} \Am(\mu) [\Am(\mu)]^{-1} \bs f_\IndexM= \mathcal D_\IndexM(\mu,u).
\end{align*}
\end{proof}
We wold like to remark that the above result can be interpreted as a simplified application of $\Gamma$-convergence of quadratic forms on $\W_\IndexM$.
\begin{proposition}\label{technicallemma}
Under our assumptions  we have
\begin{align}
&\Ene(\mu)\leq \Ene_\infty(\mu),\;\forall \mu\in \M^+(\Omega): \mathcal L(\mu)<+\infty,\label{technicallemmaclaim1}\\
& \Ene(\mu)\geq \Ene_\infty(\mu),\;\forall \mu\in L^\infty_+(\Omega).\label{technicallemmaclaim2}
\end{align}
\end{proposition}
\begin{proof}
For notational convenience we define the sets
$$
\PMeaI:=\cup_{ \IndexN \in \N} \PMeaN^+\;\;\; \W_\infty:=\cup_{ \IndexM \in \N} \W_\IndexM\,.
$$
We begin proving \eqref{technicallemmaclaim1}. Let us pick $\mu\in \M^+$. Then, using the notation $\llcorner$ for the restriction of a function to a set, i.e., 
$$f\llcorner_A(x):=\begin{cases}
f(x)& \text{ if }x\in A\\
+\infty&\text{otherwise}
\end{cases}\;,$$ we can write
\small
\begin{align*}
&\Ene_\infty(\mu)=\sup_\IndexM\scm\inf_\IndexN\EneNM(\mu)\\
=&\scm\left[\sup_\IndexM\left( \sup_{u\in \W_\IndexM} 2\int f u dx-\int |\nabla u|^2(d\nu+\delta_{\IndexM}dx)+\int d\nu\right)\llcorner_{\nu\in \PMeaI^+}\right](\mu)\\
=& \scm\left[\left(\sup_\IndexM \sup_{u\in \W_\IndexM} 2\int f u dx-\int |\nabla u|^2(d\nu+\delta_{\IndexM}dx)+\int d\nu\right)\llcorner_{\nu\in \PMeaI^+}\right](\mu)\\
\geq &\scm \left[ \sup_\IndexM \sup_{u\in \W_\IndexM} 2\int f u dx-\int |\nabla u|^2(d\nu+\delta_{\IndexM_0}dx) +\int d\nu\right]\llcorner_{\nu\in \PMeaI^+}(\mu)\\
=& \scm \left[ \sup_{u\in\W_\infty} 2\int f u dx-\int |\nabla u|^2(d\nu+\delta_{\IndexM_0}dx) +\int (d\nu+\delta_{\IndexM_0}dx)\right]\llcorner_{\nu\in \PMeaI^+}(\mu) -\delta_{\IndexM_0}|\Omega|.
\end{align*}
\normalsize
Now, if $\nu\in \PMeaI^+$, we have 
\small
\begin{align*}
&\sup_{u\in\W_\infty} 2\int f u dx-\int |\nabla u|^2(d\nu+\delta_{\IndexM_0}dx) +\int (d\nu+\delta_{\IndexM_0}dx)\\
=& \sup_{u\in H^1(\Omega),\;\int u=0} 2\int f u dx-\int |\nabla u|^2(d\nu+\delta_{\IndexM_0}dx) +\int (d\nu+\delta_{\IndexM_0}dx)\\
\geq& \sup_{u\in \mathscr C^1(\overline\Omega),\;\int u=0} 2\int f u dx-\int |\nabla u|^2(d\nu+\delta_{\IndexM_0}dx) +\int (d\nu+\delta_{\IndexM_0}dx)\\
=& \Ene(\nu+\delta_{\IndexM_0}),
\end{align*}
\normalsize
if instead $\nu\in \M^+\setminus \PMeaI^+$, by our definition of restriction, we have
$$\left( \sup_{u\in\W_\infty} 2\int f u dx-\int |\nabla u|^2(d\nu+\delta_{\IndexM_0}dx) +\int (d\nu+\delta_{\IndexM_0}dx)\right)\llcorner_{\PMeaI^+}=+\infty.$$
In other words, due to the arbitrariness of $\IndexM_0$, for any $\delta>0,$ we have
\begin{equation}
\Ene_\infty(\mu)\geq \scm\left(\Ene\big|_{\PMeaI^+}   \right)(\mu+\delta)-\delta |\Omega|.
\end{equation}
Taking the $\liminf$ as $\delta\downarrow 0$ and using the lower semicontinuity of $\scm\left(\Ene\big|_{\PMeaI^+}   \right)$, we obtain
\begin{equation}
\Ene_\infty(\mu)\geq \scm\left(\Ene\big|_{\PMeaI^+}   \right)(\mu).
\end{equation}
Since $\M^+(\Omega)$ is first countable, the relaxed functional $\scm\left(\Ene\big|_{\PMeaI^+}   \right)$ has the following equivalent characterization (see \cite[Prop. 3.6]{Da93}).
\small
\begin{align}
&\forall \mu\in \M^+(\Omega)\;\exists \{\mu_j\}_{j\in \N}\rightharpoonup^*\mu:\;\mathcal \scm\left(\Ene\big|_{\PMeaI^+}   \right)(\mu)\geq \limsup_j \left(\Ene\big|_{\PMeaI^+}   \right)(\mu_j),\label{limsupchar}\\
&\scm\left(\Ene\big|_{\PMeaI^+}   \right)(\mu)\leq \liminf_j \left(\Ene\big|_{\PMeaI^+}   \right)(\mu_j),\;\;\forall \{\mu_j\}_{j\in \N}\rightharpoonup^*\mu.\label{liminfchar}
\end{align}  
\normalsize
Let us pick $\mu_j$ as in \eqref{limsupchar}: it is evident that, if $\Ene_\infty(\mu)<+\infty$, then $\mu_j\in  \PMeaI^+$. As a consequence and because of the lower semi-continuity of $\Ene$, we obtain
\begin{equation}
\label{onesideinequality}
\Ene_\infty(\mu)\geq \limsup_j \left(\Ene\big|_{\PMeaI^+}   \right)(\mu_j)\geq \liminf_j\Ene(\mu_j)\geq \Ene(\mu).
\end{equation}
This concludes the proof of \eqref{technicallemmaclaim1}.

Before proving \eqref{technicallemmaclaim2}, we claim that, for any $\mu \in \M^+(\Omega),$
\begin{equation}\label{claimcontinuity}
\lim_{\delta\to 0^+}\Ene(\mu+\delta)=\Ene(\mu).
\end{equation}
In order to prove this, we first notice that, again due to lower semicontinuity, we have
\begin{equation}\label{halfclaim}
\liminf_{\delta\to 0^+}\Ene(\mu+\delta)\geq\Ene(\mu).
\end{equation} 
On the other hand, we can prove the reverse inequality for the $\limsup$ using the definition of $\Ene$. Note that we assume that either the supremum defining $\Ene(\mu+\delta)$ is indeed a maximum that is achieved for some $u_\delta,$ or there exists a sequence $u_s\in \mathscr C^1(\overline \Omega)$ such that 
\begin{equation}\label{casesup}
\Ene(\mu+\delta)=\lim_s 2\int f u_sdx-\int |\nabla u_s|^2 (d\mu+\delta dx)+\int (d\mu+\delta dx).
\end{equation} 
In the latter case we have
\begin{align*}
\Ene(\mu+\delta)=&\lim_s 2\int f u_sdx-\int |\nabla u_s|^2 (d\mu+\delta dx)+\int (d\mu+\delta dx)\\
\leq & \lim_s 2\int f u_sdx-\int |\nabla u_s|^2 d\mu+\int d\mu+\delta |\Omega|\\
\leq & \sup_{u\in \mathscr C^1(\overline\Omega),\;\int udx=0}2\int f u dx-\int |\nabla u|^2 d\mu+\int d\mu+\delta |\Omega|\\
=&\Ene(\mu)+\delta|\Omega|.
\end{align*}
In the former case we have
\begin{align*}
\Ene(\mu+\delta)&=2\int f u_\delta dx-\int |\nabla u_\delta|^2 (d\mu+\delta dx)+\int (d\mu+\delta dx)\\
&\leq \sup_{u\in \mathscr C^1(\overline\Omega),\;\int udx=0}2\int f u dx-\int |\nabla u|^2 d\mu+\int d\mu+\delta |\Omega|\\
=&\Ene(\mu)+\delta|\Omega|.
\end{align*}
Thus in both cases we have
\begin{equation}\label{otherhalf}
\limsup_{\delta\to 0^+}\Ene(\mu+\delta)\leq \Ene(\mu).
\end{equation}
Equations \eqref{halfclaim} and \eqref{otherhalf} prove \eqref{claimcontinuity}.

We are now ready to finish the proof of \eqref{technicallemmaclaim2}. To this aim, let us recall that, if $g$ is a continuous function on a topological space $X$ and $Y$ is a dense subset of $X$, then $\sup_{x\in Y}g(x)=\sup_{x\in X}g(x)$. Notice also that, for $\mu\in L^\infty_+(\Omega)$ and any $\delta>0$ the function $u\mapsto 2\int fu dx-\int (\mu+\delta)|\nabla u|^2dx$ is continuous on $X:=H^1(\Omega)\cap\{u:\int_\Omega u dx=0\}$, while $Y:=\{u\in \mathscr C^1(\overline{\Omega}),\;\int_\Omega udx=0\}$ is dense in $X.$ Hence we have 
$$\sup_{u\in Y}\mathcal D(\mu+\delta,u)=\sup_{u\in X}2\int fu dx-\int (\mu+\delta)|\nabla u|^2 dx.$$
 Using this last fact and \eqref{claimcontinuity}, for any $\mu\in L^\infty_+(\Omega)$, we have the following.
\begin{align*}
\Ene(\mu)&=\lim_\IndexM \sup_{u\in \mathscr C^1(\overline{\Omega}),\;\int_\Omega udx=0}\mathcal D(\mu+\delta_{\IndexM},u)+\int_\Omega (\mu+\delta_{\IndexM})dx\\
&=\lim_\IndexM \sup_{u\in H^1(\Omega),\;\int_\Omega udx=0}\mathcal D(\mu+\delta_{\IndexM},u)+\int_\Omega (\mu+\delta_{\IndexM})dx\\
&\geq\lim_\IndexM \sup_{u\in \W_\IndexM}\mathcal D(\mu+\delta_{\IndexM},u)+\int_\Omega \mu dx\\
&=\sup_\IndexM \sup_{u\in \W_\IndexM}\mathcal D_\IndexM(\mu,u)+\int_\Omega \mu dx.
\end{align*} 
Lemma \ref{continuitylemma} allows us to select $\muN\in \PMeaN^+$ such that $\muN\to \mu$ almost everywhere in $\Omega$ and 
\begin{equation*}
\sup_{u\in \W_\IndexM}\mathcal D_\IndexM(\mu,u)+\int_\Omega \mu dx= \lim_\IndexN \sup_{u\in \W_\IndexM}\mathcal D_\IndexM(\muN,u)+\int_\Omega \muN dx=\lim_\IndexN \EneNM(\muN).
\end{equation*}
Now we observe that, since we are defining the spaces $\PMeaN^+$ by uniform refinements,
$$
\EneNM(\muN)=\begin{cases}
\Ene_{s,\IndexM}(\muN)& \text{ if }s\geq \IndexN\\
+\infty& \text{ otherwise}
\end{cases}\,,
$$
then $\EneNM(\muN)=\inf_s\Ene_{s,\IndexM}(\muN)$ for any $\IndexN\in \N$. Hence
\begin{align*}
\Ene(\mu)&\geq \sup_\IndexM \lim_\IndexN\inf_s\Ene_{s,\IndexM}(\muN) \geq \sup_\IndexM\lim_\IndexN \scm \inf_s\Ene_{s,\IndexM}(\muN)\\
&\geq \sup_\IndexM\liminf_\IndexN (\scm \inf_s\Ene_{s,\IndexM})(\muN)=\sup_\IndexM\liminf_\IndexN \widetilde{\Ene}_\IndexM(\muN)\geq \sup_\IndexM\widetilde{\Ene}_\IndexM(\mu)\\
&=\Ene_\infty(\mu).
\end{align*}
This concludes the proof of \eqref{technicallemmaclaim2}.
\end{proof}

We are now ready to conclude the proof of Theorem \ref{thm:mainresult1}.
\begin{proof}[of Theorem \ref{thm:mainresult1}]
Let us pick $\muN\in \argmin \EneN$ for any $h\in \N.$ Due to Proposition \ref{prop:gammaconvergence} there exists a weak$^*$ convergent subsequence and a limit point $\bar \mu$ such that $\bar\mu\in \argmin \Ene_\infty.$ Recall that $\mu^*\in L^\infty_+(\Omega)$ and notice that
\begin{align*}
\Ene(\mu^*)&\leq \Ene(\bar\mu)&\text{ by the }\Ene\text{-optimality of }\mu^*\\
\Ene(\bar\mu)&\leq \Ene_\infty(\bar\mu)&\text{ by \eqref{technicallemmaclaim1}}\\
\Ene_\infty(\bar\mu)&\leq\Ene_\infty(\mu^*)&\text{ by the }\Ene_\infty\text{-optimality of }\bar\mu\\ 
\Ene_\infty(\mu^*)&\leq \Ene(\mu^*)&\text{ by \eqref{technicallemmaclaim2}}.
\end{align*}
Thus
$$\Ene(\mu^*)\leq \Ene(\bar\mu)\leq \Ene_\infty(\bar\mu)\leq\Ene_\infty(\mu^*)\leq \Ene(\mu^*).$$
Hence the above inequalities are indeed equalities and $\bar\mu$ is a global minimizer of $\Ene.$

Since the unique global minimizer of $\Ene$ is $\mu^*$ then $\bar\mu=\mu^*$ and the whole sequence $\muN^*$ converges to $\mu^*$ in the weak$^*$ topology of measures.
\end{proof}

\section{Well-posedness and conditioning of finite dimensional minimum problems}
\label{sec:conditioning}
We have seen that the sequence of minimum problems $\bmuh^*\in \argmin_{\Rpn} E_{\IndexN,\IndexN}$ converges to the optimal transport density, it is quite natural to ask whether each of this problems is well-posed and well-conditioned. We briefly discuss these questions under two different sets of assumptions, namely Hypothesis \ref{hyp:H1} and Hypothesis \ref{hyp:H2} below. 
We will look at both  $\DEneN:=E_{\IndexN,\IndexN}$ and $\FEneN:=F_{\IndexN,\IndexN}$ defined in \eqref{eq:Edef} and lines below. 
\begin{hypothesis}\label{hyp:H1}
There exists $\bmuh^*\in \argmin_{\Rpn} \DEneN$ such that
\begin{equation}\label{eq:H1}
\left( \bmuh^*+\Ker \Hess \DEneN(\bmu^*_h)\right)\cap \Rpn =\{\bmuh^*\}.
\end{equation}
\end{hypothesis}
\begin{proposition}[Well-posedness under Hypothesis \ref{hyp:H1}]\label{prop:wellpos}
If Hypothesis \ref{hyp:H1} holds true, then the functional $\DEneN$ admits the unique minimizer $\bmuh^*$. 
\end{proposition}
\begin{proof}
We first prove that there exists $\lambda>0$ such that, for any $\nu\in \Rpn$, we have
\begin{equation}\label{localconvexity}
\inf_{\bs\nu\in \Rpn}(\boldsymbol\nu-\bmuh^*)^t\Hess \DEneN(\bmuh^*)(\boldsymbol\nu-\bmuh^*)\geq \lambda \|\boldsymbol\nu-\bmuh^*\|^2.
\end{equation}
For, let us consider the eigendecomposition of the symmetric positive semidefinite matrix $\Hess \DEneN(\bmuh^*)$. More precisely we pick an orthonormal basis $\{\bs\theta^1,\dots,\bs\theta^{\Ntdens}\}$ of $\Rpn$, where $K:= \Ker \Hess \DEneN(\bmuh^*)=\Span\{\bs\theta^{R+1},\dots\bs\theta^{\Ntdens}\}$, $K^\perp=\Span\{\bs\theta^{1},\dots\bs\theta^{R}\}$, and, for suitable $\lambda_1,\dots\lambda_R>0$, we have 
$$(\comp{\bs\theta}{i})^t \Hess \DEneN(\bmuh^*)\bs\theta_j=\begin{cases}
\lambda_i\delta_{i,j},&\text{ if }\max\{i,j\}\leq R\\
0,&\text{ otherwise}
\end{cases}.$$
Therefore we have  
\begin{equation}\label{eq:minimalsingularvalue}
(\boldsymbol\nu-\bmuh^*)^t\Hess \DEneN(\bmuh^*)(\boldsymbol\nu-\bmuh^*)=\sum_{i=1}^R\lambda_i|\langle\boldsymbol\nu-\bmuh^*;\bs\theta^i\rangle|^2\geq \min_{i\leq R}\lambda_i\|\pi_{K^\perp} (\boldsymbol\nu-\bmuh^*)\|^2,
\end{equation}
 
where we denoted by $\pi_{K^\perp}$ the orthogonal projection onto $K^\perp.$ We claim that there exists $C>0$ such that 
\begin{equation}\label{eq:geometricclaim}
\sqrt{C}\leq \inf_{\bmu\geq -\bmuh^*,\;\|\bmu\|=1}\|\pi_{K^\perp}\bmu\|= \inf_{\bmu\geq -\bmuh^*}\frac{\|\pi_{K^\perp}\bmu\|}{\|\bmu\|}=\inf_{\bs\nu\in \Rpn}\frac{\|\pi_{K^\perp} (\boldsymbol\nu-\bmuh^*)\|}{\|\boldsymbol\nu-\bmuh^*\|}.
\end{equation}
The two equalities are quite obvious, while the inequality can be proven by contradiction using the hypothesis \eqref{eq:H1} and the fundamental fact that we are working in a finite dimensional space. Using \eqref{eq:minimalsingularvalue} and \eqref{eq:geometricclaim} we can conclude that \eqref{localconvexity} holds true with $\lambda:=C \min_{i\leq R}\lambda_i.$

Now we want to conclude that $\bmuh^*$ is the unique minimizer of $\DEneN$ in $\Rpn$. Assume that we can pick $\bs\nu_h^*\in \Rpn$ with $\DEneN(\bs\nu_\IndexN^*)=\DEneN(\bmuh^*),$ and define $g:[0,1]\rightarrow \R$ with $g(t):=\DEneN((1-t)\bmuh^*+t\bs\nu_\IndexN^*).$ On one hand, since $\DEneN$ is convex, so it is $g$. On the other hand, since $\DEneN(\bs\nu_\IndexN^*)=\DEneN(\bmuh^*)\geq g(t)$, $g$ needs to be constant. However we also have 
$$g''(0)=(\boldsymbol\nu_\IndexN^*-\bmuh^*)^t\Hess \DEneN(\bmuh^*)(\boldsymbol\nu_\IndexN^*-\bmu^*_\IndexN)\geq \lambda \|\boldsymbol\nu_\IndexN^*-\bmuh^*\|^2>0.$$
This is a contraddiction, so $\bmuh^*$ is the unique minimizer of $\DEneN$ in $\Rpn$.
\end{proof}

\begin{hypothesis}\label{hyp:H2}
There exists $\bmuh^*\in \argmin_{\Rpn} \DEneN$ and $\lambda>0$ such that,
\begin{align}
&[\Hess \DEneN(\bmuh^*)]_{\llcorner \{\comp{\bs e}{i}:\comp{\bmuh^*}{i}>0\}}\succeq \lambda \mathbb I,\label{eq:H21}\\
&\partial_i \DEneN(\bmuh^*)\neq 0,\;\;\forall i: \comp{\bmuh^*}{i}=0.\label{eq:H22}
\end{align}
\end{hypothesis}
\begin{proposition}[Well-posedness and well-conditioning under Hypothesis \ref{hyp:H2}]\label{prop:wellcond}
If Hypothesis \ref{hyp:H2} holds true, then the functional $\DEneN$ admits a unique minimizer $\bmuh^*$. Moreover, for any
\begin{equation}\label{Lambdabound}
0<\Lambda<\min\left\{4\lambda \min_{i:\comp{\bmuh^*}{i}\neq 0}\comp{\bmuh^*}{i}\;,\;2\min_{i:\bmuhi^*= 0}\partial_i \DEneN(\bmuh^*)  \right\},
\end{equation}
 there exists $R>0$ such that, for any $\bsigmah^*$ such that $\Sq{(\bsigmah^*)}=\bmuh^*$ and for any $\bmu:=\Sq{\bsigma}$  with $\|\bsigma-\bsigmah^*\|<R,$ we have
\begin{equation}\label{ConditionNumber}
\frac{\|\bmu-\bmuh^*\|}{\|\bmuh^*\|}\leq \frac {16} \Lambda\max_{i: \comp{\bsigma^*_h}{i}\neq 0  }|\partial_i \DEneN(\bmu)|.
\end{equation}  
\end{proposition}
\begin{proof}
Using equations \eqref{eqn:secondderivativeF}, \eqref{eq:H21}, and \eqref{eq:H22}, we obtain
\begin{equation}
\Hess \FEneN(\bsigmah^*)\succeq \min\left\{4\lambda \min_{i:\comp{\bmuh^*}{i}\neq 0}\comp{\bmuh^*}{i}\;,\;2\min_{i:\comp{\bmuh^*}{i}= 0}\partial_i \DEneN(\bmuh^*)  \right\} \mathbb I.
\end{equation} 
Thus we can pick $0<R<\|\bsigmah^*\|_\infty$ such that, for any $\bsigma\in \R^{\Ntdens}$, with $|\bsigmah^*-\bsigma|<R$, we have
\begin{equation}
\Hess \FEneN(\bsigma)\succeq \Lambda \mathbb I.
\end{equation}
Using the first order Taylor expansion of $\nabla \FEneN(\bsigma)$ centered at $\bsigmah^*$, we can write, for any $\bsigma$ as above,
\begin{align}
&\|\bsigma-\bsigmah^*\|=\left\|[\Hess \FEneN(\bs\eta)]^{-1}\nabla \FEneN(\bsigma)  \right\|\leq \frac 1 \Lambda \|\nabla \FEneN(\bsigma)\|\notag\\
\leq& \frac{2\|\bsigma\|_\infty}{\Lambda}\max_{i:\bsigmai\neq 0}|\partial_i \DEneN(\bmu)|,\label{sigmaestimate}
\end{align}
where $\bs\eta\in[\bsigma,\bsigmah^*].$ Notice that
$$\|\bmu-\bmuh^*\|^2\leq \|\bsigma+\bsigmah^*\|_\infty^2 \|\bsigma-\bsigmah^*\|^2 \leq 4\max\left\{  \|\bsigma\|_\infty^2\;,\;\|\bsigmah^*\|_\infty^2   \right\} \|\bsigma-\bsigmah^*\|^2,$$
thus
\begin{equation}\label{muestimate}
\|\bmu-\bmuh^*\|\leq  2(R+\|\bsigmah^*\|_\infty)\|\bsigma-\bsigmah^*\|\leq 4\|\bsigmah^*\|_\infty\|\bsigma-\bsigmah^*\| .
\end{equation}
Therefore 
\begin{align*}
&\|\bmu-\bmuh^*\|\leq \frac{8 \|\bsigmah^*\|_\infty\|\bsigma\|_\infty}{\Lambda}\max_{i:\bsigmai\neq 0}|\partial_i \DEneN(\bmu)|\\
\leq  &\frac{16 \|\bsigmah^*\|_\infty^2}{\Lambda}\max_{i:\bsigmai\neq 0}|\partial_i \DEneN(\bmu)|\leq \frac{16 \|\bmuh^*\|}{\Lambda}\max_{i:\bsigmai\neq 0}|\partial_i \DEneN(\bmu)|.
\end{align*}
\end{proof}
Note that, as a byproduct of this proof, we can state the following.
\begin{corollary}\label{cor:H2alternative}
Hypothesis \eqref{hyp:H2} holds true at a given $\bmuh^*\in \Rpn$ for some $\lambda>0$ if and only if $\Hess \FEneN(\bsigmah^*)$ is strictly positive definite at any $\bsigmah^*$ such that $\Sq{\bsigmah^*}=\bmuh^*.$
\end{corollary}

\section{Existence, uniqueness and long time behavior of minimizing flows for $\DEneN$}
\label{sec:continuous_time}
In order to minimize the functional $\DEneN$ we consider two different gradient flows. The more natural choice is to design a gradient flow for $\DEneN$ including explicitly the non-negativity constraint $\bmu\in\PMeaN^+$. It turns out that the latter constraint makes the flow trajectories possibly non-smooth. The second gradient flow we propose is applied directly to $\FEneN$ yelidng smooth trajectories and ideal characteristics for the design of a more efficient numerical scheme. For this reason this will be our method of choice described in the next section.
\subsection{The Euclidean gradient flow of $\DEneN$}\label{subsec:gradflow}
Consider the following:
\begin{equation}\label{eqn:h-gradientflow}
\begin{cases}
\frac d{dt}\bs \mu=-\partial^\circ \DEneN(\bs \mu)& t\in[0,+\infty[\\
\bs \mu(0)=\bs \mu^0\in \Rpn& 
\end{cases}\;,
\end{equation}
where $\partial^\circ \DEneN(\bs \mu)$ is the (unique) element of minimal norm of the convex subdifferential of $\DEneN$ at $\bmu.$ We show next that every trajectory of this flow converges to a minimizer of $\DEneN$, as $t\to+\infty$, regardless of the choice of $\bs \mu^0\in \Rpn .$

\begin{theorem}[Existence,uniqueness and long time behavior of the flow]\label{thm:existenceuniqueness}
Given $\bmuh^0\in \Rpn$ with $ \partial^\circ \DEneN(\bmu^0)\neq 0$, there exists a unique locally absolutely continuous and a.e. diffferentiable curve $[0,+\infty[\ni t\rightarrow \bmuh(t;\bmuh^0)\in \Rpn $, with $\bmuh(0;\bmuh^0)=\bmuh^0$, that is an energy solution of \eqref{eqn:h-gradientflow}, i.e.,
\begin{equation}\label{eqn:energyequality}
\DEneN(\bmuh(t;\bmuh^0))=\DEneN(\bmuh^0)-\int_0^t\|\partial^\circ \DEneN\|^2(\bmuh(s;\bmuh^0))\,ds,\;\forall t>0.
\end{equation}
In particular $\bmuh'(t;\bmuh^0)=-\partial^\circ \DEneN(\bmuh(t;\bmuh^0))$ for almost any $t>0.$

Moreover, for any $\bmu^0\in \PMeaN^+ $ with $ \partial^\circ \DEneN(\bmu^0)\neq 0$, we have
\begin{equation}\label{LTB}
\lim_{t\to +\infty}\bmuh(t;\bmuh^0)\in \argmin_{\Rpn}  \DEneN(\bmuh^*).
\end{equation} 
\end{theorem}
\begin{proof}
The proof of existence and uniqueness of the flow (together with the local absolute continuity of the trajectories) relies on the application of \cite[Th. 4.0.4]{AmGiSa00} and \cite[Prop. 1.4.1]{AmGiSa00}. We first need to verify that the assumptions of \cite[Th. 4.0.4]{AmGiSa00} are satisfied, i.e.,  
Namely:
\begin{enumerate}[(i)]
\item $(\Rpn , d)$, with $d(\bmu,\bs \nu):=\|\bmu-\bs \nu\|_2$, is a complete metric space.
\item $\DEneN$ is proper, lower semicontinuous, and bounded from below.
\item For any $\bmu,\bmu^0,\bmu^1\in \Rpn$ there exists a curve $\gamma:[0,1]\rightarrow \Rpn $, with $\gamma(0)=\bmu^0,$ $\gamma(1)=\bmu^1$, such that, for any $\tau\in(0+\infty)$,
 
\begin{equation}\label{0conv}
\Phi(\tau,\bmu;\gamma(t))\leq (1-t) \Phi(\tau,\bmu;\bmu^0)+t \Phi(\tau,\bmu;\bmu^1)-\frac 1{2\tau}t(1-t)d^2(\bmu,\bmu^1), \;\forall t\in (0,1),
\end{equation}
 
where we use the notation
$$ \Phi(\bs u,\tau;\bs v):=\DEneN(\bs v)+\frac{d^2(\bs u,\bs v)}{2\tau}.$$
\end{enumerate}
The property (\emph{i}) hods because the positive cone $\Rpn$ is a closed set of the complete metric space $(\R^{\Ntdens},d)$ and thus it is complete. Property (\emph{ii}) follows directly from Proposition \ref{prop:differentiability} and the fact that $\DEneN$ is non-negative. Finally \eqref{0conv} can be verified on line $\gamma(t):=(1-t)\bmu^0+t\bmu^1$ employing the convexity of $\DEneN$ and the $2$-convexity of $d^2$ along straight lines, namely
$$d^2(\bmu,\gamma(t))\leq (1-t)d^2(\bmu,\bmu^0)+td^2(\bmu,\bmu^1)-t(1-t)d^2(\bmu^0,\bmu^1).$$
By \cite[Th. 4.0.4 (\emph{ii})]{AmGiSa00}, there exists a locally Lipschitz curve of maximal slope for $\DEneN$ with respect to the metric slope of $\DEneN.$ In our rather simple setting the metric slope corresponds to the norm of the minimal subdifferential (i.e., $\|\partial^\circ \DEneN\|_2$) and the definition of a curve of maximal slope simplifies to
\begin{equation}\label{eqn:curveofmaximalslope}
|\bmuh'(t;\bmuh^0)|^2=\|\partial^\circ \DEneN\|_2^2(\bmuh(t;\bmuh^0))\\=-\frac{d}{dt}\left(\DEneN(\bmuh(t;\bmuh^0))   \right),\; \text{ a.e. in }[0,+\infty[.
\end{equation}
Since $\R^{\Ntdens}$ is a Hilbert space, \eqref{eqn:curveofmaximalslope} implies (see \cite[Prop. 1.4.1 and Cor. 1.4.2]{AmGiSa00}) that $\bmuh'(t;\bmuh^0)=-\partial^\circ \DEneN(\bmuh(t;\bmuh^0))$ for almost any $t>0.$

We are left to prove that $\lim_{t\to +\infty} \bmuh(t;\bmuh^0)$ exists and is a minimizer of $\DEneN$ regardless of the choice of the initial data. To this aim we apply \cite[Cor. 4.0.6]{AmGiSa00} and, using again the fact that the metric slope of $\DEneN$ at $\bmuh$ is equal to $\|\partial^\circ \DEneN(\bmuh)\|_2$, we obtain
\begin{align}
&\DEneN(\bmuh(t;\bmuh^0))-\min_{\bmuh\in \Rpn}\DEneN(\bmuh)\leq \min_{\bmuh\in \argmin \DEneN}\frac{\|\bmuh-\bmuh^0\|_2^2}{2t}\to 0,\label{eqn:bycorollary1}\\
&\|\bmuh'(t;\bmuh^0)\|_2=\|\partial^\circ \DEneN(\bmuh(t;\bmuh^0))\|_2\leq \min_{\bmuh\in \argmin \DEneN}\frac{\|\bmuh-\bmuh^0\|_2}{t}\to 0.\label{eqn:bycorollary2}
\end{align}
Note that \eqref{eqn:bycorollary1} implies that, if the curve $\bmuh(t;\bmuh^0)$ admits a limit point $\bmuh\in \Rpn$ for $t\to+\infty$, then $\bmuh\in \argmin \DEneN.$ Such a limit indeed exists due to \cref{eqn:bycorollary2}.
\end{proof}

\subsection{The Euclidean gradient flow of $\FEneN$}
We recall here the definition of the coordinate square map $\cdot^{\bs 2}:\R^{\Ntdens}\rightarrow\R^{\Ntdens}$ introduced in \eqref{eqn:coordinatesquaremap} of Sec. \ref{subsec:FEMdis}:
$$\bs \sigma^{\bs 2}:=(\comp{\bsigma}{1}^2,\comp{\bsigma}{2}^2,\dots,\comp{\bsigma}{\Ntdens}^2),$$
and the $\FEneN$ functional:
$$\FEneN(\bs \sigma):=\DEneN(\bs{\sigma^2}).$$
The minimization of $\FEneN$ on $\R^{\Ntdens}$ has at least two advantages with respect to the minimization of $\DEneN$ over $\Rpn,$ namely it is an unconstrained optimization problem and the objective function is globally real analytic. On the other hand a number of potetial drawbacks can be listed. The objective function is not convex anymore and, by the symmetry of $\Sq{}$, uniqueness of the minimizer is lost.

We attack the problem of minimizing $\FEneN$ considering the following gradient flow:
\begin{equation}\label{eq:GhGradientflow}
\begin{cases}
\frac d{dt}\bsigmah=-\nabla \FEneN(\bsigmah),& \forall t>0\\
\bsigmah(0)=\bsigmah^0
\end{cases}\;.
\end{equation} 
We first prove existence, uniqueness, and regularity of the solutions of \eqref{eq:GhGradientflow}, then we prove that all the trajectories converge to a global minimizer of $\FEneN$ in Theorem \ref{thm:GhLongTime}. 
\begin{theorem}[Global existence, uniqueness, and regularity for \eqref{eq:GhGradientflow}]\label{thm:GhExUn}
Let $\bsigmah^0\in \R^{\Ntdens}.$ There exists a unique real-analytic classical solution $]0,+\infty[\ni t\mapsto \bsigmah(t;\bsigmah^0)\in \R^{\Ntdens}$ of \eqref{eq:GhGradientflow}.
\end{theorem}
\begin{proof}
For any $c\in \R$ and any $\epsilon>0$ we consider the compact sets 
$$S_c:=\{\bsigma\in \R^N:\FEneN(\bsigma)\leq c^2\},\;\;\;S_c^\epsilon:=\cup_{\bsigma \in S_c}B(\bsigma,\epsilon].$$
Note that $\max_{\bsigma\in S_c}\|\bsigma\|\leq |c|$ and $\max_{\bsigma\in S_c^\epsilon}\|\bsigma\|\leq |c|+\epsilon$, since $\FEneN\geq \|\cdot\|^2$ by definition.
Let us pick $c\in \R$ such that $\bsigmah^0\in S_c.$ It is clear that
\begin{align*}
&\max_{\|\bsigma-\bsigmah^0\|\leq \epsilon}\|\nabla \FEneN(\bsigma)\|\leq\max_{\bsigma \in S_c^\epsilon} \|\nabla \FEneN(\bsigma)\|\\
\leq& 2\max_{\bsigma\in S_c^\epsilon}\left(\|\bsigma\|_2 \max_i \left| |T_i^\IndexN|-(\bs f_\IndexN)^t(A^{(\IndexN,\IndexN)}(\Sq{\bsigma}))^{-1}A^{(i,\IndexN,\IndexN)}(A^{(\IndexN,\IndexN)}(\Sq{\bsigma}))^{-1}\bs f_\IndexN \right|\right)\\
\leq & 2\max_{\bsigma\in S_c^\epsilon}\|\bsigma\|_2\max_i \left(|T_i^\IndexN|+\frac{C^2\lambda_{max}(A^{(i,\IndexN,\IndexN)})}{\delta_\IndexN^2}\|\bs f_\IndexN\|^2\right) \\
\leq &2\max_{\bsigma\in S_c^\epsilon}\|\bsigma\|_2\max_i \left(|T_i^\IndexN|+\frac{C^2\lambda_{max}(A^{(i,\IndexN,\IndexN)})}{\delta_\IndexN^2}\|\bs f_\IndexN\|_2^2\right)\\
\leq & (|c|+\epsilon)\max_i \left(|T_i^h|+\frac{C^2\lambda_{max}(A^{(i,\IndexN,\IndexN)})}{\delta_\IndexN^2}\|\bs f_\IndexN\|_2^2\right)\\
=:&(|c|+\epsilon)K_\IndexN,
\end{align*}
Where we denoted by $C$ the Poincare' constant of $\Omega$ and we used that $A^{(\IndexN,\IndexN)}(\Sq{\bsigma})$ is symmetric positive definite and its minimal eigenvalue is at least $\delta_\IndexN/C$ for any $\bmu\in \R_+^N.$ Moreover $A^{(i,\IndexN,\IndexN)}$ is symmetric and positive semi-definite for any $i=\,2,\dots, N.$

Since $\nabla \FEneN$ is a real analytic function on $\R^{\Ntdens}$, it is in particular a uniformly Lipschitz function on $S_c^\epsilon\supseteq B]\bsigmah^0,\epsilon].$ 
By the classical Picard-Lindel\"of Theorem, there exists a unique solution $\bsigmah^{(0)}(t;\bsigmah^0)$ of \eqref{eq:GhGradientflow} in the interval $[-\tau,\tau]$ with
$$\tau:= \frac{\epsilon}{(|c|+\epsilon)K_\IndexN}.$$
Let us introduce the sequence
$$t_j:=j \frac 2 3\tau,\;\;\;j=0,1,\dots$$
and the notation $\bsigmah^{j+1}:=\bsigmah^{(j)}(t_{j+1};\bsigmah^s)$ 
Now notice that 
\begin{align*}
\FEneN(\bsigmah^1)=&\FEneN(\bsigmah^0)+\int_0^{t_1}\langle\nabla \FEneN(\bsigmah^{(0)}(s;\bsigmah^0)),\frac d{dt}\bsigmah^{(0)}(s;\bsigmah^0))  \rangle ds\\
=&\FEneN(\bsigmah^0)-\int_0^{t_1}\|\nabla \FEneN(\bsigmah^{(0)}(s;\bsigmah^0))\|^2 ds\leq c^2.
\end{align*}
Therefore $\bsigmah^1\in S_c.$ Thus we can repeat the argument to construct a curve $\bsigmah^1(\cdot;\bsigmah^1):[t^1-\tau,t^1+\tau]\rightarrow \R^{\Ntdens}$ which is the uique solution of the Cauchy problem 
\begin{equation*}
\begin{cases}
\dfrac d{dt}\bsigma=-\nabla \FEneN(\bsigma)& t>t^1\\
\bsigma(t^1)=\bsigmah^1
\end{cases}.
\end{equation*}
Uniqueness implies that $\bsigmah^{(0)}(t;\bsigmah^0)=\bsigmah^{(1)}(t-t^1;\bsigmah^1)$ for any $t\in[-1/3\tau,\tau]$. This provides existence and uniqueness of the solution of \eqref{eq:GhGradientflow} in $[-\tau, \frac 5 3\tau].$ Iterating the same argument we can  construct a solution $\bsigmah(t;\bsigmah^0)$ of \eqref{eq:GhGradientflow} in $[0,+\infty[$ that turns out to be unique. Analyticity of the flow follows from the Cauchy Kovalevskaya Theorem and Proposition \ref{prop:differentiability}.
\end{proof}
\begin{theorem}[Long-time asymptotics of \eqref{eq:GhGradientflow}]\label{thm:GhLongTime}
Let $\bsigmah^0\in \R^{\Ntdens}$ be such that $\nabla \FEneN(\bsigmah^0)\neq 0$. Then there exists $\bmuh^*\in \argmin_{\Rpn}\DEneN$ such that
\begin{equation}
\lim_{t\to +\infty}(\bsigmah(t;\bsigmah^0))^{\bs 2}=\bmuh^*.
\end{equation}
\end{theorem}
\begin{proof}
For a given $\bsigmah^0$ we denote by $\bsigmah(t):=\bsigmah(t;\bsigmah^0)$ the curve that solves \eqref{eq:GhGradientflow}. By the Fundamental Theorem of Integral Calculus we can write, for any $t>0$,
\begin{equation*}
\min \FEneN-c^2\leq \FEneN(\bsigmah(t))-\FEneN(\bsigmah^0)\\
=\int_0^t\langle \nabla \FEneN(\bsigmah(s));\bsigmah'(s)\rangle ds= -\int_0^t\|\bsigmah'(s)\|^2 ds.
\end{equation*} 
Thus we have
$$\int_0^{+\infty}\|\bsigmah'(t)\|^2 dt\leq c^2-\min \FEneN<+\infty.$$
In particular this implies that for any $\epsilon>0$ there exists $T_\epsilon>0$ such that $\int_{T_\epsilon}^{+\infty}\|\bsigmah'(t)\|^2 dt\leq\epsilon$, and that any sequence $\{\bsigmah(t_k)\}_{k\in \N}$, with $t_k\to+\infty$, is a Cauchy sequence. Hence we can define 
$$\hatbsigmah:=\lim_{t\to +\infty}\bsigmah(t)=\bsigmah^0+\int_0^{+\infty} \bsigmah'(t)dt.$$
Now we want to show that $\lim_{t\to +\infty}\bsigmah'(t)=0$. Indeed thanks to the computations above (see the Proof of Theorem \ref{thm:GhExUn}) we can write
\begin{align*}
\sup_{t>0}\|\bsigmah'\|&= \max_{\bsigma\in S_c}\|\nabla \FEneN(\bsigma)\|\leq |c|K_\IndexN\\
\sup_{t>0}\|\bsigmah''\|^2&= \sup_{t>0}\bsigmah'(t)^t\Hess \FEneN(\bsigmah(t)) \bsigmah'(t)\leq R^2\sup_{t>0}\|\bsigmah'(t)\|^2, 
\end{align*}
where 
$$R^2=\sup_{\bsigma\in S_c}\rho(\Hess \FEneN(\bsigma) )<\infty,$$
and $\rho(\cdot)$ denotes the spectral radius. It follows that $\sup_{t>0}\|\bsigmah''(t)\|\leq R|c|K_\IndexN.$ Let us assume by contradiction that $\limsup_{t\to +\infty}\|\bsigmah'(t)\|^2=l>0.$ Then we can pick $\{t_k\}$ such that $\inf_k|t_k-t_{k+1}|=: \alpha>0$ and $s_k:=\|\bsigmah'(t_k)\|^2\to l,\;\;\text{ as }k\to +\infty.$ Notice that for any $t>0$ such that 
$$|t-t_k|\leq \inf_j\{\sqrt{s_j}/(4RcK_\IndexN),\alpha\},\;k=1,2,\dots ,$$
then
\begin{align*}
\|\bsigmah'(t)\|^2=&\|\bsigmah'(t_k)+(t-t_k)\bsigmah''(\xi)\|^2\\
=& \|\bsigmah'(t_k)\|^2+2(t-t_k)\langle \bsigmah''(\xi),\bsigmah'(t_k)\rangle+(t-t_k)^2\|\bsigmah''(\xi)\|^2\\
\geq& s_k-2|t-t_k|\sup_{s>0}\|\bsigmah''(s)\|\sqrt{s_k}\\
\geq& s_k-2|t-t_k|R|c|K_\IndexN\sqrt{s_k}\geq \frac{s_k}2\;.
\end{align*}
Therefore 
\begin{align*}
+\infty&>\int_0^{+\infty}\|\bsigmah'(s)\|^2 ds\geq \sum_{k=1}^{+\infty}\int_{t_k-\alpha}^{t_k+\alpha}\|\bsigmah'(s)\|^2 ds\geq \frac\alpha 2 \sum_{k=1}^{+\infty}s_k\;,
\end{align*}
contradicting the hypothesis (i.e., $\lim_k s_k=l>0$). Thus $\limsup_{t\to +\infty} \|\bsigmah'(t)\|=0$ and, since $\|\bsigmah'(t)\|\geq 0$ $\forall t>0$, $\lim_{t\to +\infty}\|\bsigmah'(t)\|=0.$ In particular, $\hatbsigmah$ is a stationary point for $\FEneN.$

We observe that, by the Karush-Kuhn-Tucker (KKT) Theorem and the convexity of $\DEneN$ the following conditions are necessary and sufficient for $\bmuh^*\in \Rpn$ to be a minimizer of $\DEneN$ on $\Rpn$
\begin{equation}\label{KKT}
\begin{cases}
\partial_i \DEneN(\bmuh^*)=0& \forall i: \bmuhi^*>0\\
\partial_i \DEneN(\bmuh^*)\geq 0& \forall i: \bmuhi^*=0
\end{cases}\;.
\end{equation}
Since
\begin{equation}\label{eq:recallgradient}
0=\nabla \FEneN(\hatbsigmah)=\comp{2(\hatbsigmah)}{1}\partial_1 \DEneN(\Sq{\hatbsigmah}),\dots,2(\hatbsigmah)_{\Ntdens}\partial_{\Ntdens} \DEneN(\Sq{\hatbsigmah})^{}),
\end{equation}
the first equation in \eqref{KKT} is always satisfied by $\bmuh^*:=\Sq{\hatbsigmah}.$

It is not difficult to prove (using an argument similar to the one used above while proving $\lim_{t\to +\infty}\|\bsigmah'(t)\|=0$) that we can pick a sequence $0<t_k\to +\infty$ such that $\bsigmahi'(t_k)\cdot\bsigmahi(t_k)\leq 0$ for any such $i$. Thus, using the definition of the gardient flow,
\begin{equation*}
\begin{cases}
\bsigmahi'(t_k)\leq 0 &\text{ if } \bsigmahi(t_k)\geq 0\\
\bsigmahi'(t_k)\geq 0 &\text{ if } \bsigmahi(t_k)\leq 0
\end{cases}\;\Longrightarrow\;\\
\begin{cases}
-2\bsigmahi(t_k)\partial_i \DEneN(\bsigmah(t_k))\leq 0 &\text{ if } \bsigmahi(t_k)\geq 0\\
-2\bsigmahi(t_k)\partial_i \DEneN(\bsigmah(t_k))\geq 0 &\text{ if } \bsigmahi(t_k)\leq 0
\end{cases},
\end{equation*}
$\forall k$ and for any $i$ such that $\hatbsigmahi=0.$ Hence $\partial_i \DEneN(\bsigmah(t_k))\geq 0$ $\forall k$ and for any such $i$, i.e., also the second equation of \eqref{KKT} holds true.
\end{proof}

Theorem \ref{thm:GhLongTime} is a qualitative result, but we can give a sharp quantitative estimate of the rate of convergence exploiting the real analyticity of $\FEneN$. The key element here is the \emph{\Loja{} Theorem} (see \cite{Lo63}). \Loja{}'s. Precisely, if $D\subseteq \R^N$ is an open set and $f:D\rightarrow\R$ is real analytic, then, for any $x\in D$, there exists $R>0$, a real number $\vartheta\in(0,1/2]$ (called the \emph{\Loja{} exponent} of $f$ at $x$), and $L>0$ such that 
\begin{equation}
\label{eq:Loja}
|f(x)-f(y)|^{1-\vartheta}\leq L\|\nabla f(y)\|,\;\;\forall y\in B(x,R).
\end{equation}
The above is the so-called \Loja{} Inequality, which has already been used whithin similar contexts for proving the convergence of finite and infinite dimensional gradient flows, see e.g.  \cite{BoDaLe06}, and for estimating the rate of convergence of the flow to a critical point of the objective, see \cite{Ch03,HaMa19} and references therein. Inequality \eqref{eq:Loja} has been successfully used also for proving the convergence of the time discretization of gradient flows by e.g. backward Euler schemes, see \cite{MePi10}.  We can specialize this quite classical tool to our setting using the results of \cite{Haraux2001DecayET} as follows.

\begin{theorem}[Rate of convergence to equilibria of \eqref{eq:GhGradientflow}, Thm. 2.2 of \cite{Haraux2001DecayET}]
  \label{thm:rateofconv}
Let $\bsigmah(t):=\bsigmah(t;\bsigmah^0)$ be a trajectory of \eqref{eq:GhGradientflow} and $\bsigmah^*:=\lim_{t\to+\infty}\bsigmah(t)\in \argmin \FEneN$. Then:
\begin{enumerate}[i)]
\item If the \L{}ojasiewicz Inequality \eqref{eq:Loja} holds for $\FEneN$ at $\bsigmah^*$ with $\vartheta<1/2$, then there exists $t^0\geq 0$ such that, for any $t>t^0$, we have
\begin{equation}\label{polynomialspeed}
\|\bsigmah(t)-\bsigmah^*\|\leq \left(\|\bsigmah(t^0)-\bsigmah^*\|^{\frac{2\vartheta-1}{\vartheta}}+\vartheta^{\frac{1-2\vartheta}{\vartheta}}\frac{1-2\vartheta}{L^{\frac{1}{\vartheta}}}(t-t^0)  \right)^{\frac{\vartheta}{2\vartheta-1}}.
\end{equation} 
\item If the \L{}ojasiewicz Inequality \eqref{eq:Loja} holds for $\FEneN$ at $\bsigmah^*$ with $\vartheta=1/2$ (and in particular if Hypothesis \ref{hyp:H2} holds), then there exists $t^0\geq 0$ such that, for any $t>t^0$, we have
\begin{equation}\label{exponentialspeed}
\|\bsigmah(t)-\bsigmah^*\|\leq \|\bsigmah(t^0)-\bsigmah^*\| \exp\left(-\frac{1}{2L^2}(t-t^0)  \right).
\end{equation}  
\end{enumerate}
\end{theorem}

\subsection{Refined conditioning estimates by \L{}ojasiewicz Inequality}
\label{subsec:Cond_Loja}
In this subsection we exploit the existence and regularity of the gradient flow of $\FEneN$ together with the \Loja{} Inequality to derive a conditioning estimate similar to \eqref{ConditionNumber} (obtained in Proposition \ref{prop:wellcond}) valid also when Hypothesis \ref{hyp:H2} does not hold. This result will be used later in the design of the stopping criterion for our main minimization algorithm \ref{alg:backEuler}.

In the next proposition we use the following notation for the Euclidean distance of the point $\bsigma$ from the set minimizers of $\FEneN$:
$$d(\bsigma,\argmin \FEneN)=\min \{\|\bsigma-\bsigmah^*\|,\;\bsigmah^*\in \argmin \FEneN\}.$$
Note that the minimum is well defined because $\argmin \FEneN$ is compact due to the coercivity and lower semicontinuity of $\FEneN$.
\begin{proposition}[Refined conditioning estimates]\label{prop:refinedCond}
For any $h\in \N$ there exists $R>0$, $L>0$, and $\vartheta\in (0,1/2]$, such that, for any $\bsigma$ such that $d(\bsigma,\argmin \FEneN)<R$, there exists $\bsigmah^* \in \argmin \FEneN$ with
\begin{equation}\label{eq:polycond}
  \|\bsigmah-\bsigmah^*\|<
  \frac{L^{\frac{1}{1-\vartheta}}}{\vartheta}
  \|\nabla \FEneN(\bsigmah(t))\|^{\frac{\vartheta}{1-\vartheta}}\,.
\end{equation}
If Hypothesis \ref{hyp:H2} holds, we can take $\vartheta=1/2$, i.e. we have
\begin{equation*}
  \|\bsigmah-\bsigmah^*\|<2L^{2}\|\nabla \FEneN(\bsigmah(t))\|\,.
\end{equation*} 
\end{proposition} 
\begin{proof}
We notice that the constants in the \Loja{} Inequality can be chosen uniformly with respect to $\bsigmah^*$ because of the aforementioned compactness of $\argmin \FEneN.$ We construct the trajectory $t\mapsto \bsigmah(t)$ by solving the gradient flow of $\FEneN$ starting at $\bsigmah$. This is a well-defined unique analytic curve due to Theorem \ref{thm:GhExUn} that has a limit point $\bsigmah^*\in \argmin \FEneN$ due to Theorem \ref{thm:GhLongTime}. Then, using eq. (2.10) found in the proof of Thm. 2.2 of \cite{Haraux2001DecayET}, 
we can write
\begin{equation*}
  \|\bsigmah(t)-\bsigmah^*\|^2\leq
  \frac{L^{\frac{1}{1-\vartheta}}}{\vartheta}
  \|\nabla \FEneN(\bsigmah(t))\|^{\frac{\vartheta}{1-\vartheta}}
\end{equation*}
If Hypothesis \ref{hyp:H2} holds true, then the Hessian of $\FEneN$ at $\bsigmah^*$ is not degenerate. Due to e.g., \cite[Prop. 2.2]{MePi10}, the \L{}ojasiewicz Inequality holds for $\FEneN$ in a suitable neighbourhood of $\bsigmah^*$ with $\vartheta=1/2.$
\end{proof}
\begin{remark}
It is worth pointing out that Proposition \ref{prop:refinedCond} is a slightly stronger result than the condition estimate \eqref{ConditionNumber} that we proved in Section \ref{sec:conditioning}, since the latter depends on Hypothesis \ref{hyp:H2}. On the contrary, Proposition \ref{prop:refinedCond} still leads to the (weaker) estimate \eqref{eq:polycond} even if Hypothesis \ref{hyp:H2} does not hold true. 
\end{remark}

\section{Minimization algorithms from numerical integration of the gradient flows}
\label{sec:discrete_time}

Since the gradient flow of $\EneN$ converges to $\bmuh^*\in \argmin\EneN$ as $t\to +\infty$, we may try to numerically integrate the flow, in order to design an approximation algorithm for $\bmuh^*$. The most straighforward idea is to adapt the forward Euler scheme to our constrained setting by enforcing the positivity constraint at each iteration.  When the feasible set is a convex subset of an Hilbert space, one can modify such a scheme including an orthogonal projection of the Euler update on the feasible set. This algorithm (and its many variants) can be casted within the framework of \emph{proximal algorithms} (see \cite{Ro76}) and is termed \emph{projected gradient descent} or \emph{projected forward Euler scheme}, and its application to our case is described in Algorithm \ref{alg:projEuler}.
\begin{algorithm}
\caption{Projected forward Euler Scheme}
\label{alg:projEuler}
\begin{algorithmic}
\STATE{Input $\bs{\mu}\in \R_+^N$, $\tau>0$, $n_{step}\in \N$, $toll>0$ }
\WHILE{$s<n_{step}$ \AND $err>toll$}
\STATE{$s=s+1$}
\STATE{Compute $\bs{u}=A^{-1}(\bs{\mu}+\Lifth)\bs{f}$}
\STATE{Compute $\bs{v}=(\bs{u}^tA^{(j)}\bs{u}^t)_{j}-1$}
\STATE{Set $v^{(j)}=(v^{(j)})^+$ if $\mu^{(j)}=0$}
\STATE{Set $\bs{\mu}:=(\bs{\mu}+\tau\bs{v})^+$}
\STATE{$err=\|\bs v\|_2$}
\ENDWHILE
\RETURN $\bs{\mu}$
\end{algorithmic}
\end{algorithm}  
The convergence analysis of Algorithm \ref{alg:projEuler} is rather standard. It can be proven that, for sufficiently small and constant time steps, any cluster point of the sequence computed by Algorithm \ref{alg:projEuler} is a stationary point of the objective functional, see for instance \cite{Be99}. Unfortunately, the use of the orthogonal projection in the definition of the iterates implies a non-monotone convergence and may cause few stability issues, expecially if the objective tends to be non-strongly convex.   

For these reasons we pursue a different strategy by working with $\FEneN$, thus avoiding the explicit implementation of the non-negativity constraint. This entails the derivation of a numerical algorithm that mimics De Giorgi's \emph{minimizing movements} which have been introduced for the construction of gradient flows in very general contexts, such as Banach or metric spaces, \cite{AmGiSa00}. This approach applied to the construction of the gradient flow of the functional $\FEneN$ consists in picking an initial point $\bsigmah^0$ and a time step $\tau$ (or a time step sequence $\{\tau^{\iter}\}$), then computing iteratively $\bsigmah^{\iter+1}$ such that
\begin{equation}\bsigmah^{\iter+1}\in \argmin_{\bsigma\in \R^{\Ntdens}} \GEneN(\bsigma;\bsigmah^{\iter},\tau):=\argmin_{\bsigma\in \R^{\Ntdens}}\FEneN(\bsigma)+\frac{\|\bsigma-\bsigmah^{\iter}\|^2}{2\tau}.\label{eq:vbEs}
\end{equation}
Then, one shrinks the size of the time step (or the size of the partition in the case of variable time stepping) to zero and tries to prove that the linear interpolation of the discrete time trajectories is converging to a solution of the gradient flow equation.

Here we are interested in approximating a minimizer of $\FEneN$ rather than computing the trajectory of the flow. Therefore we are not going to shrink the time step to zero, instead we keep iterating eq. \eqref{eq:vbEs} and consider $\lim_\iter \bsigmah^{\iter+1}.$ The derived optimization technique is called \emph{variational backward Euler method} \cite{MePi10} and is still casted by some authors (see e.g., \cite{Gu91}) in the framework of proximal algorithms. 

Note that equation \eqref{eq:vbEs} does  not define yet a numerical algorithm, since one needs to choose a method for solving the optimization problem and a stopping criterion for $\iter$, possibly with error bounds. Starting from \eqref{eq:vbEs}, the proposed numerical algorithm is based on the following choices:
\begin{itemize}
\item the variational problem \eqref{eq:vbEs} is replaced by the first-order optimality condition
\begin{equation}\label{eq:nne}
\nabla \GEneN(\bsigmah^{\iter+1};\bsigmah^{\iter},\tau)=\nabla F(\bsigmah^{\iter+1})+\frac{\bsigmah^{\iter+1}-\bsigmah^{\iter}}{\tau}=0,
\end{equation} 
which in principle may be not equivalent to eq. \eqref{eq:vbEs}.
\item the solution of the non-linear problem \eqref{eq:nne} is approximated by Newton method starting at $\bsigmah^{\iter}$,
\item Newton method is stopped with a specific criterion based on the relative residual and on the sign of the new iterate, quantities that are both easily computed \emph{a posteriori},
\item the algorithm stops when $\|\nabla F(\bsigmah^{\iter})\|$ is smaller than a prescribed tollerance, as suggested by Proposition \ref{prop:refinedCond}.
\end{itemize}
The final scheme is reported in Algorithm \ref{alg:backEuler}.
\begin{algorithm}
\caption{Backward Euler Scheme with Newton solver}
\label{alg:backEuler}
\begin{algorithmic}
\STATE{Input $\bsigmah^0\in \R^N$, $\tau>0$, $n_{step}\in \N$, $toll>0$, $\epsilon>0$ }
\STATE{Set $\iter:=0$}
\STATE{Compute $res=\|\nabla \FEneN(\bsigmah^0)\|$}
\IF{$res=0$}
\STATE{\textbf{Exit} with error.}
\ENDIF
\WHILE{$\iter<n_{step}$ \AND $res>toll$}
\STATE{Set $\iter=\iter+1$, $\bsigma^{new}:=\bsigma^{old}$}
\STATE{Compute $res_{Newton}:=\nabla \GEneN(\bsigma^{new};\bsigma^{old},\tau)$  }
\WHILE{$|\comp{res_{Newton}}{i}|>\epsilon |\comp{\bsigma^{new}-\bsigma^{old}}{i}|$ for some $i$ \OR $\sign \bsigma^{new}\neq \sign \bsigma^{old}$}
\STATE{Compute $\bsigma^{new}=\bsigma^{new}-[\Hess \GEneN(\bsigma^{new};\bsigma^{old},\tau)]^{-1}\nabla \GEneN(\bsigma^{new};\bsigma^{old},\tau) $}
\STATE{Compute $res_{Newton}:=\nabla \GEneN(\bsigma^{new};\bsigma^{old},\tau)$}
\ENDWHILE
\STATE{Compute $res=\|\nabla \FEneN(\bsigma^{new})\|$}
\ENDWHILE
\RETURN $\bsigma^{new}$
\end{algorithmic}
\end{algorithm}  

Passing from the study of the iterated solution of \eqref{eq:vbEs} to the study of Algorithm \ref{alg:backEuler} some issues come into play. First, since the function $\GEneN(\cdot;\bsigmah^{\iter},\tau)$ may be not globally convex, the first order optimality condition \eqref{eq:nne}, though necessary, is not in general sufficient for \eqref{eq:vbEs} holding true. Moreover, equations \eqref{eq:vbEs} and \eqref{eq:nne} may have more than one solution. Lastly, convergence of Newton method is guaranteed only for a suitable choice of the initial guess. In Theorem \ref{thm:concistency} we will prove that, for a sufficiently small $\tau$ depending only on $\bsigma^0$, we can iterativelly compute $\bsigmah^{\iter+1}$ by solving \eqref{eq:nne} by Newton method (possibly with an infinite number of steps) with $\bsigmah^{\iter}$ as initial guess. The resulting sequence $\{\bsigmah^{\iter}\}_{\iter\in \N}$ satisfies \eqref{eq:vbEs} at each step and the sequence $\{\Sq{(\bsigmah^{\iter})}\}_{\iter\in \N}$ converges to a global  minimizer of $\DEneN.$ Note that the use of the Newton method and the prescribed choice of the initial data provide a natural well-defined sequence $\{\bsigmah^{\iter}\}_{\iter\in \N}$ even when the solution of \eqref{eq:vbEs} is not unique.

Theorem \ref{thm:concistency} refers to an exact solution of \eqref{eq:nne}, typically obtained by infinite iterations. On the other hand, the propagation of the error introduced when Newton method is stopped erlier must be taken into account in the numerical analysis of the scheme.  In Algorithm \ref{alg:backEuler} we choose a specific stopping criterion that allows us to provide a \emph{stability estimate} in terms of the decrease of the functional $\FEneN$ along the computed discrete trajectory, see Proposition \ref{prop:stability}. Finally, using these last results, we are able to prove in Theorem \ref{thm:convergence} the \emph{convergence} of Algorithm \ref{alg:backEuler} as $n_{step}\to +\infty$ towards the solution of the discrete Monge-Kantorovich equations \eqref{eqn:discretemongekantorovich}.
\begin{theorem}[Consistency of Algorithm~\ref{alg:backEuler}]\label{thm:concistency}
Let $\min \FEneN<c<+\infty$. Then, there exists $\tau^*(c)>0$ such that, for any $\bsigmah^0$ with $\FEneN(\bsigmah^0)<c$ and $\nabla \FEneN(\bsigmah^0)\neq 0$, and for any $0<\tau<\tau^*(c),$ the following sequences are well defined
\begin{align*}
&\bsigmah^{\iter+1,0}:=\bsigmah^{\iter}\\
&\bsigmah^{\iter+1,r+1}:=\bsigmah^{\iter+1,r}-\left[\Hess \GEneN(\bsigma^{\iter+1,r};\bsigmah^{\iter},\tau)\right]^{-1}\nabla \GEneN(\bsigma^{\iter+1,r};\bsigmah^{\iter},\tau) \\
&\bsigmah^{\iter+1}:=\lim_r\bsigmah^{\iter+1,r}.
\end{align*}
For any $\iter\in \N$ we have
\begin{equation}\label{discreteGF}
\nabla \FEneN(\bsigmah^{\iter+1})=-\frac{\bsigmah^{\iter+1}-\bsigmah^{\iter}}{\tau}.
\end{equation}
Moreover there exists $\bmuh^*\in \argmin \DEneN$ such that
\begin{equation}
\lim_\iter \Sq{(\bsigmah^{\iter})}=\bmuh^*.
\end{equation}
\end{theorem}

\begin{proof}
The first part of the proof is a sharp version of the classical proof of the local convergence of Newton method. 
Let $\Omega_\iter$ be the connected component of  the interior of the set $\{\FEneN\leq\FEneN(\bsigmah^{\iter})\}$ containing $\bsigmah^{\iter}$. Let us set 
\begin{align}
& d_\iter:=\diam(\Omega_\iter),\\
&U_\iter:=\cup_{\bsigma\in \Omega_\iter}B(\bsigma,d_\iter),\\
& \gamma_\iter:=\max_{\bs\varsigma\in \co U_\iter}\max_{\bsigma\in \Omega_\iter}\|\nabla \GEneN(\bs\varsigma;\bsigma,\tau)\|,\\
&\lambda_\iter:=\min_{\bs\varsigma\in \co U_\iter} \lambda_{min}(\Hess \FEneN(\bs\varsigma)),\\
&\Lambda:=\max_{\bs\varsigma\in \overline{\Omega}} \lambda_{max}(\Hess \FEneN(\bs\varsigma)),\label{Lambdamax_def}\\
&P_i(\bs\varsigma):=\Hess \partial_i\FEneN(\bs\varsigma)),\;i=1,2,\dots,\Ntdens,\\
&R_\iter:=\left(\max_{\bs\varsigma\in \co U_\iter}\rho(P_1(\bs\varsigma)),\dots,\max_{\bs\varsigma\in \co U_\iter}\rho(P_{\Ntdens}(\bs\varsigma))\right).
\end{align}
Notice that the function $\GEneN(\cdot;\bsigma,\tau)$ has Hessian matrix independent by $\bsigma$:
$$\Hess \GEneN(\bs\xi;\bsigma,\tau)=\Hess \FEneN(\bs\xi)+\frac 1 \tau \mathbb I$$
and it is strongly convex on $U_{\iter}$ for any $\bsigma\in\Omega_\iter$, provided
\begin{equation}\label{taubound1}
\tau< \frac 1{\lambda_\iter^-}=:\tau_{1,\iter}.
\end{equation}
Assuming \eqref{taubound1}, we can denote by $\hatbsigmah^{\iter+1}$ the unique minimizer of $\GEneN(\cdot;\bsigmah^{\iter},\tau)$ in $U_{\iter}.$ Note that, since 
$$\FEneN(\hatbsigmah^{\iter+1})+\frac{\|\hatbsigmah^{\iter+1}-\bsigmah^{\iter}\|^2}{2\tau}=\GEneN(\hatbsigmah^{\iter+1};\bsigmah^{\iter},\tau)\leq \GEneN(\bsigmah^{\iter};\bsigmah^{\iter},\tau)=\FEneN(\bsigmah^{\iter}),$$
then
\begin{equation}\label{estimateF}
\FEneN(\bsigmah^{\iter})-\FEneN(\hatbsigmah^{\iter+1})\geq \frac{\|\hatbsigmah^{\iter+1}-\bsigmah^{\iter}\|^2}{2\tau}.
\end{equation}
This in particular implies that $\FEneN(\hatbsigmah^{\iter+1})\leq \FEneN(\bsigmah^{\iter})$, but $\hatbsigmah^{\iter+1}$ may lie in a component of $\{\FEneN<F(\bsigmah^{\iter})\}$ different from $\Omega_\iter.$ This cannot occour because  $\GEneN(\cdot):=\GEneN(\cdot\,;\bsigmah^{\iter},\tau)$ is $(\lambda_\iter+1/\tau)$-convex in $U_\iter$, for any $\bsigma,\bs\varsigma\in U_\iter$, i.e., 
$$\GEneN((1-t)\bsigma+t\bs\varsigma)\leq (1-t)\GEneN(\bsigma)+t\GEneN(\bs\varsigma)-t(1-t)\frac{\lambda_\iter\tau+1}{2\tau}|\bsigma-\bs\varsigma|^2,\;\;\forall t\in(0,1)\;.$$
Indeed, if we assume that $\hatbsigmah^{\iter+1}$ lies in the disjoint component $\tilde \Omega_\iter$, take $\bs\xi\in [\hatbsigmah^{\iter+1},\bsigmah^{\iter}]\cap \partial\tilde \Omega_\iter$ so that $\FEneN(\bs\xi)=\FEneN(\bsigmah^{\iter})$, and set $\bsigma=\bsigmah^{\iter}$ and $\bs\varsigma=\hatbsigmah^{\iter+1}$ to yeld
\begin{align*}
  G(\bs\xi)&=
             \FEneN(\bs\xi)+\frac{\|\bs\xi-\bsigmah^{\iter}\|^2}{2\tau}= \FEneN(\bsigmah^{\iter})+\frac{\|\bs\xi-\bsigmah^{\iter}\|^2}{2\tau}\\
           & \leq (1-t)\FEneN(\bsigmah^{\iter})+t\FEneN(\hatbsigmah^{\iter+1})\\
           &\;\;\;\;\;\;\;\;\;\;\;\;\;\;+
             t\frac{\|\hatbsigmah^{\iter+1}-\bsigmah^{\iter}\|^2}{2\tau}-t(1-t)\frac{\lambda_\iter\tau+1}{2\tau}\|\hatbsigmah^{\iter+1}-\bsigmah^{\iter}\|^2\, ,
\end{align*}
where $t=\|\bs\xi-\bsigmah^{\iter}\|/\|\hatbsigmah^{\iter+1}-\bsigmah^{\iter}\|$.
From the previous equation we can write:
\begin{align*}
  &t(\FEneN(\bsigmah^{\iter})-\FEneN(\hatbsigmah^{\iter+1}))\\
  \leq&
    -\frac{\|\bs\xi-\bsigmah^{\iter}\|}{2\tau}
    +t\frac{\|{\hatbsigmah}^{\iter+1}-\bsigmah^{\iter}\|}{2\tau}
    -t(1-t)(\lambda_\iter\tau+1)\frac{\|{\hatbsigmah}^{\iter+1}-\bsigmah^{\iter}\|}{2\tau} \\
  \leq& -t(1-t)\lambda_\iter\tau\frac{\|\hatbsigmah^{\iter+1}-\bsigmah^{\iter}\|^2}{2\tau}\leq 0
\end{align*}
which is in contrast with \eqref{estimateF}. As a consequence $\hatbsigmah^{\iter+1}\in \Omega_\iter.$ 

We want to show that using Newton method with $\bsigmah^{\iter}$ as initial guess we compute  $\hatbsigmah^{\iter+1}$ (possibly with an infinite number of steps).
It is convenient to introduce the following notations in order to simplify the computations:
\begin{align*}
&e^{\iter+1,r}:=\bsigmah^{\iter+1,r}-\hatbsigmah^{\iter+1}\\
&s^{\iter+1,r}:=\bsigmah^{\iter+1,r}-\hatbsigmah^{\iter+1,r-1}\,.
\end{align*}
Then, we can rewrite the Newton step as
\begin{equation*}
  s^{\iter+1,r+1}=-[\Hess \GEneN(\bsigmah^{\iter+1,r})]^{-1}\nabla \GEneN(\bsigmah^{\iter+1,r}).
\end{equation*}
We now assume that, for a given $r\in\mathbb{N}$ and for $s=0,1,\ldots,r$, $\bsigmah^{\iter+1,s}\in U_\iter.$ Writing the  second order Taylor expansion for $0=\nabla \GEneN(\hatbsigmah^{\iter+1})$ centered at $\bsigmah^{\iter+1,r}$, multipling both sides of the equation by $[\Hess \GEneN(\bsigmah^{\iter+1,r})]^{-1}$ and using Newton equation, we obtain 
\begin{align*}
&\|e^{\iter+1,r+1}\|\\
\leq& \frac 1 2 \left\|[\Hess \GEneN(\bsigmah^{\iter+1,r})]^{-1}((e^{\iter+1,r})^tP_1(\bs\xi)e^{\iter+1,r},\dots, (e^{\iter+1,r})^tP_{\Ntdens}(\bs\xi)e^{\iter+1,r})^t\right\|\\
               \leq& \frac{\tau}{2(\lambda_\iter\tau+1)}\|R\|\|e^{\iter+1,r}\|^2.
\end{align*} 
Thus
\begin{equation}\label{quadraticnewton}
\|e^{\iter+1,r+1}\|\leq \frac{\|R\|\tau}{2(\lambda_\iter\tau+1)}\|e^{\iter+1,r}\|^2.
\end{equation}
Using the first order Taylor expansion of $0=\nabla \GEneN(\hatbsigmah^{\iter+1})$ centered at $\bsigmah^{\iter}=\bsigmah^{\iter+1,0}$ we obtain 
\begin{equation*}
  \|e^{\iter+1,0}\|\leq \frac{\tau}{\lambda_\iter\tau+1}\|\nabla \GEneN(\bsigmah^{\iter})\|\leq \frac{\tau\gamma}{\lambda_\iter\tau+1}.
\end{equation*}
It is not difficult to see that, if
\begin{equation}
\label{taubound2}
\tau<\frac{1}{(\sqrt{\|R_\iter\|\gamma_\iter}-\lambda_\iter)^+}=:\tau_{2,\iter},
\end{equation}
then
$\frac{\|R\|\tau}{2(\lambda_\iter\tau+1)}\|e^{\iter+1,0}\|<1/2$. Thus, due to~\eqref{quadraticnewton}, we have
\begin{equation}
\label{linearnewton}
\|e^{\iter+1,1}\|\leq\frac{1}{2}\|e^{\iter+1,0}\|\,.
\end{equation}
In particular,
\begin{equation*}
  \frac{\|R\|\tau}{2(\lambda_\iter\tau+1)}\|e^{\iter+1,1}\|\leq\frac{1}{2}\frac{\|R\|\tau}{2(\lambda_\iter\tau+1)}\|e^{\iter+1,0}\|
\end{equation*}
and, by finite induction, $\|e^{\iter+1,r+1}\|\leq\frac{1}{2^{r+1}}\|e^{\iter+1,0}\|$.
Therefore, we can conclude that $\bsigmah^{\iter+1,r+1}\in B(\hatbsigmah^{\iter+1},d_\iter)\subseteq U_\iter$ and that we can iterate the above reasoning to get
\begin{equation*}
  \|e^{\iter+1,r+1}\|\leq\frac{1}{2^{r+1}}d_\iter
  \qquad \mbox {for all } r\in\mathbb{N}\,.
\end{equation*}
Using in addition \eqref{quadraticnewton}, we can conclude that $\bsigmah^{\iter+1,r}$ is converging quadratically to $\bsigmah^{\iter+1}:=\hatbsigmah^{\iter+1}$ and \eqref{discreteGF} holds true.    

Note that the bounds~\eqref{taubound1} and~\eqref{taubound2} on the time-step size can be uniformly satisfied by setting
\begin{equation}\label{eq:tau-bound}
  \tau^*(c):=\min\{\tau^*_1(c),\tau^*_2(c)\}\leq \min\{\inf_\iter \tau_{1,\iter}, \inf_\iter \tau_{2,\iter}  \}\, ,
\end{equation}
where
\begin{equation*}
  \tau^*_1(c):=\frac{1}{\lambda^-_c}
  \qquad
  \tau^*_2(c):=\frac{1}{(\sqrt{\|R_c\|\gamma_c}-\lambda_c)^+}\, ,
\end{equation*}
and
\begin{align*}
&d:=\diam(\Omega_0)\, ,\\
&U:=\cup_{\bsigma\in \Omega_0}B(\bsigma,d)\, ,\\
&\lambda_c:=\min_{\bs\varsigma\in \co U} \lambda_{min}(\Hess \FEneN(\bs\varsigma))\, ,\\
&R_c:=\left(\max_{\bs\varsigma\in \co U}\rho(P_1(\bs\varsigma)),\dots,\max_{\bs\varsigma\in \co U}\rho(P_{\Ntdens}(\bs\varsigma))\right)\, ,\\
& \gamma_c:=\max_{\bs\varsigma\in \co U}\max_{\bsigma\in \Omega_0}\|\nabla \GEneN(\bs\varsigma;\bsigma,\tau)\|\, .
\end{align*}

Now we show that the sequence $\{\bsigmah^{\iter}\}_{\iter\in \N}$ is convergent to a point $\bsigmah^*$, that needs to be a critical point for $\FEneN.$ To this aim, since $0=\tau \nabla \GEneN(\bsigmah^{\iter};\bsigma^{\iter-1},\tau)=\tau \nabla \FEneN(\bsigmah^{\iter})+\bsigmah^{\iter}-\bsigma^{\iter-1}$, if $\nabla \FEneN(\bsigmah^{\iter})=0$ for some $\iter$, then $\bsigmah^{\iter}\equiv \bsigma^{\iter-1}\equiv\dots\equiv\bsigma^0$. In such a case $\nabla \FEneN(\bsigma^0)=\nabla \FEneN(\bsigmah^{\iter})=0$ but this violates the hypothesis Theorem \ref{thm:concistency}. Thus we can assume $\nabla \FEneN(\bsigmah^{\iter})\neq 0$ $\forall \iter \in\N.$

Since $\bsigmah^{\iter}\in \Omega_\iter$ $\forall \iter\in \N$, using \eqref{discreteGF} and  the definintion of $\lambda_c$,  we can write the second order Taylor expansion of $\FEneN(\bsigmah^j)$ centered at $\bsigmah^{j+1}$ and get the estimate
$$\FEneN(\bsigmah^j)-\FEneN(\bsigmah^{j+1})\geq \frac{\lambda_c\tau+2}{2\tau}\|\bsigmah^{j+1}-\bsigmah^{j}\|^2=\tau\frac{\lambda_c\tau+2}2 \|\nabla \FEneN(\bsigmah^{j+1})\|^2.$$
Summing up over $j$ ranging from $0$ to $\iter-1$ we obtain
\begin{equation}\label{eq:seriesestimate}
\FEneN(\bsigma^0)-\FEneN(\bsigmah^{\iter})\geq \frac{\lambda_c\tau+2}{2\tau}\sum_{j=0}^{\iter}\|\bsigma^{j+1}-\bsigma^{j}\|^2=\tau\frac{\lambda_c\tau+2}2 \sum_{j=0}^{\iter}\|\nabla \FEneN(\bsigma^{j+1})\|^2.
\end{equation}
Since $\FEneN$ is bounded from below, then the sequences $\{\|\bsigmah^{\iter+1}-\bsigmah^{\iter}\|\}_{\iter\in \N}$ and $\{\|\nabla \FEneN(\bsigmah^{\iter})\|\}_{\iter\in \N}$ are square-summable, and we can conclude that $\{\bsigmah^{\iter}\}_{\iter\in \N}$ is Cauchy. Denoting by $\bsigmah^*$ the limit of $\{\bsigmah^{\iter}\}_{\iter\in \N},$  the continuity of $\nabla \FEneN$ allows us to state that $\bsigmah^*$ is critical for $\FEneN.$

We are still left to prove that $\bmuh^*:=\Sq{\bsigmah^*}$ is a minimizer of $\DEneN$. To this aim we introduce the set $I_0:=\{i\in \{1,2,\dots,\Ntdens\}:\,\bmuhi^*=0\}$ and we distinguish the two cases $I_0=\emptyset$ and $I_{\IndexN}\neq \emptyset.$ In the first case, if for any $i=1,2,\dots,\Ntdens$ $\bmuhi^*\neq 0$ and $0=\partial_i \FEneN(\bsigmah^*)=2\bsigmahi^*\partial_i\DEneN(\bsigmah^*)$, then we have $\nabla \DEneN(\bmuh^*)=0$. The convexity of $\DEneN$ allows to conclude.

In the second case we first notice that, for any $i=1,2,\dots, \Ntdens$, the sequence $\{\bsigmahi^{\iter}\}_{\iter\in \N}$ has constant sign, being either constantly non-negative or constantly non-positive. FMoreover we recall that, $\bsigmah^{\iter+1}$ is characterized as the unique minimizer of $\GEneN(\cdot;\bsigmah^{\iter},\tau)$ in $\Omega_\iter$. If for one $i\in \{1,2,\dots \Ntdens\}$ we have $\sign \bsigmahi^{\iter}\neq \sign (\bsigmahi^{\iter+1})$, then the point 
$$\tilde \bsigmah^{\iter+1}:=(\sign \comp{\bsigmah^{\iter}}{1} |\comp{\bsigmah^{\iter+1}}{1}|,\dots,\sign \comp{(\bsigmah^{\iter}}{\Ntdens} |\comp{\bsigmah^{\iter+1}}{\Ntdens}|)$$
is such that $\FEneN(\tilde \bsigmah^{\iter+1})=\FEneN(\bsigmah^{\iter+1})$ and  $\|\tilde \bsigmah^{\iter+1}-\bsigmah^{\iter}\|^2\leq \|\bsigmah^{\iter+1}-\bsigmah^{\iter}\|^2$.  Thus $\tilde \bsigmah^{\iter+1}$ minimizes $\GEneN(\cdot;\bsigmah^{\iter},\tau)$ in $\Omega_\iter$ as well, leading to a contradiction. Thus $\{\bsigmahi^{\iter}\}_{\iter\in \N}$ needs to be either non-negative for any $\iter$ or non-positive for any $\iter$.

Then, we can pick a subsequence $j\mapsto \iter_j$ such that $\sign \bsigmahi^{\iter_j}=\sign \partial_i \FEneN(\bsigma^{\iter_j}),$ for any $i\in I_0.$ This last claim can be proven again by contradiction. Indeed, writing $\bsigmahi^{\iter}=\bsigmahi^{\iter-1}-\tau \partial_i\FEneN(\bsigmah^{\iter})$, and using the fact that the sign of $\bsigmahi^{\iter}$ is constant, shows that $\{\bsigmahi^{\iter}\}_{\iter\in \N}$ must be a non-increasing sequence for $(\bsigmah^{0})_i<0$ or a non-decreasing sequence for $(\bsigmah^{0})_i>0$, and  cannot converge to $0$.

From the fact that
$$\sign \comp{(\bsigmah^{\iter_j})}{i}=\sign \partial_i\FEneN(\bsigmah^{\iter_j})=\sign \left(  2\comp{(\bsigmah^{\iter_j})}{i}\partial_i\DEneN(\Sq{(\bsigmah^{\iter_j})})\right),$$ we easily obtain
$$\partial_i\DEneN(\bmuh^*)=\lim_j \partial_i\DEneN(\Sq{(\bsigmah^{\iter_j})})\geq 0,\;\;\forall i\in I.$$
On the other hand, $\partial_i\DEneN(\bmuh^*)=0$ for any $i\notin I_0$ due to \eqref{eq:seriesestimate} and the definition of $I_0$. Thus we have
\begin{equation}
\label{KKTdiscrete}
\begin{cases}
\partial_i\DEneN(\bmuh^*)\geq 0,&\forall i\\
\partial_i\DEneN(\bmuh^*)=0,& \forall i: \bmuhi^*>0
\end{cases}\;\;.
\end{equation}
These are precisely the KKT necessary optimality conditions for  $\DEneN$ on $\Rpn.$ Since the function $\DEneN$ is convex, such conditions are also sufficient, i.e., $\bmuh^*\in \argmin \DEneN.$
\end{proof}
In order to continue our study of Algorithm \ref{alg:backEuler}, it is convenient to introduce some notations. Let us denote by 
$$\Phi(\cdot;\tau):\R^{\Ntdens}_\geq\rightarrow \R^{\Ntdens}_\geq$$
the map that, for any $\bsigmah^{\iter}\in \R^{\Ntdens}$, returns the exact solution of \eqref{eq:nne} provided by the Newton method with $\bsigmah^{\iter}$ as initial guess. As a biproduct of Theorem \ref{thm:concistency}  this map is well defined, provided $\tau>0$ is sufficiently small. We also define the map $\Phi_\epsilon(\cdot;\tau):\R^{\Ntdens}\rightarrow \R^{\Ntdens}$ as
$$\Phi_\epsilon(\bsigma^\iter;\tau):=\bsigma^{\iter+1,\hat r},$$
where
\small
\begin{equation}
\label{rdef}
\hat r:=\min\left\{r: |\partial_i \GEneN(\bsigmah^{\iter+1,r})|\leq \epsilon |\comp{\bsigmah^{\iter+1,r}}{i}-\comp{\bsigmah^{\iter}}{i}|,\;\sign \comp{\bsigmah^{\iter+1,r}}{i}=\sign\comp{\bsigmah^{\iter}}{i}\,\forall i\right\}.
\end{equation}
\normalsize
We remark that, given an intial point $\bsigmah^0$ and suitable $\epsilon,\tau>0$, the Algorithm \ref{alg:backEuler} computes the finite sequence
$$(\bsigmah^0,\bsigmah^1,\bsigmah^2,\bsigmah^3,\dots):=\left(\bsigmah^0, \Phi_\epsilon(\bsigmah^0;\tau), \Phi^2_\epsilon(\bsigmah^0;\tau),\Phi^3_\epsilon(\bsigmah^0;\tau),\dots\right)$$
of length at most $n_{step}+1$.
\begin{proposition}[Stability estimate for Algorithm \ref{alg:backEuler}]
\label{prop:stability}
Let $\min \FEneN<c<+\infty$ and let $\tau<\tau^*(c)$ be as above. There exists  $\epsilon^*>0$, depending only on $c$ and $\tau$, such that 
\begin{equation}
\FEneN(\bsigma)-\FEneN(\Phi_\epsilon(\bsigma;\tau))\geq \frac {\FEneN(\bsigma)-\FEneN(\Phi(\bsigma;\tau))} 2,\;\;\forall \bsigma:\FEneN(\bsigma)<c, 
\end{equation}
for any $0<\epsilon\leq\epsilon^*$.
\end{proposition}

\begin{proof}
It is convenient to introduce the notation
\begin{align*}
  \Delta &:= \FEneN(\bsigma)-\FEneN(\Phi(\bsigma;\tau))\;,\\
  \Delta_\epsilon &:=\FEneN(\bsigma)-\FEneN(\Phi_\epsilon(\bsigma;\tau))\;,\\
  \bss &:=\Phi(\bsigma;\tau)-\bsigma \;,\\
  \bsse &:= \Phi_\epsilon(\bsigma;\tau)-\bsigma\;.
\end{align*}
Let us pick $0<\epsilon<\lambda_c+1/\tau$.
Using the standard error bound for the Newton method we can write
\begin{equation*}
  \|\bs{e}_\epsilon\|:=\|\Phi_\epsilon(\bsigma;\tau)-\Phi(\bsigma;\tau)\|
  \leq \max_{\bs\xi\in \Omega_0}\|[\Hess \GEneN(\bs\xi)]^{-1}\|\|\nabla \GEneN(\Phi_\epsilon(\bsigma;\tau))\|.
\end{equation*}
Then, by the $(\lambda_c+1/\tau)$-convexity of $\GEneN$ and the stopping criterion \eqref{rdef}, we get
\begin{equation}\label{ebound}
\|\bs{e}_\epsilon\|\leq \frac{\tau\epsilon}{\lambda_c\tau+1}.
\end{equation} 
By the triangular inequality we have
\begin{equation*}
  \|\bs{s}_\epsilon\|
  \leq\frac{\lambda_c\tau+1}{(\lambda_c-\epsilon)\tau+1}\|\bs{s}\|
  :=C_{\epsilon}\|\bs{s}\|\, ,
\end{equation*}
which yields:
\begin{equation}\label{eq:eepsilonestimate}
    \|\bs{e}_\epsilon\|\leq
    \frac{C_{\epsilon}\epsilon\tau}{\lambda_c\tau+1}\|\bs{s}\|\, .
\end{equation}
Using the second order Taylor expansion of $\GEneN(\bsigma)$ centered at
$\Phi(\bsigma;\tau)$ we can obtain
\begin{equation}\label{sbound}
  \Delta\geq \frac{\lambda_c\tau+2}{2\tau}\|\bs{s}\|^2.
\end{equation}
On the other hand, due to \eqref{eq:eepsilonestimate} and \eqref{sbound}, we can write
\begin{align*}
  \Delta_\epsilon
  &=\Delta-(\FEneN(\Phi(\bsigma;\tau))-\FEneN(\Phi_\epsilon(\bsigma;\tau)))
    =\Delta+\langle \bs{e}_{\epsilon},\frac{\bs{s}}{\tau}\rangle -
    \frac{1}{2} \bs{e}_{\epsilon}^T\Hess \FEneN(\bs{\xi})\bs{e}_{\epsilon} \\
  &\geq \Delta - \frac{\|\bs{s}\|\|\bs{e}_{\epsilon}\|}{\tau}
    -\frac{\Lambda}{2}\|\bs{e}_{\epsilon}\|^2 
    \geq \Delta - \|\bs{s}\|^2
    \left(
    C_{\epsilon}\frac{\epsilon}{\lambda_c\tau+1}+
    \frac{\Lambda
    C_{\epsilon}^2}{2}\frac{(\epsilon\tau)^2}{(\lambda_c\tau+1)^2}
    \right) \\
  &\geq \Delta
    \left[
    1-\frac{2\tau
    C_{\epsilon}\epsilon}{(\lambda_c\tau+2)(\lambda_c\tau+1)}
    \left(
    1+\frac{\Lambda\tau^2C_{\epsilon}\epsilon}{\lambda_c\tau+1}
    \right)
    \right]\;.
\end{align*}
In order to conclude the proof, we are left to verify that for small
$\epsilon>0$ we have
\begin{equation*}
    1-\frac{2\tau
    C_{\epsilon}\epsilon}{(\lambda_c\tau+2)(\lambda_c\tau+1)}
    \left(
    1+\frac{\Lambda\tau^2C_{\epsilon}\epsilon}{\lambda_c\tau+1}
    \right)\geq\frac{1}{2}.
\end{equation*}
Thus
\begin{equation}\label{eq:Cepsilonepsilonestimate}
  C_{\epsilon}\epsilon\leq
  \left(\sqrt{1+\Lambda\tau(\lambda_c\tau+2)}-1\right)
  \frac{\lambda_c\tau+1}{2\Lambda\tau^2}\,.
\end{equation}
Note that $C_{\epsilon}\epsilon=\epsilon+o(\epsilon)$ as
$\epsilon\to 0^+$, so \eqref{eq:Cepsilonepsilonestimate} is certainly satisfied for $\epsilon>0$ small enough.
\end{proof}

\begin{theorem}[Convergence of Algorithm \ref{alg:backEuler}]
\label{thm:convergence}
Let $\min \FEneN<c<+\infty$ and $\bsigmah^0$ such that $\FEneN(\bsigmah^0)<c$ and $\nabla \FEneN(\bsigmah^0)\neq 0$. Let $\tau>0$ and $\epsilon$ satisfy the hypothesis of Proposition \ref{prop:stability}. Then the sequence $\{\bsigmah^\iter\}_{\iter\in \N}:=\{\Phi_\epsilon^{\iter}(\bsigmah^0;\tau)\}_{\iter\in \N}$ admits a limit $\bsigmah^*$. Moreover
\begin{equation}
\bmuh^*:=\Sq{(\bsigmah^*)}\in \argmin \DEneN.
\end{equation}
\end{theorem}
\begin{proof}
If $\nabla \FEneN(\Phi_\epsilon^{(\iter)}(\bsigmah^0;\tau))=0$ for certain $\iter$, then the use of $\Phi_\epsilon^{(\iter)}(\bsigmah^0;\tau)$ as initial guess in the Newton method for the computation of $[\Phi_\epsilon]^{(\iter+1)}(\bsigmah^0;\tau)$ guarantees that $\Phi_\epsilon^{(\iter+j)}(\bsigmah^0;\tau)=\Phi_\epsilon^{(\iter)}(\bsigmah^0;\tau)$ for any $j\in \N.$ More importantly, using \eqref{rdef}, we can repeat the argument of the proof of Theorem \ref{thm:concistency}, to show that $\Phi_\epsilon^{(\iter)}(\bsigmah^0;\tau)\equiv \bsigmah^0$, which is not possible since we are assuming $\nabla F(\bsigmah^0)\neq 0.$ Therefore we can assume without lost of generality that $\nabla \FEneN(\Phi_\epsilon^{(\iter)}(\bsigmah^0;\tau))\neq 0$ $\forall \iter\in \N.$ 

From Proposition \ref{prop:stability} and the optimality of the exact step, i.e.,  $\Phi^{(\iter)}(\bsigmah^0;\tau)-\Phi^{(\iter-1)}(\bsigmah^0;\tau)=-\tau \nabla \FEneN(\Phi^{(\iter)}(\bsigmah^0;\tau))$, to get
\begin{align*}
&\FEneN(\Phi_\epsilon^{(\iter)}(\bsigmah^0;\tau))-\FEneN(\Phi_\epsilon^{(\iter+1)}(\bsigmah^0;\tau))\\
\geq &\frac 1 2 \FEneN(\Phi_\epsilon^{(\iter)}(\bsigmah^0;\tau))-\FEneN(\Phi(\Phi_\epsilon^{(\iter)}(\bsigmah^0;\tau));\tau))\\
\geq&\frac{\lambda_c\tau+1}{4\tau}\left\|\Phi(\Phi_\epsilon^{(\iter)}(\bsigmah^0;\tau));\tau)-\Phi_\epsilon^{(\iter)}(\bsigmah^0;\tau)  \right\|^2.
\end{align*}
Notice that, using the notation introduced in the proof of Proposition \ref{prop:stability} and in particular eq. \eqref{sbound}, this last inequality can be written in the compact form as
\begin{equation}\label{estimateDeltas}
\Delta_\epsilon\geq \frac 1 2\Delta\geq \frac{\lambda_c\tau+2}{2\tau} \|\bs{s}\|^2>\frac{\lambda_c\tau+1}{4\tau}\|\bs{s}\|^2.
\end{equation}
On the other hand, using \eqref{ebound}, we get 
\begin{equation}\label{estimatebs}
\|\bs s_\epsilon\|=\|\bs s+e_\epsilon\|\leq \frac{\lambda_c \tau+1}{(\lambda_c-\epsilon)\tau+1}\|\bs s\|.
\end{equation}
Therefore we have
\begin{align*}
&\FEneN(\Phi_\epsilon^{(\iter)}(\bsigmah^0;\tau))-\FEneN(\Phi_\epsilon^{(\iter+1)}(\bsigmah^0;\tau))\geq \frac{[(\lambda_c-\epsilon)\tau+1]^2}{4\tau(\lambda_c\tau+1)}\|\bs s_\epsilon\|^2\\
=&\frac{[(\lambda_c-\epsilon)\tau+1]^2}{4\tau(\lambda_c\tau+1)}\|\Phi_\epsilon^{(\iter+1)}(\bsigmah^0;\tau)-\Phi_\epsilon^{(\iter)}(\bsigmah^0;\tau)\|^2.
\end{align*}
It is clear that, for any $\iter\in \N$ we have $\FEneN(\bsigmah^0)-\FEneN(\Phi_\epsilon^{(\iter)}(\bsigmah^0;\tau))\leq c-\min \FEneN<+\infty$. Thus we can write
\begin{align*}
&+\infty>c-\min \FEneN\geq \lim_\iter\sum_{j=0}^{\iter-1}\left(\FEneN(\Phi_\epsilon^{(j)}(\bsigmah^0;\tau))-\FEneN(\Phi_\epsilon^{(j+1)}(\bsigmah^0;\tau))  \right)\\
\geq& \frac{[(\lambda_c-\epsilon)\tau+1]^2}{2\tau(\lambda_c\tau+1)}\sum_{j=0}^{+\infty}\|\Phi_\epsilon^{(j+1)}(\bsigmah^0;\tau)-\Phi_\epsilon^{(j)}(\bsigmah^0;\tau)\|^2
\end{align*}
Thus in particular $\FEneN(\Phi_\epsilon^{(\iter)}(\bsigmah^0;\tau))-\FEneN(\Phi_\epsilon^{(\iter+1)}(\bsigmah^0;\tau))\to 0$ as $\iter\to+\infty$ and $\Phi_\epsilon^{(\iter)}(\bsigmah^0;\tau)$ is a Cauchy sequence, whose limit we denote by $\bsigmah^*$.

By a similar reasoning, starting from \eqref{estimateDeltas} we can show that 
$$\|\Phi(\Phi_\epsilon^{\iter}(\bsigmah^0;\tau);\tau)-\Phi_\epsilon^{\iter}(\bsigmah^0;\tau)\|\to 0,$$
thus $\Phi(\Phi_\epsilon^{\iter}(\bsigmah^0;\tau);\tau)\to \bsigmah^*.$ It follows, by the definition of the map $\Phi$, that we have 
$$-\tau \nabla \FEneN(\Phi(\Phi_\epsilon^{\iter}(\bsigmah^0;\tau);\tau))=\Phi(\Phi_\epsilon^{\iter}(\bsigmah^0;\tau);\tau)-\Phi_\epsilon^{\iter}(\bsigmah^0;\tau).$$
Therefore we have $\|\nabla \FEneN(\bsigmah^*)\|=\lim_\iter\|\nabla \FEneN(\Phi(\Phi_\epsilon^{\iter}(\bsigmah^0;\tau);\tau))\|=0,$ i.e., $\bsigmah^*$ is critical for $\FEneN.$ 

We are left to show that $\bsigmah^*$ is actually a local minimizer of $\FEneN$ and thus $\bmuh^*$ is a global minimizer for $\DEneN.$  We reason as in the final step of the proof of Theorem \ref{thm:concistency}. If none of the components of $\Phi_\epsilon^{\iter}(\bsigmah^0)$ tends to zero, then we have $\nabla \DEneN(\Sq{(\bsigmah^*)})=0$ and we can conclude using the convexity of $\DEneN$. 

Let us assume that for a given $i\in \{1,2,\dots,\Ntdens\}$ we have $[\Phi_\epsilon^{\iter}\bsigmahi^0;\tau)\to 0.$   Now we show that we can pick a subsequence $j\mapsto \iter_j$ such that
\begin{equation}\label{assumptionOnSubsequence}
\sign \partial_i F\left( \Phi_\epsilon^{\iter_j}(\bsigmah^0;\tau)\right)=\sign \comp{\Phi_\epsilon^{\iter_j}(\bsigmah^0;\tau)}{i},\;\;\forall j\in \N.
\end{equation}
We can prove it by contradiction. Assume that
\begin{equation}\label{negassumptionOnSubsequence}
  \sign \partial_i F\left( \Phi_\epsilon^{\iter}(\bsigmah^0;\tau)\right)=-\sign \comp{\Phi_\epsilon^{\iter}(\bsigmah^0;\tau)}{i},\;\;\forall \iter\in \N.
\end{equation}
Then, from the stopping criterion we obtain: 
\begin{equation*}
\left|\comp{\Phi_\epsilon^{\iter}(\bsigmah^0;\tau)}{i}-\comp{\Phi_\epsilon^{\iter-1}(\bsigmah^0;\tau)}{i}+\tau \partial_i F\left( \Phi_\epsilon^{\iter}(\bsigmah^0;\tau)\right)\right|\\\leq \tau \epsilon\left|\comp{\Phi_\epsilon^{\iter}(\bsigmah^0;\tau)}{i}-\comp{\Phi_\epsilon^{\iter-1}(\bsigmah^0;\tau)}{i} \right|\,.
\end{equation*}
This inequality holds only if one of the following cases occurs:
\leqnomode
\begin{align}
    \comp{\Phi_\epsilon^{\iter}(\bsigmah^0;\tau)}{i}
    &\leq\comp{\Phi_\epsilon^{\iter-1}(\bsigmah^0;\tau)}{i}-\frac{\tau}{1-\epsilon\tau}\partial_i F\left( \Phi_\epsilon^{\iter_j}(\bsigmah^0;\tau)\right)\notag\\
\comp{\Phi_\epsilon^{\iter}(\bsigmah^0;\tau)}{i}&\geq \comp{\Phi_\epsilon^{\iter-1}(\bsigmah^0;\tau)}{i}\tag{A}\label{CaseA}\\
\comp{\Phi_\epsilon^{\iter}(\bsigmah^0;\tau)}{i}&\geq \comp{\Phi_\epsilon^{\iter-1}(\bsigmah^0;\tau)}{i}-\tau\partial_iF(\Phi_\epsilon^{\iter}(\bsigmah^0;\tau)))\notag
\end{align}
\begin{align}
  \comp{\Phi_\epsilon^{\iter}(\bsigmah^0;\tau)}{i}
  &\leq\comp{\Phi_\epsilon^{\iter-1}(\bsigmah^0;\tau)}{i}-\frac{\tau}{1-\epsilon\tau}\partial_i F\left( \Phi_\epsilon^{\iter_j}(\bsigmah^0;\tau)\right)\notag\\
  \comp{\Phi_\epsilon^{\iter}(\bsigmah^0;\tau)}{i}&< \comp{\Phi_\epsilon^{\iter-1}(\bsigmah^0;\tau)}{i}\tag{B}\label{CaseB}\\
  \comp{\Phi_\epsilon^{\iter}(\bsigmah^0;\tau)}{i}&\geq \comp{\Phi_\epsilon^{\iter-1}(\bsigmah^0;\tau)}{i}-\tau\partial_iF(\Phi_\epsilon^{\iter}(\bsigmah^0;\tau)))\notag
\end{align}
\begin{align}
    \comp{\Phi_\epsilon^{\iter}(\bsigmah^0;\tau)}{i}
    &\geq\comp{\Phi_\epsilon^{\iter-1}(\bsigmah^0;\tau)}{i}-\frac{\tau}{1+\epsilon\tau}\partial_i F\left( \Phi_\epsilon^{\iter_j}(\bsigmah^0;\tau)\right)\notag\\
\comp{\Phi_\epsilon^{\iter}(\bsigmah^0;\tau)}{i}&\geq \comp{\Phi_\epsilon^{\iter-1}(\bsigmah^0;\tau)}{i}\tag{C}\label{CaseC}\\
\comp{\Phi_\epsilon^{\iter}(\bsigmah^0;\tau)}{i}&< \comp{\Phi_\epsilon^{\iter-1}(\bsigmah^0;\tau)}{i}-\tau\partial_iF(\Phi_\epsilon^{\iter}(\bsigmah^0;\tau)))\notag
\end{align}
\begin{align}
  \comp{\Phi_\epsilon^{\iter}(\bsigmah^0;\tau)}{i}
  &\geq\comp{\Phi_\epsilon^{\iter-1}(\bsigmah^0;\tau)}{i}-\frac{\tau}{1+\epsilon\tau}\partial_i F\left( \Phi_\epsilon^{\iter_j}(\bsigmah^0;\tau)\right)\notag\\
  \comp{\Phi_\epsilon^{\iter}(\bsigmah^0;\tau)}{i}&< \comp{\Phi_\epsilon^{\iter-1}(\bsigmah^0;\tau)}{i}\tag{D}\label{CaseD}\\
  \comp{\Phi_\epsilon^{\iter}(\bsigmah^0;\tau)}{i}&< \comp{\Phi_\epsilon^{\iter-1}(\bsigmah^0;\tau)}{i}-\tau\partial_iF(\Phi_\epsilon^{\iter}(\bsigmah^0;\tau))) \notag
\end{align}
\reqnomode
Notice that, due to~\eqref{negassumptionOnSubsequence}, we have the following implications:
\begin{align*}
  \text{Cases~\eqref{CaseA} or~\eqref{CaseC}}
  &\Rightarrow&
                \comp{\Phi_\epsilon^{\iter}(\bsigmah^0;\tau)}{i}\geq\comp{\Phi_\epsilon^{\iter-1}(\bsigmah^0;\tau)}{i}                        \text{ and }                                    
                \comp{\Phi_\epsilon^{\iter}(\bsigmah^0;\tau)}{i}\geq 0\,,\\
  \text{Cases~\eqref{CaseB} or~\eqref{CaseD}}
  &\Rightarrow&
                \comp{\Phi_\epsilon^{\iter}(\bsigmah^0;\tau)}{i}<\comp{\Phi_\epsilon^{\iter-1}(\bsigmah^0;\tau)}{i}                         \text{ and }                                    
                \comp{\Phi_\epsilon^{\iter}(\bsigmah^0;\tau)}{i}<0\,.
\end{align*}
Since
the (component-wise) sign of the sequence $\Phi_\epsilon^{\iter}(\bsigmah^0;\tau)$ is constant by construction, then we obtain:
$$
\begin{cases}
\comp{\Phi_\epsilon^{\iter}(\bsigmah^0;\tau)}{i}\text{ is non-decreasing }&\text{ if }\bsigmahi^0\geq 0\\
\comp{\Phi_\epsilon^{\iter}(\bsigmah^0;\tau)}{i}\text{ is decreasing }&\text{ if }\bsigmahi^0< 0
\end{cases}\,.
$$
This implies, in particular, that the sequence cannot approach $0$ as $\iter\to +\infty$, leading to a contradiction. Therefore we can pick a subsequence as in \eqref{assumptionOnSubsequence} and relabel it. Notice that 
$$\sign \partial_i F\left( \Phi_\epsilon^{\iter_j}(\bsigmah^0;\tau)\right)=\sign \comp{\Phi_\epsilon^{\iter_j}(\bsigmah^0;\tau)}{i}\partial_i \DEneN\left( \Sq{\Phi_\epsilon^{\iter_j}(\bsigmah^0;\tau)}\right),$$
so we have $\partial_i \DEneN\left( \Sq{\Phi_\epsilon^{\iter_j}(\bsigmah^0;\tau)}\right)\geq 0$, $\forall j\in \N.$ Hence
$$\partial_i \DEneN(\bmuh^*)=\lim_j \partial_i \DEneN\left( \Sq{\Phi_\epsilon^{\iter_j}(\bsigmah^0;\tau)}\right)\geq 0.$$
Finally, since we proved that the KKT conditions hold for the convex function $\DEneN$ at $\bmuh^*$, i.e.,
$$
\begin{cases}
\partial_i \DEneN(\bmuh^*)=0,& \forall i: \bmuhi^*\neq 0\\
\partial_i \DEneN(\bmuh^*)\geq 0& \forall i\in\{1,2,\dots,\Ntdens\}
\end{cases},
$$
then $\bmuh^*$ is a global minimum for $\DEneN$ on $\R^{\Ntdens}_\geq.$
\end{proof}
The combination of Theorem \ref{thm:convergence} with Theorem \ref{thm:mainresult1} leads to the following result, which summarize the outcome of our construction. 
\begin{corollary}\label{cor:fundamentalcorollary}
If we assume that $\bsigmah^0\equiv \bsigma^0$ with $\bsigma^0\in \PMeaN^+$ for any $\IndexN>\bar \IndexN$, then, under the hypothesis of Theorem \ref{thm:convergence} and Theorem \ref{thm:mainresult1}, we have
\begin{equation*}
{\lim_{\IndexN}}^*\lim_\iter\mathcal I_{\IndexN}(\Sq{\bsigmah^{\iter}})=\mu^*,
\end{equation*}
where $\mu^*$ is the optimal transport density and $\lim^*$ is the limit in the weak$^*$ topology of measures.
\end{corollary}
The stability estimate of Proposition \ref{prop:stability}, the convergence of  Algorithm \ref{alg:backEuler} proven in Theorem \ref{thm:convergence}, and the technique of \cite[Prop. 2.5]{MePi10} provide a sharp estimate for the rate of convergence of Algorithm \ref{alg:backEuler} depending on the \L{}ojasiewicz exponent of $\FEneN$ at the limit point.
\begin{proposition}[Rate of convergence Algorithm \ref{alg:backEuler}]\label{prop:rateofconv}
Let $\bsigmah^0\in \Rpn$ and let $\bsigmah^{\iter}$ be computed by Algorithm \ref{alg:backEuler}, where all the involved parameters have been set according to the hypotheses of Theorems~\ref{thm:concistency} and~\ref{thm:convergence}, and of Proposition \ref{prop:stability}, and let $\bsigmah^*$ be the limit of $\bsigmah^{\iter}$. Then:
\begin{enumerate}[i)]
\item If the \L{}ojasiewicz Inequality \eqref{eq:Loja} holds for $\FEneN$ at $\bsigmah^*$ with $\vartheta<1/2$ and $L>0$,  then there exist $\bar \iter\in \N,$ $A,B,C>0$, such that, for any $\iter>\bar \iter$, we have
\begin{equation}\label{polynomialspeedAlg}
\|\bsigmah^{\iter}-\bsigmah^*\|\leq \left(C+\frac{1-2\vartheta}{\vartheta}\frac{A^{\frac{1-\vartheta}{\vartheta}}}{B}\tau(\iter-\bar \iter)  \right)^{-\vartheta/(1-2\vartheta)}.
\end{equation} 
\item If the \L{}ojasiewicz Inequality \eqref{eq:Loja} holds for $\FEneN$ at $\bsigmah^*$ with $\vartheta=1/2$ and $L>0$ (in particular if Hypothesis \ref{hyp:H2} holds), then there exist $\bar \iter\in \N,$ $A,B,C>0$, such that, for any $\iter>\bar \iter$, we have
\begin{equation}\label{exponentialspeedAlg}
\|\bsigmah^{\iter}-\bsigmah^*\|\leq C \exp\left(-\frac{A\tau}{B}(\iter-\bar \iter)  \right).
\end{equation}  
\end{enumerate}
\end{proposition}

\begin{proof}
Let us introduce the following auxiliary functions and notations. We denote by $\bsigmah:\R_{\geq 0}\rightarrow \R^{\Ntdens}$ the piecewise linear interpolation of the values $\bsigmah^0,\bsigmah^1,\bsigmah^2,\dots,$ at the nodes $0,\tau,2\tau,\dots.$ Also we denote by $H:\R^{\Ntdens}\rightarrow \R$ the piecewise linear function interpolating $\FEneN(\bsigmah^0)-\FEneN(\bsigma^*),\FEneN(\bsigmah^1)-\FEneN(\bsigma^*),\FEneN(\bsigmah^2)-\FEneN(\bsigma^*),\dots$ at $0,\tau,2\tau,\dots.$ Thus, for $t\in[(\iter-1)\tau,\iter\tau]$, we have
\begin{align*}
\bsigmah(t)&:= \frac{\iter\tau-t}{\tau} \bsigmah^{\iter-1}+ \frac{t-(\iter-1)\tau}{\tau}\bsigmah^{\iter}\\
H(t)&:=\frac{\iter\tau-t}{\tau} \FEneN(\bsigmah^{\iter-1})+ \frac{t-(\iter-1)\tau}{\tau} \FEneN(\bsigmah^{\iter})-\FEneN(\bsigma^*).
\end{align*}
Therefore we can compute, for $t\in](\iter-1)\tau,\iter\tau[$, the derivatives
\begin{align}
\bsigmah'(t)&= \frac{\bsigmah^{\iter}-\bsigmah^{\iter-1}}{\tau}=\frac{\bs s_\epsilon}{\tau}\label{eq:derivativesigam}\\
H'(t)&=\frac{\FEneN(\bsigmah^{\iter})- \FEneN(\bsigmah^{\iter-1})}{\tau}=-\frac{\Delta_\epsilon}\tau\;,\label{eq:derivativeH}
\end{align}
where we used the notation introduced in the proof of \ref{prop:stability}.

Our plan is to give upper and lower bounds for the function $H$ raised
to the power $\vartheta$, $H^\vartheta(t)$, in terms of $\|\bsigmah'(t)\|$ and $\int_t^{+\infty} \|\bsigmah'(s)\|ds$, and then combine such bounds to derive a decay estimate for $\|\bsigmah^{\iter}-\bsigmah^*\|.$ Let us start from the lower bound. We pick $\bar \iter$ large enough in order to get $\bsigmah^{\iter}$ in $B(\bsigmah^*,R)$ for all $\iter>\bar \iter$, where $R$ is the radius of the \Loja{} Inequality. 
The lower bound is obtained by means of the Fundamental Theorem of Calculus. Let us start noticing that, due to the chain rule and \eqref{eq:derivativeH}, we have
\begin{equation*}
\frac d{dt}H^\vartheta(t)=\vartheta H^{\frac \vartheta{1-\vartheta}}(t)H'(t)=-\vartheta\frac{\Delta_\epsilon}{\tau}H^{\frac \vartheta{1-\vartheta}}(t)\leq 0.
\end{equation*}
Then, using the estimates \eqref{estimateDeltas} and \eqref{estimatebs}, we obtain
\begin{equation}\label{eq:lowerintermediate}
\frac d{dt}H^\vartheta(t)\leq -\vartheta\frac{((\lambda_c-\epsilon)\tau+1)^2}{2(\lambda_c\tau+1)}H^{\vartheta-1}(t)\left(\frac{\|\bsigmah^{\iter}-\bsigmah^{\iter-1}\|}{\tau}\right)^2.
\end{equation}
Notice that, for any $t\in[(\iter-1)\tau,\iter\tau]$, there exists $\bs\xi_t$ in the segment connecting $\bsigmah^{\iter-1}$ to $\bsigmah^{\iter}$ such that $H(t)=\FEneN(\bs\xi_t)-\FEneN(\bsigmah^*).$ Therefore
$$
H^{\vartheta-1}(t)=\left[\left(\FEneN(\bs\xi_t)-\FEneN(\bsigmah^*)\right)^{1-\vartheta}\right]^{-1}\geq \frac{1}{L\|\nabla \FEneN(\bs\xi_t)\|}.
$$
Combining this last equation with \eqref{eq:lowerintermediate} and using the Taylor formula and the stopping criterion~\eqref{rdef}, we get
\begin{align*}
&\frac d{dt}H^\vartheta(t)\\
\leq & -\vartheta\frac{((\lambda_c-\epsilon)\tau+1)^2}{2(\lambda_c\tau+1)}\frac{\|\bsigmah^{\iter}-\bsigmah^{\iter-1}\|}{\tau}\frac{\|\bsigmah^{\iter}-\bsigmah^{\iter-1}\|}{\tau L\|\nabla \FEneN(\bs\xi_t)\|}\\
\leq & -\vartheta\frac{((\lambda_c-\epsilon)\tau+1)^2}{2(\lambda_c\tau+1)}\frac{\|\bsigmah^{\iter}-\bsigmah^{\iter-1}\|}{\tau}\frac{\|\bsigmah^{\iter}-\bsigmah^{\iter-1}\|}{\tau L\|\nabla \FEneN(\bsigmah^{\iter})+\Hess \FEneN(\eta)(\bs\xi_t-\bsigmah^{\iter})\|}\\
\leq & -\vartheta\frac{((\lambda_c-\epsilon)\tau+1)^2}{2(\lambda_c\tau+1)}\frac{\|\bsigmah^{\iter}-\bsigmah^{\iter-1}\|}{\tau}\\
&\;\;\;\;\;\;\;\;\;\;\;\;\frac{\|\bsigmah^{\iter}-\bsigmah^{\iter-1}\|}{\tau L\left\|\nabla \FEneN(\bsigmah^{\iter})+\frac{\bsigmah^{\iter}-\bsigmah^{\iter-1}}{\tau}-\frac{\bsigmah^{\iter}-\bsigmah^{\iter-1}}{\tau}+\Hess \FEneN(\eta)(\bs\xi_t-\bsigmah^{\iter})\right\|}\\
\leq & -\vartheta\frac{((\lambda_c-\epsilon)\tau+1)^2}{2(\lambda_c\tau+1)}\frac{\|\bsigmah^{\iter}-\bsigmah^{\iter-1}\|}{\tau}\\
&\;\;\;\;\;\;\;\;\;\;\;\;\frac{\|\bsigmah^{\iter}-\bsigmah^{\iter-1}\|}{\tau L\left(\left\|\nabla \FEneN(\bsigmah^{\iter})+\frac{\bsigmah^{\iter}-\bsigmah^{\iter-1}}{\tau}\right\|+\left\|\frac{\bsigmah^{\iter}-\bsigmah^{\iter-1}}{\tau}+\Hess \FEneN(\eta)(\bs\xi_t-\bsigmah^{\iter})\right\|\right)}\\
\leq & -\vartheta\frac{((\lambda_c-\epsilon)\tau+1)^2}{2(\lambda_c\tau+1)}\frac 1{L((\Lambda+\epsilon)\tau+1)}\frac{\|\bsigmah^{\iter}-\bsigmah^{\iter-1}\|}{\tau}\\
=& -\vartheta\frac{((\lambda_c-\epsilon)\tau+1)^2}{2(\lambda_c\tau+1)}\frac 1{L((\Lambda+\epsilon)\tau+1)}\|\bsigmah'(t)\|=:-A\|\bsigmah'(t)\|,
\end{align*}
where $\Lambda$ has been defined in \eqref{Lambdamax_def}. Therefore we can write
\begin{align*}
&H^\vartheta(T)-H^\vartheta(t)\\
=&\int_t^{\lceil t/\tau\rceil}\frac d{dt}H^\vartheta(s) ds +\sum_{\iter=\lceil t/\tau\rceil}^{\lfloor T/\tau\rfloor-1}\int_{\iter\tau}^{(\iter+1)\tau} \frac d{dt}H^\vartheta(s) ds+ \int_{\lfloor T/\tau\rfloor\tau}^{T} \frac d{dt}H^\vartheta(s) ds\\
\leq&-A\int_t^{\lceil t/\tau\rceil}\|\bsigmah'(s)\| ds -A\sum_{\iter=\lceil t/\tau\rceil}^{\lfloor T/\tau\rfloor-1}\int_{\iter\tau}^{(\iter+1)\tau} \|\bsigmah'(s)\| ds-A \int_{\lfloor T/\tau\rfloor\tau}^{T} \|\bsigmah'(s)\| ds\\
=& -A\int_t^T\|\bsigmah'(s)\|ds.
\end{align*}
Taking the limit as $T\to+\infty$ and using the fact that $\lim_{t\to +\infty}H(t)=0$, we obtain the lower bound:
\begin{equation}\label{eq:lower}
H^\vartheta(t)\geq A\int_t^{+\infty}\|\bsigmah'(s)\|ds.
\end{equation}
In particular, it follows that the function 
$$v(t):=\int_t^{+\infty}\|\bsigmah'(s)\|ds$$
is a well-defined continuous and almost everywhere differentiable function on $[\bar \iter\tau,+\infty[$ that vanishes at $\infty$ and with $v'(t)=\|\bsigmah'(t)\|$.

The upper bound is slightly easier to be obtained. Using the definition of $H$, the \Loja{} Inequality \eqref{eq:Loja}, and the stopping criterion \eqref{rdef}, for $t>\bar \iter \tau$, we have
\begin{align*}
H^\vartheta(t)=&(\FEneN(\bs\xi_t)-\FEneN(\bsigmah^*))^\vartheta\leq (L\|\nabla \FEneN(\bs\xi_t)\|)^{\frac \vartheta{1-\vartheta}}\\
\leq& \left(L\left\| \nabla \FEneN(\bsigmah^{\iter})+\frac{\bsigmah^{\iter}-\bsigmah^{\iter-1}}{\tau}-\frac{\bsigmah^{\iter}-\bsigmah^{\iter-1}}{\tau}+\Hess \FEneN(\eta)(\bs\xi_t-\bsigmah^{\iter-1})\right\|\right)^{\frac \vartheta{1-\vartheta}}\\
\leq& \Big(L\left\| \nabla \FEneN(\bsigmah^{\iter})+\frac{\bsigmah^{\iter}-\bsigmah^{\iter-1}}{\tau}\right\|+L\left\|\frac{\bsigmah^{\iter}-\bsigmah^{\iter-1}}{\tau}\right\|\\
&\;\;\;\;\;\;\;\;\;\;\;+L\left\|\Hess \FEneN(\eta)(\bs\xi_t-\bsigmah^{\iter-1})\right\|\Big)^{\frac \vartheta{1-\vartheta}}\\
\leq&\left(L((\Lambda+\epsilon)\tau +1)\frac{\|\bsigmah^{\iter}-\bsigmah^{\iter-1}\|}{\tau}\right)^{\frac \vartheta{1-\vartheta}}.
\end{align*}
Therefore, for any $t\in](\iter-1)\tau,\iter\tau[$, we have
\begin{equation}\label{eq:upper}
H^\vartheta(t)\leq (L(\Lambda+\epsilon)\tau +1)^{\frac \vartheta{1-\vartheta}} \|\bsigmah'(t)\|^{\frac \vartheta{1-\vartheta}}=:B^{\frac \vartheta{1-\vartheta}} \|\bsigmah'(t)\|^{\frac \vartheta{1-\vartheta}}. 
\end{equation}
The combination of \eqref{eq:lower}, \eqref{eq:upper}, and the definition of $v$ leads to
\begin{equation}\label{eq:vderivativeestimate}
  A^{\frac{1-\vartheta}\vartheta} v^{\frac{1-\vartheta}\vartheta}(t)\leq \left(H^\vartheta(t)\right)^{\frac{1-\vartheta}\vartheta}\leq -Bv'(t),\;\;
  \mbox{ for a.e. } t\in[\bar \iter\tau,+\infty[. 
\end{equation}
Now we need to distinguish the two cases $\vartheta<1/2$ and $\vartheta=1/2.$ In the latter case \eqref{eq:vderivativeestimate} reads as 
$$\frac{v'}{v}(t)\leq-\frac A B,\;\;\mbox{ for a.e. } t\in]\bar \iter\tau,+\infty[\,,$$
which by integration leads to~\eqref{exponentialspeedAlg}.
In the former case, \eqref{eq:vderivativeestimate} reads as 
$$\frac d{dt}v^{\frac{2\vartheta-1}{\vartheta}}(t)=\frac{2\vartheta-1}{\vartheta}v^{\frac{\vartheta-1}{\vartheta}}(t)v'(t)\geq \frac{1-2\vartheta}{\vartheta}\frac{A^{\frac{1-\vartheta}{\vartheta}}}{B},\;\;\mbox{ for a.e. } t\in]\bar .$$
Integration over $[\bar \iter\tau,\iter\tau]$ leads to
\begin{align*}
v(\iter\tau)\leq& \left(v^{\frac{2\vartheta-1}{\vartheta}}(\bar \iter\tau)+\frac{1-2\vartheta}{\vartheta}\frac{A^{\frac{1-\vartheta}{\vartheta}}}{B}\tau(\iter-\bar \iter)\right)^{\frac{-\vartheta}{1-2\vartheta}}\\
=:& \left(C+\frac{1-2\vartheta}{\vartheta}\frac{A^{\frac{1-\vartheta}{\vartheta}}}{B}\tau(\iter-\bar\iter)\right)^{\frac{-\vartheta}{1-2\vartheta}}.
\end{align*}
The proof is concluded by observing that
$$\|\bsigmah^{\iter}-\bsigmah^*\|=\left\|\int_{\iter\tau}^{+\infty}\bsigmah'(s) ds\right\|\leq \int_{\iter\tau}^{+\infty}\|\bsigmah'(s)\| ds=v(\iter\tau).$$
\end{proof}

\section{An improved algorithm with adaptive time-stepping}

In the present section we introduce a variant of Algorithm~\ref{alg:backEuler} that fully exploits the local convexity of the objective functional $\FEneN$ around its minimizers achieving an improved global convergence.
Indeed, under the conditions stated in Hypothesis~\ref{hyp:H2}, local convexity of $\FEneN$ and optimal \Loja\ constant $\vartheta=1/2$ are guaranteed. Thus, we can think of replacing the fixed time-step size by a growing and possibly diverging sequence, provided that the initial Newton guess given by the numerical solution at the previous time-step is sufficiently close to the new time-step solution. 

The development of this enhanced algorithm takes inspiration from a closer examination of the first part of the proof of Theorem~\ref{thm:concistency}, i.e., the convergence of the Newton method at each time-step and the variational characterization of the backward Euler scheme. A key step in the proof of this theorem is to assume that the time-step size $\tau$ satisfies the bound in \eqref{eq:tau-bound}, which derives from the uniformization of the sharper requirements \eqref{taubound1} and \eqref{taubound2} with respect to the time-step $\iter$. Note that this uniformization procedure is required only when working with fixed time-step size.
On the other hand, when we allow for a variable time-step size $\tau_\iter$, we can repeat the argument of the first part of the proof of Theorem~\ref{thm:concistency} directly using \eqref{taubound1} and \eqref{taubound2} instead of \eqref{eq:tau-bound}. The rest of the proof works also in this variable time-step size framework with only minor modifications. The same holds true for Proposition~\ref{prop:stability} and Theorem~\ref{thm:convergence}. 

It is very natural to wonder which sequences $\{\tau_\iter\}$ can be used to define our modified algorithm. Notice that, under Hypothesis~\ref{hyp:H2}, the sequences of the sets $\Omega_\iter$ and $U_\iter$ shrink to the point $\bsigmah^*$, a minimizer of $\FEneN$ that possesses a neighbourhood $U$ such that $\FEneN$ is strongly convex on $U$. Thus
\begin{align*}
  \lim_\iter\lambda_\iter=&\lim_{\iter}\min_{\bs\varsigma\in \co U_\iter} \lambda_{min}(\Hess \FEneN(\bs\varsigma))\\
  =&\lambda_{min}(\Hess \FEneN(\bsigmah^*))>0,\\
  \lim_\iter\gamma_\iter =& \lim_{\iter}\max_{\bs\varsigma\in \co U_\iter}\max_{\bsigma\in \Omega_\iter}			\|\nabla \GEneN(\bs\varsigma;\bsigma,\tau_\iter)\|\\
  \leq&\lim_\iter \max_{\bs\varsigma\in \co U_\iter}\max_{\bsigma\in \Omega_\iter} \|				\nabla \FEneN(\bsigmah^{\iter})\|+\left(\frac 1{\tau_\iter}+\Lambda_\iter\right)\|			\bs\varsigma-\bsigmah^{\iter}\|+\frac{\diam \Omega_\iter}{\tau_\iter}=0\\
  \lim_\iter R_\iter=&\lim_\iter\left(\max_{\bs\varsigma\in \co U_\iter}\rho(P_1(\bs\varsigma)),\dots,\max_{\bs\varsigma\in \co U_\iter}\rho(P_{\Ntdens}(\bs\varsigma))\right)\\
  =&\left(\rho(P_1(\bsigmah^*)),\dots,\rho(P_{\Ntdens}(\bsigmah^*))\right)\in \R^\Ntdens.
\end{align*}
Hence both the upper bounds for $\tau_\iter$ are diverging as $\iter\to +\infty$: 
\begin{equation}\label{heuristicsorigin}
\lim_\iter \min\left\{\frac 1{\lambda_\iter^-},\frac 1 {(\sqrt{\|R_\iter\|\gamma_\iter}-\lambda_\iter)^+} \right\}=+\infty.
\end{equation}
This provides a justification for the use of an increasing sequence $\{\tau_\iter\}$ of time-step sizes. However a  trade-off is clearly needed in picking $\{\tau_\iter\}$: on one hand a faster diverging sequence would lead to a faster converging algorithm, but on the other hand, in order to ensure the Newton method convergence,  $\tau_\iter$ needs to satisfy the upper bounds \eqref{taubound1} and \eqref{taubound2} at any finite stage $\iter$.  

Our strategy is to choose the classical geometric sequence whereby the new tentative time-step size for time level $\iter+1$ is set to
\begin{equation}\label{heuristics}
\tau_{\iter+1}=\alpha\tau_\iter\, ,
\end{equation}
with $\alpha>1$.
If the subsequent Newton method does not achieve the desired tolerance within the given maximum number of iterations, then the time-step is discarded and restarted with a smaller $\tau_{\iter+1}$.  This procedure is summarized in Algorithm~\ref{alg:backEuler_adaptive}. We note that, in our experiments, even if $\alpha$ is chosen considerably larger than $1$, typically $1.1\div 2$, only few Newton iterations are needed to achieve the desired tolerance and very few restarts occur.

We now show that, for a finite number of restarts, our method achieves super-exponential convergence towards the optimum. Indeed, repeating the argument of the proof of Proposition~\eqref{prop:rateofconv} taking into account~\eqref{heuristics}, we can write:
\begin{equation}\label{eq:fasterestimate}
  \frac{v'}{v}(t)\leq-\frac A B,\;\;\mbox{ for a.e. } t>\alpha^{\bar \iter}\tau,
\end{equation}
where $v$ is an upper bound of the error, i.e. $v(\alpha^{\iter}\tau)\geq \|\bsigmah^\iter-\bsigmah^*\|$ holds for $\iter>\bar\iter$. Then, integrating over  $[\alpha^{\bar \iter}\tau,\alpha^{\iter}\tau]$, we obtain
\begin{align*}
  v(\alpha^{\iter}\tau)\leq& \exp\left(-\frac{A\tau}B (\alpha^{\iter}-\alpha^{\bar\iter})\right)v(\alpha^{\bar \iter}\tau)= \exp\left(\frac{A\tau}{B}\alpha^{\bar\iter}\right) \exp\left(-\frac{A\tau}{B}\alpha^{\iter}\right)v(\alpha^{\bar \iter}\tau) \\
  =&C \exp\left(-\frac{A}{B}\alpha^{\iter}\right)v(\alpha^{\bar \iter}\tau)\,.
\end{align*}
Note that the sequence $\iter\mapsto \exp(-A/B\tau \alpha^\iter)$ converges to $0$ with order exactly equal to $\alpha$ and unit asymptotic constant.
\begin{algorithm}
\caption{Adaptive Backward Euler Scheme with Newton Solver}
\label{alg:backEuler_adaptive}
\begin{algorithmic}
\STATE{Input $\bsigmah^0\in \R^N$, $\tau>0$, $n_{step}\in \N$, $toll>0$, $\epsilon>0$, $r_{max}\in \N$, $\alpha>1$ }
\STATE{Set $\iter=0$, $r=0$, go=\TRUE}
\STATE{Compute $res=\|\nabla \FEneN(\bsigmah^0)\|$}
\IF{$res=0$}
\STATE{\textbf{Exit} with error.}
\ENDIF
\WHILE{$\iter<n_{step}$ \AND $res>toll$}
\STATE{Set $\iter=\iter+1$, $\bsigma^{new}:=\bsigma^{old}$}
\STATE{$\tau=\alpha \tau$}
\WHILE{ go=\TRUE}
\STATE{Compute $res_{Newton}:=\nabla \GEneN(\bsigma^{new};\bsigma^{old},\tau)$  }
\WHILE{$r<r_{max}$ \AND ($|\comp{res_{Newton}}{i}|>\epsilon |\comp{\bsigma^{new}-\bsigma^{old}}{i}|$ for some $i$ \OR $\sign \bsigma^{new}\neq \sign \bsigma^{old}$)}
\STATE{$r=r+1$}
\STATE{Compute $\bsigma^{new}=\bsigma^{new}-[\Hess \GEneN(\bsigma^{new};\bsigma^{old},\tau)]^{-1}\nabla \GEneN(\bsigma^{new};\bsigma^{old},\tau) $}
\STATE{Compute $res_{Newton}:=\nabla \GEneN(\bsigma^{new};\bsigma^{old},\tau)$}
\ENDWHILE
\IF{$|\comp{res_{Newton}}{i}|<\epsilon |\comp{\bsigma^{new}-\bsigma^{old}}{i}|$ for all $i$ \AND $\sign \bsigma^{new}= \sign \bsigma^{old}$}
\STATE{go = \FALSE}
\ELSE
\STATE{$\tau=\tau/\alpha$}
\ENDIF
\ENDWHILE
\STATE{Compute $res=\|\nabla \FEneN(\bsigma^{new})\|$}
\ENDWHILE
\RETURN $\bsigma^{new}$
\end{algorithmic}
\end{algorithm}
\section{Numerical experiments}
\label{sec:numerics}

In this section we test our theoretical results by looking at: i) verification of Hypothesis~\ref{hyp:H2}; ii) convergence of our discrete gradient flow (Algorithms~\ref{alg:backEuler} and~\ref{alg:backEuler_adaptive}), i.e., convergence towards the minimum of $\DEneN$; and iii) convergence of our numerical solution towards the real optimal transport density $\mu^*$, as given in Theorem~\ref{thm:mainresult1}. To this aim, in i) we use  as optimal point the converged numerical solution for a given tolerance as obtained by our algorithm; in ii) we look at quantities such as differences in the numerical solution at consecutive time steps and residuals of approximate KKT conditions; in iii) we make use of an explicit solution to the OT problem considered in~\cite{Buttazzo03}, which constitutes our test case  and is described in Section~\ref{sec:test-case}.

We recall that Hypothesis~\ref{hyp:H2} implies the uniqueness of $\muN^*$ (see Proposition~\ref{prop:wellcond}) and the exponential convergence of Algorithm~\ref{alg:backEuler} (see Proposition \ref{prop:rateofconv}), and guarantees the well-conditioning of the optimization problem (see Proposition~\ref{prop:refinedCond}). Note that the last proposition provides an a-posteriori error estimate in terms of the gradient of the functional $\FEneN$. Thus the following exit criterion for the gradient flow based on $\|\nabla \FEneN(\bsigma_h^\iter)\|$ can be used:
\begin{equation*}
  \|\bsigmah^\iter-\bsigmah^{\iter-1}\|
  \le C\|\nabla \FEneN(\bsigmah^{\iter})\|\le \tau_\iter\mbox{Toll}\, .
\end{equation*}

Verification of Hypothesis~\ref{hyp:H2} is done for different  $\delta_\IndexN$ sequences, namely $\delta_\IndexN\in\{\MeshPar_\IndexN,\MeshPar^2_\IndexN\}$.
Convergence of the gradient flow is verified by measuring the $\bsigma$-increment between time-steps 
\begin{equation}
\label{eq:sigma-increment}
    \Delta \bsigma^\iter := 
    \| \sigma_\IndexN^\iter- \sigma_\IndexN^{\iter-1}\|_{L^2(\Omega)}\,,
\end{equation}
which, as seen above, is a local estimator of the error. In addition, we control the approximate discrete KKT conditions in~\eqref{KKTdiscrete} by making sure that they satisfy:
\begin{equation}
\label{eq:kkt-approx}
  \mbox{KKT}_{\tolkkt}(\bmu) := \max_i 
  \left(\,\left\{
    \begin{aligned}
      &\max\left\{0,-\partial_i\DEneN(\bmu)\right\} 
      &\mbox{ if } \bmui < \tolkkt
      \\
      &\partial_i\DEneN(\bmu)) 
      &\mbox{ if } \bmui\geq\tolkkt
    \end{aligned}
  \right.\;\right),
\end{equation}
where $\tolkkt$ is the tolerance used to distinguish whether $\bmui>0$ or not.
Verification of the convergence of the gradient flow and the convergence towards $\mu^*$ is again performed for the different sequences of the relaxation parameter $\delta_\IndexN$ defined above. 

\subsection{The test case}\label{sec:test-case}

\begin{figure}
  \centerline{
    \includegraphics[width=0.45\textwidth]{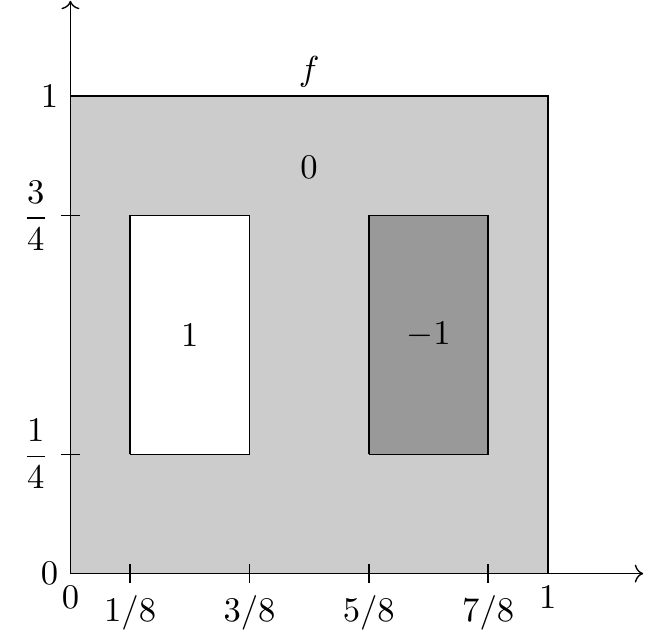}
    \includegraphics[width=0.45\textwidth]{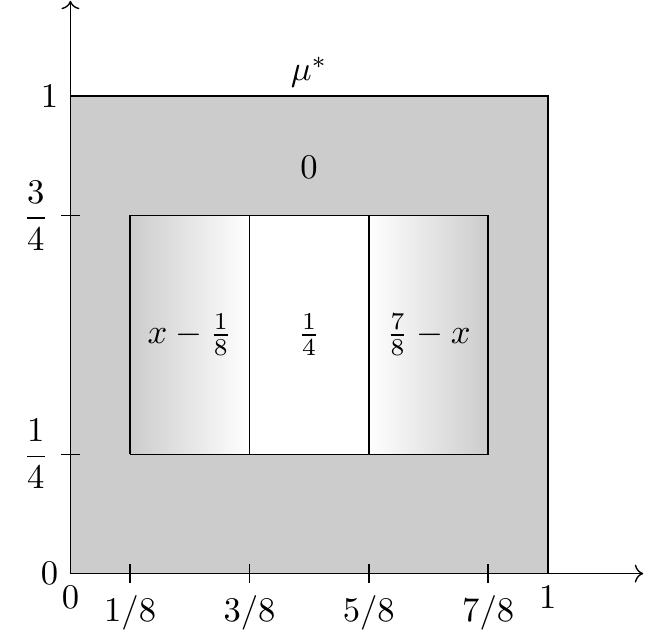}
  }
  \caption{The left panel shows the domain $\Omega$ and forcing function $f$ of the test case. The corresponding exact optimal density $\mu^*$ is represented in the right panel.}
  \label{fig:exact-mu}
\end{figure}

\begin{figure}
 \centerline{
   \includegraphics[width=0.5\textwidth]{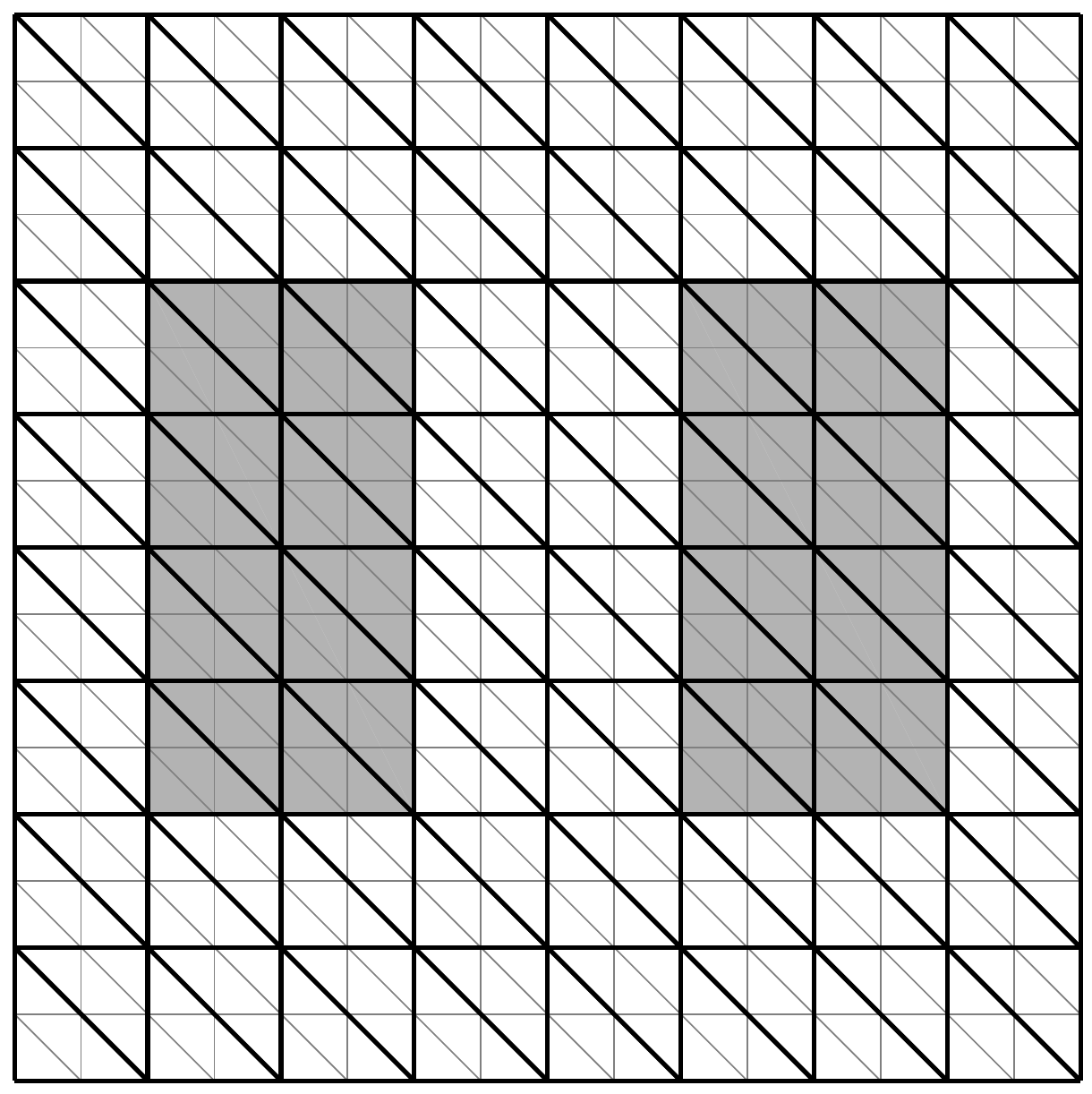}
 }
 \caption{Initial mesh $\Triang^0$ and its uniform refinement used in the definition of $\M_0^+$ and $\W_0$. The supports of $f^+$ and $f^-$, which are aligned with the mesh, are gray-shaded.}
 \label{fig:mesh-forcing}
\end{figure}

All our simulations address the test case proposed in \cite{FaDaCaPu20} that considers a rigid transport of a unit density mass of rectangular shape from left to right. We look for the numerical solution of~\eqref{eq:EvGaPDE} defined on the domain $\Omega=]0,1[^2$ and with forcing functions  $f^+=\chi([1/8,3/8]\times [1/4,3/4])$ and $f^-=\chi([5/8,7/8]\times [1/4,3/4])$ (see Figure~\ref{fig:exact-mu}, left). The explicit solution $(\mu^*,u^*)$ for this problem is calculated in~\cite{Buttazzo03} and is given by:
\begin{equation}\label{eq:buttazzo-sol-mu}
  \mu^*(x,y)=\left\{
    \begin{aligned}
      &(x-1/8) && (x,y)\in[1/8,3/8]\times[1/4,3/4]\\ 
      & 1/4 && (x,y)\in[3/8,5/8]\times[1/4,3/4]\\ 
      & (7/8-x) && (x,y)\in[5/8,7/8]\times[1/4,3/4]\\
      & 0 && \mbox{otherwise}
    \end{aligned}
             \right.
    \ ,
\end{equation}
and
\begin{equation}\label{eq:buttazzo-sol-pot}
  u^*(x,y) = x \qquad x\in\support(\mu^*)\ .
\end{equation}
The spatial distribution of $\mu^*$ is shown in Figure~\ref{fig:exact-mu}, right panel.
This (apparently) simple test case captures most of the main challenges in solving the Monge-Kantorovich equations. In fact, the optimal transport density $\mu^*$ is zero in a large part of $\Omega$, leading to the degeneracy of the elliptic PDE in~\eqref{eq:EvGaPDE}, and discontinuous across the segments that delimit the lower and upper boundary of its support, causing potential interpolation difficulties.

The space $\mathcal{M}_\IndexN^+$ is defined by means of a sequence of meshes $\Triang^\IndexN$ with $\IndexN=0,\ldots,4$. The coarsest mesh is a uniform triangulation of the unit square domain obtained by subdividing each side into $2^{3}$ intervals. Each subsequent mesh is obtained by uniform refinement of the previous one. Figure~\ref{fig:mesh-forcing} shows  the triangulations $\Triang^0$ and $\Triang^1$ used in the definition of $\mathcal{M}_0^+$ and $\mathcal{W}_0$. The support of the forcing functions $f^+$ and $f^-$, identified in the figure by gray areas, is aligned with the triangle edges to avoid any geometrical error. Automatically, the mesh results aligned with the entire support of $\mu^*$.

\subsection{Verification of Hypothesis~\ref{hyp:H2}}

\begin{figure}
\begin{center}
\begin{tabular}{cc}
\includegraphics[width=0.46 \textwidth]{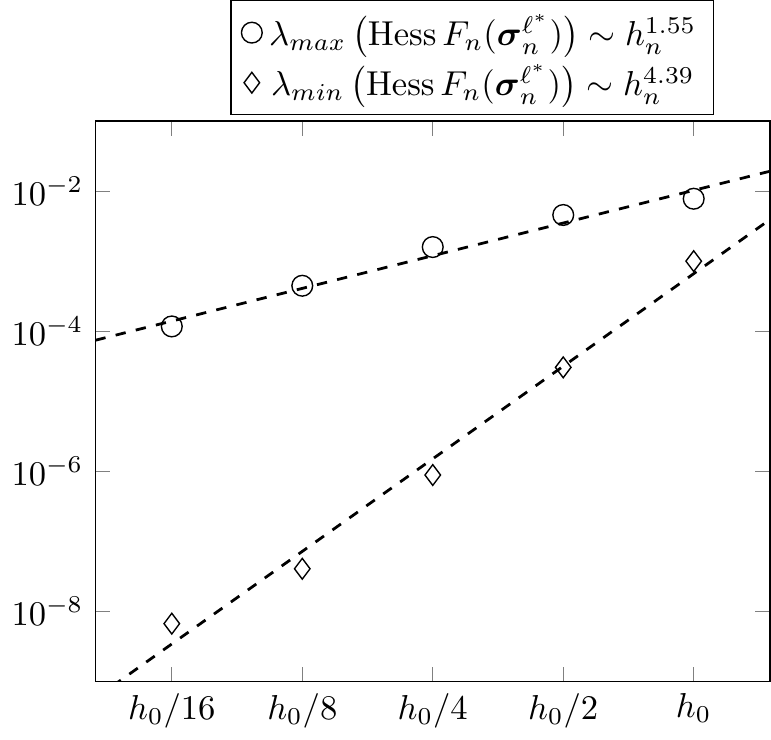}&
\includegraphics[width=0.46 \textwidth]{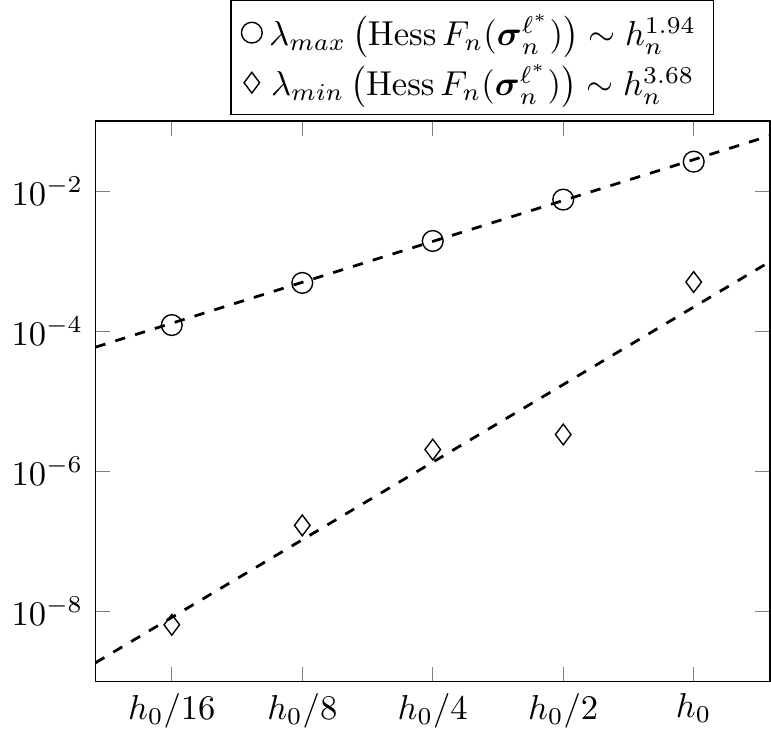}
\end{tabular}
\end{center}
\caption{Minimum and maximum eigenvalues of the Hessian of $F_\IndexN$ with respect to the mesh parameter $\MeshPar_\IndexN$ for $\delta_\IndexN=\MeshPar_\IndexN$ (left) and $\delta_\IndexN=\MeshPar_\IndexN^2$ (right).
}
\label{fig:eigs}
\end{figure}

The numerical validation of Hypothesis \ref{hyp:H2} presents two difficulties. First, the conditions in \eqref{eq:H21} \eqref{eq:H22} are expressed in terms of a minimizer $\bsigmah^*$ of $\FEneN$, which is not known a priori. Second, the numerical evaluation of eq. \eqref{eq:H22} is not robust, because it requires some sort of threshold procedure to check $\bmuhi^*\neq 0.$
In our experiments we consider the equivalent characterization of Hypothesis \ref{hyp:H2} given by Corollary \ref{cor:H2alternative} tested at the numerically computed optimum $\bsigmah^{\iter^*}$, i.e., we verify
\begin{equation*}
  \eigenval_{min}\left(\Hess\FEneN(\bsigmah^{\iter^*})\right)>0.
\end{equation*}
Figure~\ref{fig:eigs} shows the minimum ($\eigenval_{min}$) and maximum ($\eigenval_{max}$) for the mesh sequence $\Triang^0,\Triang^1,\dots,\Triang^4$ for $\delta_\IndexN\in\{\MeshPar_\IndexN,\MeshPar_\IndexN^2\}$, where $\MeshPar_\IndexN$ is denotes the mesh parameter of $\Triang^\IndexN$. We observe that the calculated minimum eigenvalues are always positive, suggesting that Hypothesis~\ref{hyp:H2} is always verified at the numerical optimum. We note that the minimum eigenvalue tends to zero with $\MeshPar_\IndexN$ at different rates depending on the chosen relaxation parameter $\delta_\IndexN$. Namely, we observe experimentally that $\eigenval_{\min}\sim \MeshPar_\IndexN^{4.4}$, in the case $\delta_\IndexN=\MeshPar_\IndexN$, and  $\eigenval_{min}\sim\MeshPar_\IndexN^{3.7}$, in the case $\delta_\IndexN=\MeshPar_\IndexN^2.$  
In addition, in both $\delta_\IndexN$ cases, the maximum and the minimum eigenvalues scale approximately at a constant rate, leading to an experimental spectral condition number of the Hessian that is proportional to $\MeshPar_\IndexN^3$ in the case $\delta_\IndexN=\MeshPar_\IndexN$ and to $\MeshPar_\IndexN^2$ in the case $\delta_\IndexN=\MeshPar_\IndexN^2$.

\subsection{Convergence of the gradient flow}

\begin{figure}
    \centering
    \includegraphics[width=0.8\textwidth]{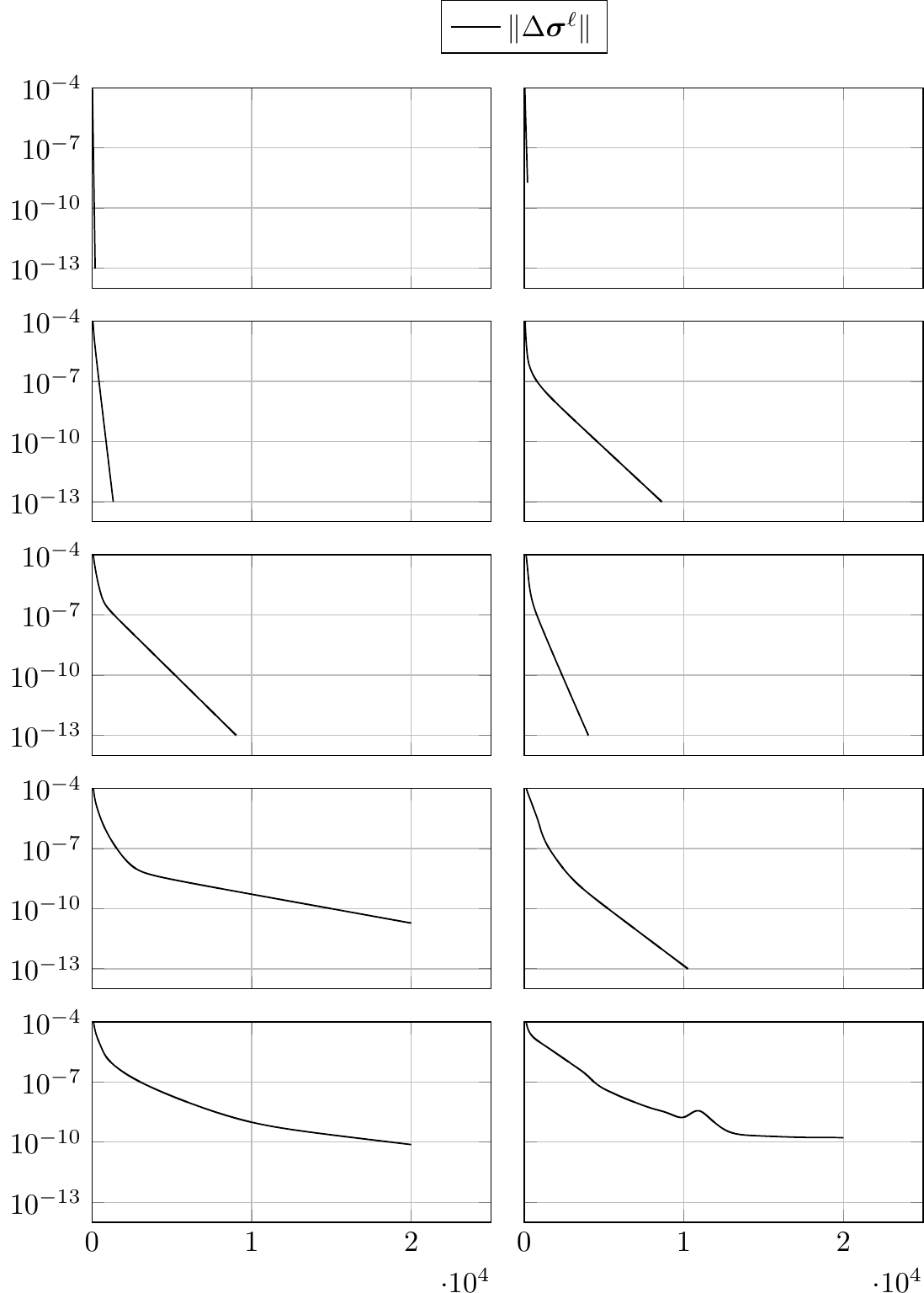}
    \caption{Evolution of $\|\Delta \sigma\|$, defined in ~\eqref{eq:sigma-increment}, with respect to the iterations $\iter$, computed by Algorithm~\ref{alg:backEuler} with $\tau=1.$ The results refer to the five meshes $\Triang^0,\Triang^1,\dots,\Triang^4$, from the coarsest (top panel) to the finest. Each mesh is the conformal refinement of the other. The left column reports the results obtained using $\Lift_\IndexN=\MeshPar_\IndexN$, while the right one $\Lift_\IndexN=\MeshPar^2_\IndexN$ .}
    \label{fig:kkt-increment-fixed}
\end{figure}

In Figure~\ref{fig:kkt-increment-fixed} we report the convergence profiles of Algorithm \ref{alg:backEuler} when applied to the solution of our test case with constant time-step size $\tau=1$ on the different mesh levels $n=0,1,\dots,4$ ($n=0$ top and $n=4$ bottom rows) and with $\Lift_\IndexN\in\{\MeshPar_\IndexN,\MeshPar_\IndexN^2\}$ (left and right columns). The $y$-axis represents the quantity $\|\Delta \bsigma^{\iter}\|$ defined in~\eqref{eq:sigma-increment}, while the $x$-axis represents time-steps ($\iter$).
The results shown in the first four rows correspond to mesh levels $n=0,1,2,3$ and clearly display the geometric convergence of the proposed method, as predicted by eq.~\eqref{exponentialspeedAlg} of Proposition \ref{prop:rateofconv}. Indeed, after a short pre-asymtpotic phase, the curves are essentially  straight lines. On the other hand, the last row, corresponding to $n=4$, shows a slightly decreased experimental convergence rate with respect to our theoretical results, in particular in the last few time-steps. This saturation can be attributed to the errors in the numerical solution of the linear systems, which, from experimental calculations, are characterized by large condition numbers of the order $\kappa\approx10^{6}$. Indeed, increments $\Delta\bsigma$ smaller than $\|\Delta\bsigma\|\approx10^{-10}$ would require a linear system solution with a residual norm smaller than  $\|\Delta\bsigma\|/\kappa\approx10^{-16}$, i.e., our machine precision.

\begin{figure}
    \centering
    \includegraphics[width=0.8\textwidth]{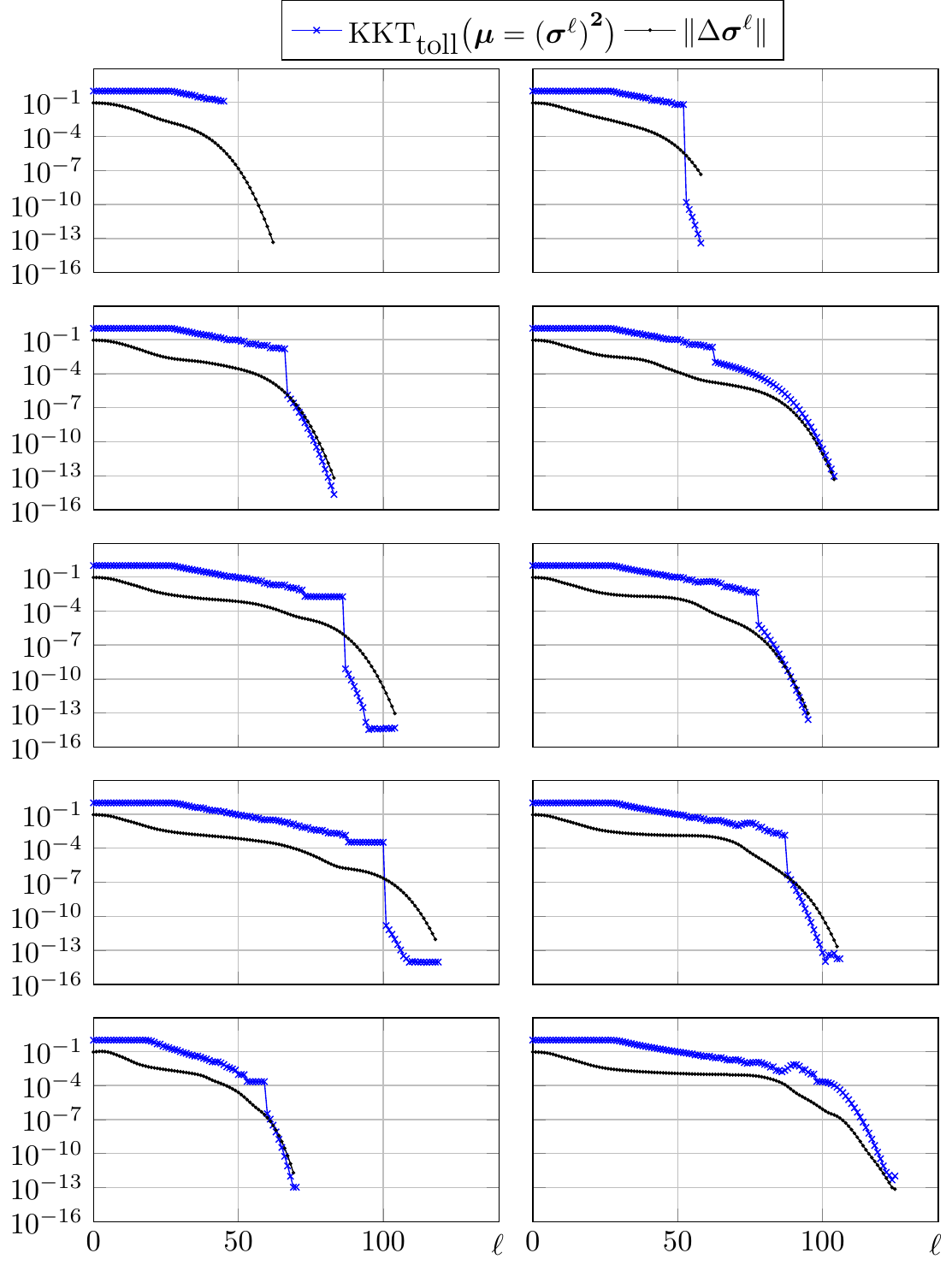}
    \caption{Evolution of $\|\Delta \sigma\|$  and $\mbox{KKT}_{\tolkkt}$ (defined in ~\eqref{eq:sigma-increment} and~\eqref{eq:kkt-approx}), with respect to the iterations $\iter$, computed by Algorithm~\ref{alg:backEuler_adaptive} with $\tau=1$ and (in equation \eqref{heuristics}) $\alpha=1.2$
    The results refer to fives meshes, from the coarsest (top panel) to the finest. Each mesh is the conformal refinement of the other. The left column reports the result using $\Lift=h$, while the right one $\Lift=h^2$ (note that in the top left panel, the latest value of $\mbox{KKT}_{\tolkkt}$ are dropped because they are equal to zero).
    }
    \label{fig:kkt-increment}
\end{figure}

We repeat the same experiment using Algorithm \ref{alg:backEuler_adaptive} instead of Algorithm~\ref{alg:backEuler} and report the obtained results in Figure~\ref{fig:kkt-increment}. Here the quantity $\|\Delta \sigma\|$ is reported together with the approximate KKT condition residual defined in equation~\eqref{eq:kkt-approx}, which is used to experimentally verify that the computed critical point is indeed a minimum for the functional $\FEneN$. The latter quantity naturally exhibits some drops due to the boolean nature of the test $\bmui >toll$. In all the considered cases the super-exponential convergence predicted by eq.\eqref{eq:fasterestimate} and lines below is observed. This has a remarkable effect on the number of iterations needed to achieve a predefined precision, which, compared to the constant step case, drops by a multiplicative factor of approximately 100. Moreover, we observe that also the rate at which iterations grow as $n$ increases is slower when using Algorithm~\ref{alg:backEuler_adaptive} instead of Algorithm~\ref{alg:backEuler}. Finally, we note that the above mentioned saturation effect is weaker in these simulations with respect to the ones obtained by Algorithm~\ref{alg:backEuler}.
\subsection{Convergence towards $\mu^*$}
\begin{figure}
    \centerline{
    \includegraphics[width=0.48\textwidth]{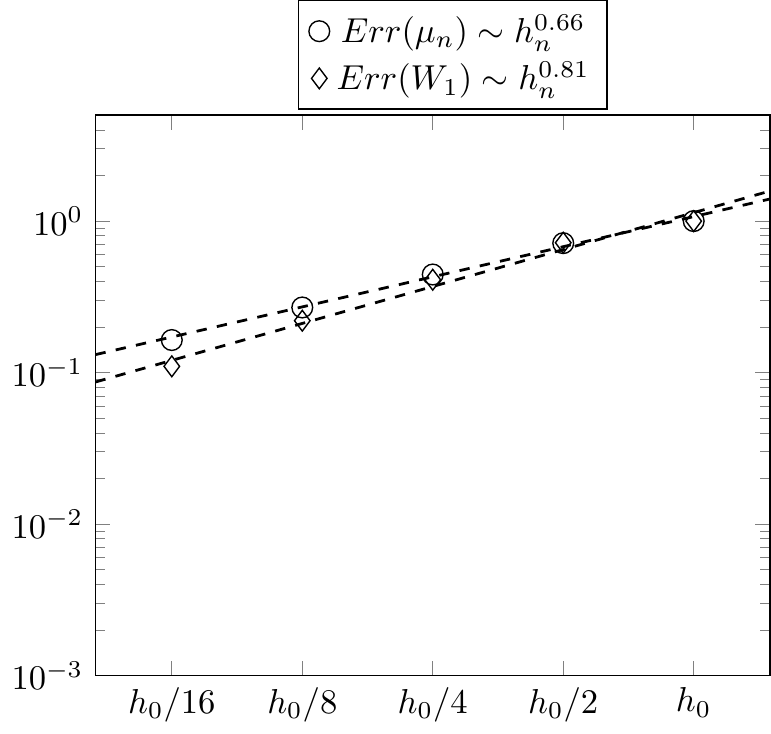}
    \includegraphics[width=0.48\textwidth]{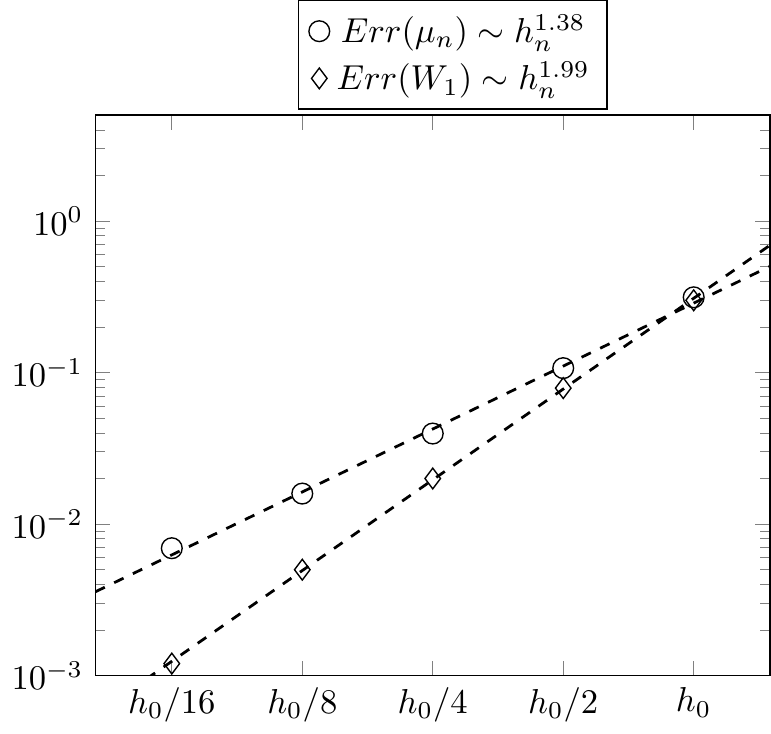}
    }
    \caption{
    Error with respect to the exact solution $\mu^*$ and the exact Wasserstein-1 distance for our approximate solution $\mu_\IndexN$  while halving to the mesh parameter $\MeshPar_\IndexN$. In the legend we report their approximate power law scaling with respect $\MeshPar_\IndexN$ (dashed lines). The left panel reports their behaviour for $\Lift_\IndexN=\MeshPar_\IndexN$, while the right panel for $\Lift_\IndexN=\MeshPar_\IndexN^2$.
    }
    \label{fig:errors}
\end{figure}

The last experiment is aimed at testing numerically the spatial convergence of the proposed method using the different mesh levels and the two choices of $\delta_\IndexN$ defined before. Figure~\ref{fig:errors} reports the related results in terms of the $L^2$-norms of the errors on the minimizer, $\|\mu^*_\IndexN-\mu^*\|$ vs. $\MeshPar_\IndexN$, and the errors on the optimal transport energy, $|\EneN(\mu^*_\IndexN)-\Ene(\mu^*)|$ vs. $\MeshPar_\IndexN$.
We would like to stress that the optimal value of the transport energy is precisely the Wasserstein-1 distance between $f^+$ and $f^-$~\cite{FaDaCaPu20}, i.e.:
\begin{equation*}
  W_1(f^+,f^-)=\Ene(\mu^*)\, . 
\end{equation*}

The graph on the left is related to $\delta_\IndexN=\MeshPar_\IndexN$ and shows that the convergence rate for both plotted quantities is smaller than 1, with a slightly greater convergence rate for the Wasserstein-1 (0.81 vs. 0.66). These results are a consequence of error saturation arising from the first order convergence of $\delta_\IndexN$ towards zero.
On the other hand, the convergence profiles for the case $\delta_\IndexN=\MeshPar_\IndexN^2$ of the right panel display superlinear convergence rates. The transport density converges towards the optimal value with a calculated rate of 1.38, while the Wasserstein-1 distance displays full second order convergence in accordance with the rate of convergence to zero of $\delta_\IndexN$.
We would like to note that similar results were found in~\cite{BeFaPu22}, where the algorithm proposed in~\cite{FaDaCaPu20} was extended to address OT problems defined on three-dimensional embedded surfaces. The algorithm used in~\cite{FaDaCaPu20} differs from Algorithm~\ref{alg:backEuler_adaptive} in the fact that forward Euler is used in place of backward Euler and $\delta_\IndexN=0$, leading to obvious restrictions on time-step sizes and considerably higher computational costs. The obtained experimental results showed a third order convergence rate attributable to the absence of error saturation effects caused by the presence of $\delta_\IndexN$.


\section*{Acknowledgments}

\bibliographystyle{siam}      
\bibliography{biblio}
\end{document}